\RequirePackage{snapshot}
\documentclass[english,a4paper,fleqn]{cas-sc}
\usepackage[T1]{fontenc}
\usepackage[utf8]{inputenc}
\usepackage{color}
\usepackage{units}
\usepackage{amsmath}
\usepackage{amsthm}
\usepackage{amssymb}
\usepackage{stmaryrd}
\usepackage{graphicx}
\usepackage[numbers]{natbib}
\usepackage{subscript}

\makeatletter

\providecommand{\tabularnewline}{\\}

\def\RCSfile{cas-sc}%
 \def\RCSversion{2.3}%
 \def\RCSdate{2021/05/11}%
\newif\iflongmktitle    \longmktitlefalse
\newif\ifdc             \global\dcfalse
\newif\ifsc             \global\sctrue
\newif\ifcasreviewlayout  \global\casreviewlayoutfalse
\newif\ifcasfinallayout   \global\casfinallayoutfalse

\def\blstr#1{\gdef\@blstr{#1}}
\def\@blstr{1}
\newdimen\@bls
\@bls=\baselineskip

\RequirePackage{graphicx}
\RequirePackage{amsmath,amsfonts,amssymb}
\allowdisplaybreaks

\RequirePackage{expl3,xparse}
\@ifundefined{regex_match:nnTF}{\RequirePackage{l3regex}}{}
\RequirePackage{etoolbox}
\RequirePackage{booktabs,makecell,multirow,array,colortbl,dcolumn,stfloats}
\RequirePackage{xspace,xstring,footmisc}
\RequirePackage[svgnames,dvipsnames]{xcolor}

 \let\comma\@empty
\let\tnotesep\@empty
\RequirePackage{cas-common}

\ExplSyntaxOn

\RenewDocumentCommand \maketitle {}
{
\ifbool { usecasgrabsbox }
  {
    \setcounter{page}{0}
    \thispagestyle{empty}
    \unvbox\casgrabsbox
  } { }
\pagebreak
\ifbool { usecashlsbox }
  {
    \setcounter{page}{0}
    \thispagestyle{empty}
    \unvbox\casauhlbox
  } { }
\pagebreak
\thispagestyle{first}
\ifbool{longmktitle}
{
  \ifnum\theblind>0\relax
      \LongMaketitleBox[blind]
  \else
		\LongMaketitleBox
  \fi
  \ProcessLongTitleBox
}
{
  \ifnum\theblind>0\relax
      \MaketitleBox[blind]
  \else
		\MaketitleBox
  \fi
  \printFirstPageNotes
}
\normalcolor \normalfont
\setcounter{footnote}{\int_use:N \g_stm_fnote_int}
\renewcommand\thefootnote{\arabic{footnote}}
\gdef\@pdfauthor{\infoauthors}
\gdef\@pdfsubject{Complex ~STM ~Content}
\ifbool{casreviewlayout}{\doublespacing}{}
}

\RequirePackage[T1]{fontenc}

\file_if_exist:nTF { stix.sty }
{
\file_if_exist:nTF { charis.sty }
{
  \RequirePackage[notext]{stix}
  \RequirePackage{charis}
}
{ \RequirePackage{stix} }
}
{
\iow_term:x {  *********************************************************** }
\iow_term:x { ~Stix ~ and ~ Charis~ fonts ~ are ~ not ~ available ~ }
\iow_term:x { ~ in ~TeX~system.~Hence~CMR~ fonts~ are ~ used. }
\iow_term:x {  *********************************************************** }
}

\file_if_exist:nTF { inconsolata }
{ \RequirePackage[scaled=.85]{inconsolata} }
{ \tex_gdef:D \ttdefault { cmtt } }

\ifbool{casreviewlayout}{\RequirePackage{setspace}}{}

\ExplSyntaxOff

\usepackage[%
paperwidth=192mm,
paperheight=262mm,
headsep=12pt,
footskip=12pt,
]{geometry}
\RequirePackage[colorlinks]{hyperref}
\colorlet{scolor}{black}
\colorlet{hscolor}{DarkSlateGrey}
\hypersetup{%
pdfcreator={LaTeX3; cas-sc.cls; hyperref.sty},
pdfproducer={pdfTeX;},
linkcolor={hscolor},
urlcolor={hscolor},
citecolor={hscolor},
filecolor={hscolor},
menucolor={hscolor},
}

\numberwithin{equation}{section}
\numberwithin{figure}{section}
\theoremstyle{plain}
\newtheorem{thm}{\protect\theoremname}
\theoremstyle{remark}
\newtheorem{rem}[thm]{\protect\remarkname}
\theoremstyle{plain}
\newtheorem{lem}[thm]{\protect\lemmaname}
\newcommand\RevisionTextThree[1]{\textcolor{teal}{#1}}
\newcommand\RevisionTextOne[1]{\textcolor{red}{#1}}

\usepackage{algpseudocode,algorithm,algorithmicx}
\usepackage{tikz}
\usetikzlibrary{shapes,arrows,calc,decorations.pathreplacing}
\usepackage{lineno}

\@ifundefined{showcaptionsetup}{}{%
 \PassOptionsToPackage{caption=false}{subfig}}
\usepackage{subfig}
\makeatother

\usepackage{babel}
\providecommand{\lemmaname}{Lemma}
\providecommand{\remarkname}{Remark}
\providecommand{\theoremname}{Theorem}

\begin{document}
\let\WriteBookmarks\relax 
\def\floatpagepagefraction{1} 
\def\textpagefraction{.001}

\shorttitle{}
\shortauthors{E.J. Ching et al.}

\title[mode = title]{Positivity-preserving and entropy-bounded discontinuous Galerkin method for the chemically reacting, compressible Euler equations. Part II: The multidimensional case}

\author[1]{Eric J. Ching}
\cormark[1] 
\ead{eric.ching@nrl.navy.mil} 
\affiliation[1]{organization={Laboratories of Computational Physics and Fluid Dynamics},             
	     addressline={U.S. Naval Research Laboratory},              
	     city={Washington},              
		 state={DC},   
	     postcode={20375}}

\author[1]{Ryan F. Johnson}
\author[1]{Andrew D. Kercher}
\cortext[1]{Corresponding author}

\begin{abstract}
In this second part of our two-part paper, we extend to multiple spatial
dimensions the one-dimensional, fully conservative, positivity-preserving,
and entropy-bounded discontinuous Galerkin scheme developed in the
first part for the chemically reacting Euler equations. Our primary
objective is to enable robust and accurate solutions to complex reacting-flow
problems using the high-order discontinuous Galerkin method without
requiring extremely high resolution. Variable thermodynamics and detailed
chemistry are considered. Our multidimensional framework can be regarded
as a further generalization of similar positivity-preserving and/or
entropy-bounded discontinuous Galerkin schemes in the literature.
In particular, the proposed formulation is compatible with curved
elements of arbitrary shape, a variety of numerical flux functions,
general quadrature rules with positive weights, and mixtures of thermally
perfect gases. Preservation of pressure equilibrium between adjacent
elements, especially crucial in simulations of multicomponent flows,
is discussed. Complex detonation waves in two and three dimensions
are accurately computed using high-order polynomials. Enforcement
of an entropy bound, as opposed to solely the positivity property,
is found to significantly improve stability. Mass, total energy, and
atomic elements are shown to be discretely conserved.
\end{abstract}

\maketitle
\global\long\def\middlebar{\,\middle|\,}%
\global\long\def\average#1{\left\{  \!\!\left\{  #1\right\}  \!\!\right\}  }%
\global\long\def\expnumber#1#2{{#1}\mathrm{e}{#2}}%
 \newcommand*{\horzbar}{\rule[.5ex]{2.5ex}{0.5pt}}

\global\long\def\revisionmathone#1{\textcolor{red}{#1}}%

\global\long\def\revisionmathtwo#1{\textcolor{blue}{#1}}%

\global\long\def\revisionmaththree#1{\textcolor{teal}{#1}}%

\makeatletter \def\ps@pprintTitle{  \let\@oddhead\@empty  \let\@evenhead\@empty  \def\@oddfoot{\centerline{\thepage}}  \let\@evenfoot\@oddfoot} \makeatother

\let\svthefootnote\thefootnote\let\thefootnote\relax\footnotetext{\\ \hspace*{90pt}DISTRIBUTION STATEMENT A. Approved for public release. Distribution is unlimited.}\addtocounter{footnote}{-1}\let\thefootnote\svthefootnote

\section{Introduction\label{sec:Introduction}}

This paper is the second part of a series of two that introduces a
fully conservative, positivity-preserving, and entropy-bounded discontinuous
Galerkin (DG) method for simulating the multicomponent, chemically
reacting, compressible Euler equations. In Part I~\citep{Chi22},
we addressed the one-dimensional case; here, we focus on the multidimensional
case. The starting point of our methodology is the fully conservative,
high-order scheme previously developed by Johnson and Kercher~\citep{Joh20_2}
that can maintain pressure equilibrium (in an approximate sense) in
smooth regions of the flow or across material interfaces when the
temperature is continuous. The generation of such oscillations is
a major, well-known issue that inhibits fully conservative numerical
schemes~\citep{Abg88,Kar94,Abg96}. Nonconservative methods are a
common alternative to circumvent this drawback. The fully conservative
formulation in~\citep{Joh20_2} instead maintains both pressure equilibrium
and discrete conservation of mass and total energy through consistent
evaluation of both the complex thermodynamics and the resulting semidiscrete
form, as well as the proper choice of nodal basis. A series of complex
multicomponent reacting flows was computed, including a three-dimensional
reacting shear flow, which did not require additional stabilization
to calculate. Also simulated was a two-dimensional moving detonation
wave. With artificial viscosity to stabilize the flow-field discontinuities,
the correct cellular structure was predicted. However, a linear polynomial
approximation and a very fine mesh were required to maintain robustness,
illustrating the challenge of using high-order methods to achieve
stable and accurate solutions to multidimensional detonation-wave
problems on relatively coarse meshes. In light of this difficulty,
we aim to develop a positivity-preserving and entropy-bounded DG method
that can robustly and efficiently simulate complex reacting flows
in multiple dimensions using high-order polynomial approximations. 

\subsection{Summary of Part I}

In Part I~\citep{Chi22}, we introduced the groundwork to construct
such a formulation. First, we established a minimum entropy principle
satisfied by entropy solutions to the multicomponent Euler equations
with chemical source terms, extending to the reacting case the proof
by Gouasmi et al.~\citep{Gou20} of a minimum entropy principle in
the nonreacting case. This principle, which states that the spatial
minimum of the specific thermodynamic entropy increases with time,
is a critical component of the theoretical basis for the formulation.
 In the remainder of Part I~\citep{Chi22}, we introduced the mathematical
framework in one spatial dimension that ensures the discrete solution
satisfies the following: nonnegative species concentrations, positive
density, positive pressure, and specific thermodynamic entropy bounded
from below. We discussed how to maintain compatibility with the strategies
devised in~\citep{Joh20_2} to maintain pressure equilibrium between
adjacent elements. In addition, a two-point numerical state function
was derived, as part of an entropy-stable DG scheme based on diagonal-norm
summation-by-parts operators for treating stiff chemical reactions
in the reaction step of an operator splitting procedure. Applying
the coupled solver to canonical one-dimensional test cases, we demonstrated
that the formulation can achieve robust and accurate solutions on
relatively coarse meshes. Optimal high-order convergence in a smooth
flow, namely thermal-bubble advection, was obtained. Furthermore,
we found that the enforcement of only the positivity property (i.e.,
nonnegative concentrations and positive density and pressure) can
fail to suppress large-scale nonlinear instabilities, but the added
enforcement of an entropy bound significantly improves robustness,
to a greater degree than in the monocomponent, calorically perfect
case.

\subsection{Contributions of Part II}

In this paper, we extend the solver to two and three spatial dimensions.
Our main contributions are as follows:
\begin{itemize}
\item The mathematical framework in Part I~\citep{Chi22} is enhanced to
account for curved, multidimensional elements of arbitrary shape.
Our multidimensional extension further generalizes current multidimensional
positivity-preserving/entropy-bounded DG schemes in the literature~\citep{Zha10,Zha12,Lv15_2,Jia18}
by relaxing restrictions on the volume and surface quadrature rules,
physical modeling, polynomial order of the geometric approximation,
and/or numerical flux function. We also discuss how to deal with curved
elements in the reaction step. 
\item We propose strategies to maintain compatibility with the procedures
introduced in~\citep{Joh20_2} to preserve pressure equilibrium,
which is less straightforward than in one dimension. We compute the
advection of a thermal bubble in two dimensions to determine whether
pressure equilibrium can be maintained when using curved elements
(not considered in~\citep{Joh20_2}).
\item We compute complex, large-scale detonation-wave problems in two and
three dimensions with detailed chemistry. Enforcement of only the
positivity property often fails to provide adequate stabilization
(even with artificial viscosity), while enforcing an entropy bound
enables robust calculations on coarse meshes. 
\end{itemize}

\subsection{Additional background}

Positivity-preserving DG schemes have emerged in recent years as a
popular numerical technique to simulate fluid flows in a robust manner.
\citep[Section 1]{Chi22} briefly reviews the history of the development
of positivity-preserving and related entropy-bounded DG methods pioneered
by Zhang and Shu~\citep{Zha10_2,Zha10} for the monocomponent, nonreacting
Euler equations; here, we focus on the multidimensional aspect. The
key idea in constructing such positivity-preserving/entropy-bounded
schemes is to expand the element average of the solution as a convex
combination of first-order three-point systems and pointwise values.
Zhang and Shu~\citep{Zha10} devised one such expansion on rectangular
meshes based on tensor products of Gauss-Legendre and Gauss-Lobatto
rules, which was later extended to straight-sided triangular elements~\citep{Zha12}
by utilizing triangle-rectangle transformations~\citep{Dub91,Kir06}.
Lv and Ihme~\citep{Lv15_2} introduced an expansion compatible with
curved elements of arbitrary shape. The only restriction on the volume
and surface quadrature rules is that the weights be positive; the
surface quadrature points need not be part of the set of volume quadrature
points, and the surface quadrature rules can be different among the
faces of a given element, which is crucial for prismatic elements
and $p$-adaptive calculations. However, the expansion in~\citep{Lv15_2}
relies on the Lax-Friedrichs numerical flux, and the resulting time-step-size
constraint assumes a calorically perfect gas. Jiang and Liu~\citep{Jia18}
proposed an expansion that is compatible with certain polygonal elements
and does not rely on the Lax-Friedrichs numerical flux, but the following
assumptions are made: (a) for a given face, the geometric Jacobian
and surface normal are constant and (b) the same surface quadrature
rule is employed for each face. Said expansion can be viewed as a
generalization of that by Zhang et al.~\citep{Zha12} for straight-sided
triangular elements. In this paper, we introduce an expansion that
is compatible with:
\begin{itemize}
\item curved elements of arbitrary shape
\item any invariant-region-preserving numerical flux
\item any combination of volume and surface quadrature rules (which can
differ among faces) with positive weights
\end{itemize}
Furthermore, no physics-based assumptions are placed on the resulting
time-step-size constraint. 

\subsection{Outline}

The remainder of this paper is organized as follows. The governing
equations and basic DG discretization are reviewed in Sections~\ref{sec:governing_equations}
and~\ref{sec:DG-discretization}, respectively. The following section
presents the positivity-preserving and entropy-bounded multidimensional
DG formulation for the transport step. Results for two-dimensional
thermal-bubble advection and complex two- and three-dimensional moving
detonation-wave simulations are given in Section~\ref{sec:results-2D}.
We close the paper with concluding remarks. 

It is recommended that Part I~\citep{Chi22} be read first since
the formulation developed here relies on many of the key ideas introduced
in detail in the first part. For conciseness, important concepts already
discussed in Part I are only briefly summarized in this paper.

\section{Governing equations\label{sec:governing_equations}}

The governing equations are the compressible, multicomponent, chemically
reacting Euler equations, written as
\begin{equation}
\frac{\partial y}{\partial t}+\nabla\cdot\mathcal{F}\left(y\right)-\mathcal{S}\left(y\right)=0\label{eq:conservation-law-strong-form}
\end{equation}
where $t\in\mathbb{R}^{+}$ is time and $y(x,t):\mathbb{R}^{d}\times\mathbb{R}^{+}\rightarrow\mathbb{R}^{\revisionmathone m}$
is the conservative state vector, given by

\begin{equation}
y=\left(\rho v_{1},\ldots,\rho v_{d},\rho e_{t},C_{1},\ldots,C_{n_{s}}\right)^{T}.\label{eq:reacting-navier-stokes-state}
\end{equation}
$x=(x_{1},\ldots,x_{d})$ denotes the physical coordinates, with $d$
indicating the number of spatial dimensions, $v=\left(v_{1},\ldots,v_{d}\right)$
is the velocity vector, $e_{t}$ is the specific total energy, and
$C_{i}$ is the concentration of the $i$th species. $\rho$ is the
density, computed as
\[
\rho=\sum_{i=1}^{n_{s}}\rho_{i}=\sum_{i=1}^{n_{s}}W_{i}C_{i},
\]

\noindent where $\rho_{i}$ is the partial density and $W_{i}$ is
the molecular weight of the $i$th species. $n_{s}$ is the total
number of species, and $m=d+n_{s}+1$ is the total number of state
variables. $Y_{i}=\rho_{i}/\rho$ and $X_{i}=C_{i}/\sum_{i=1}^{n_{s}}C_{i}$
are the mass fraction and mole fraction, respectively, of the $i$th
species. $\mathcal{F}(y):\mathbb{R}^{m}\rightarrow\mathbb{R}^{m\times d}$
in Equation~(\ref{eq:conservation-law-strong-form}) is the convective
flux, the $k$th spatial component of which is defined as
\begin{equation}
\mathcal{F}_{k}^{c}\left(y\right)=\left(\rho v_{k}v_{1}+P\delta_{k1},\ldots,\rho v_{k}v_{d}+P\delta_{kd},v_{k}\left(\rho e_{t}+P\right),v_{k}C_{1},\ldots,v_{k}C_{n_{s}}\right)^{T},\label{eq:reacting-navier-stokes-spatial-convective-flux-component}
\end{equation}
where $P$ is the pressure, computed as
\begin{equation}
P=R^{0}T\sum_{i=1}^{n_{s}}C_{i},\label{eq:EOS}
\end{equation}
with $T$ denoting the temperature and $R^{0}$ the universal gas
constant. The total energy is given by

\[
e_{t}=u+\frac{1}{2}\sum_{k=1}^{d}v_{k}v_{k},
\]
where $u$ is the mixture-averaged specific internal energy, calculated
as
\[
u=\sum_{i=1}^{n_{s}}Y_{i}u_{i},
\]
with $u_{i}$ indicating the specific internal energy of the $i$th
species. In this work, we employ the thermally perfect gas model,
with $u_{i}$ approximated by a polynomial function of temperature
as
\begin{equation}
u_{i}=\sum_{k=0}^{n_{p}+1}b_{ik}T^{k}.\label{eq:internal-energy-polynomial}
\end{equation}
The specific heats at constant volume and constant pressure, $c_{v}$
and $c_{p}$, the specific enthalpy, $h_{i}$, and the specific thermodynamic
entropy, $s_{i}$, of the $i$th species can be calculated by appropriately
differentiating/integrating Equation~(\ref{eq:internal-energy-polynomial})
and incorporating the integration constants from the NASA curve fits~\citep{Mcb93,Mcb02}.
The mixture-averaged thermodynamic entropy, $s$, which will be important
in subsequent sections, is given by
\[
s=\sum_{i=1}^{n_{s}}Y_{i}s_{i}.
\]
Finally, $\mathcal{S}(y):\mathbb{R}^{m}\rightarrow\mathbb{R}^{m}$
in Equation~(\ref{eq:conservation-law-strong-form}) is the chemical
source term, given by
\begin{equation}
\mathcal{S}\left(y\right)=\left(0,\ldots,0,0,\omega_{1},\ldots,\omega_{n_{s}}\right)^{T},\label{eq:reacting-navier-stokes-source-term}
\end{equation}
where $\omega_{i}$ is the production rate of the $i$th species.
Additional details on the thermodynamic relationships and chemical
production rates can be found in Part I~\citep{Chi22}.

\section{Discontinuous Galerkin discretization\label{sec:DG-discretization}}

Let $\Omega\subset\mathbb{R}^{d}$ be the computational domain, which
is partitioned by $\mathcal{T}$, comprised of non-overlapping cells
$\kappa$ with boundaries $\partial\kappa$. Let $\mathcal{E}$ denote
the set of interfaces $\epsilon$, with $\cup_{\epsilon\in\mathcal{E}}\epsilon=\cup_{\kappa\in\mathcal{T}}\partial\kappa$.
At interior interfaces, there exists $\kappa^{+},\kappa^{-}\in\mathcal{T}$
such that $\epsilon_{\mathcal{}}=\partial\kappa^{+}\cap\partial\kappa^{-}$.
$n^{+}$ and $n^{-}$ denote the outward facing normal of $\kappa^{+}$
and $\kappa^{-}$, respectively, with $n^{+}=-n^{-}$. The discrete
subspace $V_{h}^{p}$ over $\mathcal{T}$ is defined as
\begin{eqnarray}
V_{h}^{p} & = & \left\{ \mathfrak{v}\in\left[L^{2}\left(\Omega\right)\right]^{m}\middlebar\forall\kappa\in\mathcal{T},\left.\mathfrak{v}\right|_{\kappa}\in\left[\mathcal{P}_{p}(\kappa)\right]^{m}\right\} ,\label{eq:discrete-subspace}
\end{eqnarray}
where, for $d=1$, $\mathcal{P}_{p}(\kappa)$ is the space of polynomial
functions of degree less than or equal to $p$ in $\kappa$.  For
$d>1$, the choice of polynomial space typically depends on the type
of element~\citep{Har13}.

The semi-discrete problem statement is as follows: find $y\in V_{h}^{p}$
such that
\begin{gather}
\sum_{\kappa\in\mathcal{T}}\left(\frac{\partial y}{\partial t},\mathfrak{v}\right)_{\kappa}-\sum_{\kappa\in\mathcal{T}}\left(\mathcal{F}\left(y\right),\nabla\mathfrak{v}\right)_{\kappa}+\sum_{\epsilon\in\mathcal{E}}\left(\mathcal{F}^{\dagger}\left(y^{+},y^{-},n\right),\left\llbracket \mathfrak{v}\right\rrbracket \right)_{\mathcal{E}}-\sum_{\kappa\in\mathcal{T}}\left(\mathcal{S}\left(y\right),\mathfrak{v}\right)_{\kappa}=0\qquad\forall\:\mathfrak{v}\in V_{h}^{p},\label{eq:semi-discrete-form}
\end{gather}
where $\left(\cdot,\cdot\right)$ denotes the inner product, $\mathcal{F}^{\dagger}\left(y^{+},y^{-},n\right)$
is the numerical flux, and $\left\llbracket \cdot\right\rrbracket $
is the jump operator, defined such that $\left\llbracket \mathfrak{v}\right\rrbracket =\mathfrak{v}^{+}-\mathfrak{v}^{-}$
at interior interfaces and $\left\llbracket \mathfrak{v}\right\rrbracket =\mathfrak{v}^{+}$
at boundary interfaces. Due to the stiff chemical source terms, we
employ \RevisionTextThree{second-order} Strang splitting~\citep{Str68}
over a given interval $(t_{0},t_{0}+\Delta t]$ as
\begin{align}
\frac{\partial y}{\partial t}+\nabla\cdot\mathcal{F}\left(y\right)=0 & \textup{ in }\Omega\times\left(t_{0},t_{0}+\nicefrac{\Delta t}{2}\right],\label{eq:strang-splitting-1}\\
\frac{\partial y}{\partial t}-\mathcal{S}\left(y\right)=0 & \textup{ in }\left(t_{0},t_{0}+\Delta t\right],\label{eq:strang-splitting-2}\\
\frac{\partial y}{\partial t}+\nabla\cdot\mathcal{F}\left(y\right)=0 & \textup{ in }\Omega\times\left(t_{0}+\nicefrac{\Delta t}{2},t_{0}+\Delta t\right],\label{eq:strang-splitting-3}
\end{align}
where Equations~(\ref{eq:strang-splitting-1}) and~(\ref{eq:strang-splitting-3})
are integrated in time with an explicit scheme, while Equation~(\ref{eq:strang-splitting-2})
is solved using a fully implicit, temporal DG discretization for ODEs.
\RevisionTextThree{Additional discussion on operator splitting can be found in Part I~\citep{Chi22}.}

The volume and surface terms in Equation~\ref{eq:semi-discrete-form}
are evaluated with a quadrature-free approach~\citep{Atk96,Atk98}.
A nodal basis is employed, such that the element-local polynomial
approximation of the solution is given by
\begin{equation}
y_{\kappa}=\sum_{j=1}^{n_{b}}y_{\kappa}(x_{j})\phi_{j},\label{eq:solution-approximation}
\end{equation}
where $n_{b}$ is the number of basis functions, $\left\{ \phi_{1},\ldots,\phi_{n_{b}}\right\} $
are the basis functions, and $\left\{ x_{1},\ldots,x_{n_{b}}\right\} $
are the node coordinates.  The nonlinear convective flux in the second
and third integrals in Equation~(\ref{eq:semi-discrete-form}) can
be approximated as
\begin{equation}
\mathcal{F_{\kappa}}\approx\sum_{k=1}^{n_{c}}\mathcal{F}\left(y_{\kappa}\left(x_{k}\right)\right)\varphi_{k},\label{eq:flux-projection}
\end{equation}
where $n_{c}\geq n_{b}$ and $\left\{ \varphi_{1},\ldots,\varphi_{n_{c}}\right\} $
is a set of polynomial basis functions that may be different from
those in Equation~(\ref{eq:solution-approximation}). As discussed
by Johnson and Kercher~\citep{Joh20_2}, pressure equilibrium within
and between elements is maintained under either of the following conditions:
\begin{itemize}
\item $n_{c}=n_{b}$ and the integration points are in the set of solution
nodes
\item If $n_{c}>n_{b}$ (i.e., over-integration), then the flux interpolation
in \RevisionTextOne{Equation~(\ref{eq:flux-projection})} is replaced
with
\begin{equation}
\mathcal{F_{\kappa}}\approx\sum_{k=1}^{n_{c}}\mathcal{F}\left(\widetilde{y}_{\kappa}\left(x_{k}\right)\right)\varphi_{k},\label{eq:modified-flux-projection}
\end{equation}
where $\widetilde{y}:\mathbb{R}^{m}\times\mathbb{R}\rightarrow\mathbb{R}^{m}$
is a modified state defined as
\begin{equation}
\widetilde{y}\left(y,\widetilde{P}\right)=\left(\rho v_{1},\ldots,\rho v_{d},\widetilde{\rho u}\left(C_{1},\ldots,C_{n_{s}},\widetilde{P}\right)+\frac{1}{2}\sum_{k=1}^{d}\rho v_{k}v_{k},C_{1},\ldots,C_{n_{s}}\right)^{T}.\label{eq:interpolated-state-modified}
\end{equation}
\end{itemize}
$\widetilde{\rho u}$ is a modified internal energy is evaluated from
the unmodified species concentrations and a polynomial approximation
of the pressure that interpolates onto the span of $\left\{ \phi_{1},\ldots,\phi_{n_{b}}\right\} $
as
\[
\widetilde{P}_{\kappa}=\sum_{j=1}^{n_{b}}P\left(y_{\kappa}\left(x_{j}\right)\right)\phi_{j}.
\]
Standard over-integration typically fails to preserve pressure equilibrium,
resulting in the generation of spurious pressure oscillations that
can lead to solver divergence. On the other hand, employing the modified
flux interpolation~(\ref{eq:modified-flux-projection}) in both the
volume and surface flux integrals in Equation~(\ref{eq:semi-discrete-form})
achieves pressure equilibrium both internally and between adjacent
elements (except in severely underresolved computations, in which
appreciable deviations from pressure equilibrium are inevitable). 

As discussed in Part I~\citep{Chi22}, the linear-scaling limiter
used to enforce the positivity property and entropy boundedness does
not completely eliminate small-scale nonlinear instabilities, particularly
in the vicinity of flow-field discontinuities. Therefore, the artificial
dissipation term~\citep{Har13}
\begin{equation}
-\sum_{\kappa\in\mathcal{T}}\left(\mathcal{F}^{\mathrm{AV}}\left(y,\nabla y\right),\nabla\mathfrak{v}\right)_{\kappa},\label{eq:artificial-viscosity-integral}
\end{equation}
where $\mathcal{F}^{\mathrm{AV}}(y,\nabla y)=\nu_{\mathrm{AV}}\nabla y$,
is added to the LHS of Equation~(\ref{eq:semi-discrete-form}). $\nu_{\mathrm{AV}}$
is the artificial viscosity, computed as~\citep{Joh20_2}
\[
\nu_{\mathrm{AV}}=\left(C_{\mathrm{AV}}+S_{\mathrm{AV}}\right)\left(\frac{h^{2}}{p+1}\left|\frac{\partial T}{\partial y}\cdot\frac{\mathcal{R}\left(y,\nabla y\right)}{T}\right|\right),
\]
where $S_{\mathrm{AV}}$ is a shock sensor based on intra-element
variations~\citep{Chi19}, $C_{\mathrm{AV}}$ is a user-defined coefficient,
$h$ is the element size, and $\mathcal{R}\left(y,\nabla y\right)$
is the strong form of the residual~(\ref{eq:conservation-law-strong-form}).
This artificial viscosity formulation was used to successfully dampen
nonphysical oscillations near flow-field discontinuities in a variety
of multicomponent-flow simulations~\citep{Chi22,Joh20_2}. Other
types of artificial viscosity or limiters can be employed as well.
Note that the integral~(\ref{eq:artificial-viscosity-integral})
vanishes for $\mathfrak{v}\in V_{h}^{0}$, which will be important
in Section~\ref{sec:transport-step-EBDG-2D}. For further details
on the basic DG discretization, boundary-condition implementation,
and conditions under which pressure oscillations are generated, we
refer the reader to~\citep{Joh20_2}.

\section{Transport step: Entropy-bounded discontinuous Galerkin method in
multiple dimensions\label{sec:transport-step-EBDG-2D}}

In this section, we extend the one-dimensional positivity-preserving/entropy-bounded
DG scheme presented in Part I to multiple dimensions. 

\subsection{Preliminaries\label{subsec:preliminaries-EBDG-2D}}

We first review the geometric mapping from reference space to physical
space, as well as volumetric and surface quadrature rules. The minimum
entropy principle is then summarized. Finally, a first-order three-point
system, which is a building block for the general multidimensional,
high-order scheme, is discussed. 

\subsubsection{Geometric mapping}

Consider the mapping $x(\xi):\widehat{\kappa}\rightarrow\kappa$,
with $\widehat{\kappa}$ denoting the reference element, given by
\[
x(\xi)=\sum_{m=1}^{n_{g}}x_{\kappa,m}\Phi_{m}(\xi),
\]
where $\xi\in\mathbb{R}^{d}$ are the reference coordinates, $x_{\kappa,m}$
is the physical coordinate of the $m$th node of $\kappa,$ $\Phi_{m}$
is the $m$th basis function for the geometry interpolation, and $n_{g}$
is the number of basis functions. The unit (or bi-unit) square and
right triangle with unit (or bi-unit) side length are common reference
elements for quadrilateral and triangular elements, respectively.
The local solution approximation can then be written as
\[
y_{\kappa}=\sum_{j=1}^{n_{b}}y_{\kappa}(x_{j})\phi(\xi),\quad x=x(\xi)\in\kappa,\;\forall\xi\in\widehat{\kappa}.
\]
Let $\partial\kappa^{(f)}$ be the $f$th face of $\kappa.$ The geometric
mapping $x\left(\zeta^{(f)}\right):\widehat{\epsilon}\rightarrow\partial\kappa^{(f)}$
for interfaces, where $\zeta^{(f)}\in\mathbb{R}^{d-1}$ are the reference
coordinates and $\widehat{\epsilon}$ is the reference face, is given
by
\[
x\left(\zeta^{(f)}\right)=\sum_{m=1}^{n_{g,f}^{\partial}}x_{\kappa,m}^{(f)}\Phi_{m}^{(f)}\left(\zeta^{(f)}\right),
\]
where $x_{\kappa,m}^{(f)}$ is the physical coordinate of the $m$th
node of $\partial\kappa^{(f)}$, $\Phi_{m}^{(f)}$ is the $m$th basis
function, and $n_{g,f}^{\partial}$ is the number of basis functions.
The reference face can also be mapped to the reference element (i.e.,
$\xi\left(\zeta^{(f)}\right):\widehat{\epsilon}\rightarrow\widehat{\kappa}$).

\subsubsection{Quadrature rules}

Consider a volumetric quadrature rule with points $\xi_{v}$ and positive
weights $w_{v}$ that satisfy $\sum_{v=1}^{n_{q}}w_{v}=\left|\widehat{\kappa}\right|$,
where $\left|\widehat{\kappa}\right|$ is the volume of the reference
element. The volume integral over $\kappa$ of a generic function,
$g(x)$, can be evaluated as
\[
\int_{\kappa}g(x)dx=\int_{\widehat{\kappa}}g(x(\xi))\left|J_{\kappa}(\xi)\right|d\xi\approx\sum_{v=1}^{n_{q}}g\left(x(\xi_{v})\right)\left|J_{\kappa}(\xi_{v})\right|w_{v},
\]
where $J_{\kappa}$ is the geometric Jacobian and $\left|J_{\kappa}\right|$
is its determinant. The integral evaluation is exact if $g(x)$ is
a polynomial and the quadrature rule is sufficiently accurate. Similarly,
let $\zeta_{l}$ and $w_{l}^{\partial}$ be the points and positive
weights of a surface quadrature rule, with $\sum_{l=1}^{n_{q}^{\partial}}w_{l}^{\partial}=\left|\widehat{\epsilon}\right|$,
where $\left|\widehat{\epsilon}\right|$ is the surface area of the
reference face. The surface integral over $\partial\kappa^{(f)}$
of a generic function can be evaluated as
\[
\int_{\partial\kappa^{(f)}}g(x)ds=\int_{\widehat{\epsilon}}g\left(x\left(\zeta^{(f)}\right)\right)\left|J_{\partial\kappa}^{(k)}\left(\zeta^{(f)}\right)\right|d\zeta\approx\sum_{l=1}^{n_{q}^{\partial}}g\left(x\left(\zeta_{l}^{(f)}\right)\right)\left|J_{\partial\kappa}^{(f)}\left(\zeta_{l}^{(f)}\right)\right|w_{f,l}^{\partial}=\sum_{l=1}^{n_{q}^{\partial}}g\left(x\left(\zeta_{l}^{(f)}\right)\right)\nu_{f,l}^{\partial},
\]
where $\nu_{f,l}^{\partial}=\left|J_{\partial\kappa}^{(f)}\left(\zeta_{l}^{(f)}\right)\right|w_{f,l}^{\partial}$
and $J_{\partial\kappa}^{(f)}$ is the surface Jacobian. The integral
evaluation is again exact if $g(x)$ is a polynomial and the quadrature
rule is sufficiently accurate. The closed surface integral over $\partial\kappa$
of a generic function is evaluated as
\[
\int_{\partial\kappa}g(x)ds=\sum_{f=1}^{n_{f}}\int_{\partial\kappa^{(f)}}g(x)ds=\sum_{f=1}^{n_{f}}\int_{\widehat{\epsilon}}g\left(x\left(\zeta^{(f)}\right)\right)\left|J_{\partial\kappa}^{(f)}\left(\zeta^{(f)}\right)\right|d\zeta\approx\sum_{f=1}^{n_{f}}\sum_{l=1}^{n_{q,f}^{\partial}}g\left(x\left(\zeta_{l}^{(f)}\right)\right)\nu_{f,l}^{\partial},
\]
where a different quadrature rule can be employed for each face. $n_{f}$,
the number of faces, is allowed to change among elements, but a slight
abuse of notation is done for simplicity.

Given that, as mentioned in Section~\ref{sec:DG-discretization},
a quadrature-free approach~\citep{Atk96,Atk98} is employed in this
work to evaluate the integrals in Equation~(\ref{eq:semi-discrete-form}),
the reader may question why quadrature rules are reviewed in such
detail. As in Part I, the first step in our analysis in Section~\ref{subsec:entropy-bounded-high-order-DG-2D}
is to take $\mathfrak{v}\in V_{h}^{0}$ in Equation~(\ref{eq:semi-discrete-form})
to obtain the scheme satisfied by the element averages, which remains
the same whether a quadrature-based or quadrature-free approach is
employed. Furthermore, in Section~\ref{subsec:entropy-bounded-high-order-DG-2D}
we present our analysis in terms of a quadrature-based approach for
consistency with previous studies on positivity-preserving and/or
entropy-bounded DG methods.

\subsubsection{Minimum entropy principle}

Let $U(y):\mathbb{R}^{m}\rightarrow\mathbb{R}$ be a given convex
(mathematical) entropy function and $\mathcal{F}^{s}(y):\mathbb{R}^{m}\rightarrow\mathbb{R}^{d}$
the corresponding spatial entropy flux. Satisfaction of the entropy
inequality
\begin{equation}
\frac{\partial U}{\partial t}+\nabla\cdot\mathcal{F}^{s}\leq0,\label{eq:entropy-condition}
\end{equation}
distinguishes physical solutions from nonphysical solutions when discontinuities
are present. Specifically, \emph{entropy solutions} are those that
satisfy~(\ref{eq:entropy-condition}) for all entropy/entropy-flux
pairs. $U=-\rho s$ and $\mathcal{F}^{s}=-\rho sv$ form a common,
admissible entropy/entropy-flux pair for the multicomponent Euler
equations~\citep{Gou20,Gio99}. 

The minimum entropy principle, which states that the spatial minimum
of the specific thermodynamic entropy is non-decreasing in time, follows
from a particular choice of entropy/entropy-flux pair,
\begin{equation}
\left(U,\mathcal{F}^{s}\right)=\left(-\rho f_{0}(s),-\rho vf_{0}(s)\right),
\end{equation}
where $f_{0}(s)=\min\left\{ s-s_{0},0\right\} $. Although $f_{0}(s)$
is not a smooth function of $s$, it can be written as the limit of
a sequence of smooth functions, $f_{0}(s)=\underset{\epsilon\rightarrow0}{\lim}f_{\epsilon}(s)$,
where $f_{\epsilon}(s)$ is given by~\citep{Gou20}
\[
f_{\epsilon}(s)=\frac{1}{\epsilon}\int_{-\infty}^{\infty}f_{0}(s-\mathfrak{s})\frac{\exp\left(-\frac{\mathfrak{s}^{2}}{\epsilon^{2}}\right)}{\sqrt{\pi}}d\mathfrak{s},\quad\epsilon>0.
\]
According to the minimum entropy principle, we have, for $\left|x\right|\leq R$,
\begin{equation}
s(x,t)\geq s_{0}=\underset{\left|x\right|\leq R+v_{\max}t}{\text{Ess inf}}s(x,0),\label{eq:minimum-entropy-principle}
\end{equation}
where $R>0$ and $v_{\max}$ is the maximum speed in the domain at
$t=0$. The inequality~(\ref{eq:minimum-entropy-principle}) was
first proved by Tadmor~\citep{Tad86} for the monocomponent Euler
equations. It was recently generalized to the multicomponent, nonreacting
Euler equations by Gouasmi et al.~\citep{Gou20} and then extended
to the reacting case in Part I~\citep{Chi22}.

\subsubsection{\protect\RevisionTextThree{First-order} three-point system}

Let $\mathcal{G}_{\sigma}$ be the following set:
\begin{equation}
\mathcal{G_{\sigma}}=\left\{ y\mid C_{1}>0,\ldots,C_{n_{s}}>0,\rho u^{*}>0,s\geq\sigma\right\} ,
\end{equation}
where $\sigma\in\mathbb{R}$ and $u^{*}$ is the ``shifted'' internal
energy~\citep{Hua19},
\begin{equation}
u^{*}=u-u_{0}=u-\sum_{i=1}^{n_{s}}Y_{i}b_{i0},\label{eq:shifted-internal-energy}
\end{equation}
such that $u^{*}>0$ if and only if $T>0$, provided $c_{v,i}>0,\:i=1,\ldots,n_{s}$~\citep{Gio99}.
Note that $y\in\mathcal{G}_{\sigma}$ implies $P(y)>0$. By the quasi-concavity
of $s(y)$ (which follows from the convexity of the entropy function
$U=-\rho s$~\citep{Zha12_2}) and concavity of $\rho u^{*}(y)$~\citep{Chi22}
with respect to the state, for a given $\sigma$, $\mathcal{G}_{\sigma}$
is a convex set. Note that $\mathcal{G}_{\sigma}$ is similar to the
corresponding set of admissible states in~\citep{Du19,Du19_2}, but
with the addition of the entropy constraint. Under the assumption
that the exact solution to the classical Riemann problem with initial
data 
\[
y\left(x,0\right)=\begin{cases}
y_{1}, & x<0\\
y_{2}, & x>0
\end{cases}
\]
is an entropy solution that preserves positivity, $\mathcal{G}_{\sigma}$
is an invariant set~\citep{Gue19,Gue16_2,Chi22}. Specifically, $y_{1},y_{2}\in\mathcal{G}_{\sigma}$
implies that the average of the exact Riemann solution over a domain
that includes the Riemann fan is also in $\mathcal{G}_{\sigma}$. 

Consider the one-dimensional three-point system,
\begin{gather}
y_{\kappa}^{j+1}=y_{\kappa}^{j}-\frac{\Delta t}{h}\left[\mathcal{F}^{\dagger}\left(y_{\kappa}^{j},y_{\kappa^{(1)}}^{j},-1\right)+\mathcal{F}^{\dagger}\left(y_{\kappa}^{j},y_{\kappa^{(2)}}^{j},1\right)\right],\label{eq:three-point-system}
\end{gather}
which corresponds to a $p=0$, one-dimensional, element-local DG discretization
with forward Euler time stepping, where $j$ indexes the time step,
$h$ is the element size, and $\kappa^{(1)}$ and $\kappa^{(2)}$
are the elements to the left and right of $\kappa$, respectively.
If $\mathcal{F}^{\dagger}$ is an \emph{invariant-region-preserving}
numerical flux, then $y_{\kappa}^{j},y_{\kappa_{L}}^{j},y_{\kappa_{R}}^{j}\in\mathcal{G}_{\sigma}$
implies $y_{\kappa}^{j+1}\in\mathcal{G}_{\sigma}$ under the condition~\citep{Jia18}
\begin{equation}
\frac{\Delta t\lambda}{h}\leq\frac{1}{2},\label{eq:p0-time-step-constraint}
\end{equation}
where $\lambda$ is an upper bound on the maximum wave speed of the
system. In particular, $y_{\kappa}^{j+1}$ then satisfies~\citep{Jia18,Wu21_2}
\begin{equation}
s\left(y_{\kappa}^{j+1}\right)\geq\min\left\{ s\left(y_{\kappa_{L}}^{j}\right),s\left(y_{\kappa^{(1)}}^{j}\right),s\left(y_{\kappa^{(2)}}^{j}\right)\right\} .\label{eq:p0-minimum-entropy-principle}
\end{equation}
Examples of invariant-region-preserving fluxes are the Godunov, Lax-Friedrichs,
HLL, and HLLC fluxes (see~\citep{Jia18}).

\begin{figure}[tbph]
\begin{centering}
\includegraphics[width=0.4\columnwidth]{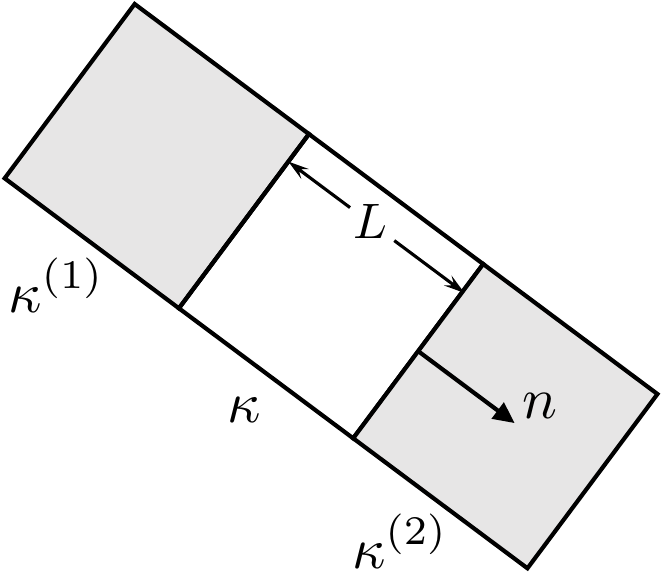}
\par\end{centering}
\caption{\label{fig:quasi-1D_system}Schematic of \protect\RevisionTextThree{a rotated three-point system in the multidimensional setting (Equation~(\ref{eq:three-point-system-2D}))}.}
\end{figure}

Next, we analyze a general three-point system in the multidimensional
setting:
\begin{gather}
y_{\kappa}^{j+1}=y_{\kappa}^{j}-\frac{\Delta t}{L}\left[\mathcal{F}^{\dagger}\left(y_{\kappa}^{j},y_{\kappa^{(1)}}^{j},n\right)+\mathcal{F}^{\dagger}\left(y_{\kappa}^{j},y_{\kappa^{\left(2\right)}}^{j},-n\right)\right],\label{eq:three-point-system-2D}
\end{gather}
where a schematic of this three-point system is given in Figure~\ref{fig:quasi-1D_system}.
 To analyze the classical Riemann problems associated with the interfaces,
we introduce a rotated coordinate system with orthonormal basis $\left\{ n,t_{1},\ldots,t_{d-1}\right\} $.
With this change of basis and letting $x^{\parallel}$ denote the
coordinate in the $n$-direction, the projected (homogeneous) governing
equations are given by~\citep{Gue16}
\begin{equation}
\frac{\partial y}{\partial t}+\frac{\partial}{\partial x^{\parallel}}\left(n\cdot\mathcal{F}(y)\right)=0,\label{eq:projected-governing-equations}
\end{equation}
with

\[
y=\left(\rho v^{\parallel},\rho v_{1}^{\perp}\ldots,\rho v_{d-1}^{\perp},\rho e_{t},C_{1},\ldots,C_{n_{s}}\right)^{T}
\]
and 
\[
n\cdot\mathcal{F}(y)=\left(\rho\left(v^{\parallel}\right)^{2}+P,\rho v^{\parallel}v_{1}^{\perp},\ldots,\rho v^{\parallel}v_{d-1}^{\perp},v^{\parallel}\left(\rho e_{t}+P\right),v^{\parallel}C_{1},\ldots,v^{\parallel}C_{n_{s}}\right)^{T},
\]
where the velocity is expanded as 
\begin{eqnarray*}
v & = & \left(v^{\parallel},v_{1}^{\perp},\ldots,v_{d-1}^{\perp}\right)\\
 & = & \left(v\cdot n,v\cdot t_{1},\ldots,v\cdot t_{d-1}\right).
\end{eqnarray*}
An equivalent system is
\begin{equation}
\begin{cases}
\frac{\partial\rho_{1}}{\partial t}+\frac{\partial}{\partial x_{1}}\left(\rho_{1}v^{\parallel}\right)=0\\
\quad\quad\quad\vdots\\
\frac{\partial\rho_{n_{s}}}{\partial t}+\frac{\partial}{\partial x_{1}}\left(\rho_{n_{s}}v^{\parallel}\right)=0\\
\frac{\partial v^{\parallel}}{\partial t}+v^{\parallel}\frac{\partial}{\partial x_{1}}\left(v^{\parallel}\right)+\frac{1}{\rho}\frac{\partial P(\rho,s)}{\partial x_{1}}=0\\
\frac{\partial v_{1}^{\perp}}{\partial t}+v^{\parallel}\frac{\partial}{\partial x_{1}}\left(v_{1}^{\perp}\right)=0\\
\quad\quad\quad\vdots\\
\frac{\partial v_{d-1}^{\perp}}{\partial t}+v^{\parallel}\frac{\partial}{\partial x_{1}}\left(v_{d-1}^{\perp}\right)=0\\
\frac{\partial s}{\partial t}+v^{\parallel}\frac{\partial}{\partial x_{1}}\left(s\right)=0
\end{cases},\label{eq:equivalent-system}
\end{equation}
where the last equation is the entropy transport equation (see~\citep[Section 4]{Chi22}).
The Jacobian of~(\ref{eq:equivalent-system}) is
\[
\left(\begin{array}{cccccccc}
v^{\parallel} & 0 & 0 & \rho_{1} & 0 & \ldots & 0 & 0\\
0 & \ddots & 0 & \vdots & \vdots & \ddots & \vdots & \vdots\\
0 & 0 & v^{\parallel} & \rho_{n_{s}} & 0 & 0\ldots & 0 & 0\\
\frac{1}{\rho}\frac{\partial P}{\partial\rho} & \ldots & \frac{1}{\rho}\frac{\partial P}{\partial\rho} & v^{\parallel} & 0 & 0\ldots & 0 & \frac{1}{\rho}\frac{\partial P}{\partial s}\\
0 & \ldots & 0 & 0 & v^{\parallel} & 0 & \ldots & 0\\
\vdots & \ddots & \vdots & \vdots & \ddots & \ddots & \ddots & \vdots\\
0 & \ldots & 0 & 0 & \ldots & 0 & v^{\parallel} & 0\\
0 & \ldots & 0 & 0 & 0 & \ldots & 0 & v^{\parallel}
\end{array}\right),
\]
which does not depend on $v^{\perp}=\left(v_{1}^{\perp},\ldots,v_{d-1}^{\perp}\right)$.
The exact solution of the classical Riemann problem associated with~(\ref{eq:projected-governing-equations})
with initial data
\[
y\left(x^{\parallel},0\right)=\begin{cases}
y_{1}, & x^{\parallel}<0\\
y_{2}, & x^{\parallel}>0
\end{cases}
\]
can then be obtained in two steps~\citep[Chapter 4.8]{Gue16,Tor13}.
The first step is to solve the one-dimensional classical Riemann problem
associated with the equations
\begin{equation}
\frac{\partial y^{\parallel}}{\partial t}+\frac{\partial\mathcal{F}\left(y^{\parallel}\right)}{\partial x^{\parallel}}=0,\label{eq:governing-equations-projected-1D}
\end{equation}
where
\[
y^{\parallel}=\left(\rho_{1},\ldots,\rho_{n_{s}},\rho v^{\parallel},\rho e_{t}^{\parallel}\right)^{T}
\]
and 
\[
\mathcal{F}\left(y^{\parallel}\right)=\left(\rho_{1}v^{\parallel},\ldots,\rho_{n_{s}}v^{\parallel},\rho\left(v^{\parallel}\right)^{2}+P,v^{\parallel}\left(e_{t}^{\parallel}+P\right)\right)^{T},
\]
with $e_{t}^{\parallel}=e_{t}-\left(v^{\parallel}\right)^{2}/2$.
The second step is to solve
\[
\frac{\partial}{\partial t}\left(\begin{array}{c}
\rho v_{1}^{\perp}\\
\vdots\\
\rho v_{d-1}^{\perp}
\end{array}\right)+\frac{\partial}{\partial x^{\parallel}}\left(\begin{array}{c}
\rho v^{\parallel}v_{1}^{\perp}\\
\vdots\\
\rho v^{\parallel}v_{d-1}^{\perp}
\end{array}\right)=0.
\]
See~\citep{Gue16} and~\citep[Chapter 4.8]{Tor13} for more information.
Notice that 
\begin{eqnarray*}
 & \rho(y)=\rho\left(y^{\parallel}\right),\quad\rho u^{*}(y)=\rho u^{*}\left(y^{\parallel}\right),\quad s(y)=s\left(y^{\parallel}\right),\\
 & C_{i}(y)=C_{i}\left(y^{\parallel}\right),i=1,\ldots,n_{s}.
\end{eqnarray*}
Consequently, since $\mathcal{G}_{\sigma}$ is an invariant set for~(\ref{eq:governing-equations-projected-1D}),
it is also an invariant set for the projected equations~(\ref{eq:projected-governing-equations}).
The invariant-region-preserving numerical fluxes for the one-dimensional
three-point system\RevisionTextThree{~(\ref{eq:three-point-system})}
can then be shown to be invariant-region-preserving for \RevisionTextThree{the rotated (quasi-one-dimensional)}
three-point system~(\ref{eq:three-point-system-2D}). Specifically,
as in~\citep{Jia18} for the one-dimensional case, the RHS of~(\ref{eq:three-point-system-2D})
can be rewritten as a convex combination of $y_{\kappa}^{j}$ and
exact-Riemann-solution averages. As such, if $y_{\kappa}^{j}$, $y_{\kappa^{(1)}}^{j}$,
and $y_{\kappa^{(2)}}^{j}$ \RevisionTextThree{in~(\ref{eq:three-point-system-2D})}
are in $\mathcal{G}_{\sigma}$, then $y_{\kappa}^{j+1}$ is also in
$\mathcal{G}_{\sigma}$ under the time-step-size constraint
\[
\frac{\Delta t\lambda}{L}\leq\frac{1}{2},
\]
where $\lambda$ is an upper bound on the maximum wave speed of the
system. \RevisionTextThree{Note that a multidimensional scheme has not yet been analyzed; such
analysis will be performed in the following subsection.} Throughout this work, we employ the HLLC numerical flux~\citep{Tor13}. 

The analysis of the \RevisionTextThree{rotated} three-point system~(\ref{eq:three-point-system-2D})
is essential for the construction of a positivity-preserving and entropy-bounded
DG scheme for $p\geq0$ on arbitrary, curved elements. Specifically,
as we demonstrate in Section~\ref{subsec:entropy-bounded-high-order-DG-2D},
the element average of the solution (for $p>0$) at the $(j+1)$th
time step, $\overline{y}_{\kappa}^{j+1}$, can be decomposed into
a convex combination of both pointwise values of $y_{\kappa}^{j}(x)$
and three-point systems involving pointwise values of $y_{\kappa}^{j}(x)$.
Therefore, with the aid of a simple limiting procedure to ensure that
said pointwise values of $y_{\kappa}^{j}(x)$ are in $\mathcal{G}_{\sigma}$,
$\overline{y}_{\kappa}^{j+1}$ will also be in $\mathcal{G}_{\sigma}$
under a time-step-size constraint.

\subsection{Entropy-bounded, high-order discontinuous Galerkin method in multiple
dimensions~\label{subsec:entropy-bounded-high-order-DG-2D}}

We are now in a position to derive a time-step constraint that ensures
that the evolved element average is in $\mathcal{G}$. The average
of $y_{\kappa}$ is given by
\begin{equation}
\overline{y}_{\kappa}=\frac{1}{\left|\kappa\right|}\int_{\kappa}ydx,\label{eq:solution-element-average}
\end{equation}
where $\left|\kappa\right|$ is the volume of $\kappa$. Let $\partial\mathcal{D}_{\kappa}$
be the set of surface integration points used to evaluate the surface
integrals in Equation~(\ref{eq:semi-discrete-form}) (i.e., the numerical
flux terms), defined as
\[
\partial\mathcal{D}_{\kappa}=\bigcup_{f=1}^{n_{f}}\left\{ x\left(\zeta_{l}^{(f)}\right),l=1,\ldots,n_{q,f}^{\partial}\right\} .
\]
As discussed by Lv and Ihme~\citep{Lv15_2}, given a sufficiently
accurate quadrature rule, the element average can be expanded as
\begin{align}
\overline{y}_{\kappa} & =\sum_{v=1}^{n_{q}}\frac{\left|J_{\kappa}(\xi_{v})\right|w_{v}}{|\kappa|}y_{\kappa}\left(\xi_{v}\right),\nonumber \\
 & =\sum_{v=1}^{n_{q}}\theta_{v}y_{\kappa}\left(\xi_{v}\right)+\sum_{f=1}^{n_{f}}\sum_{l=1}^{n_{q,f}^{\partial}}\theta_{f,l}y_{\kappa}\left(\xi\left(\zeta_{l}^{(f)}\right)\right).\label{eq:element-average-convex-combination-2D}
\end{align}
If $\partial\mathcal{D}_{\kappa}\subseteq\left\{ x(\xi_{v}),v=1,\ldots,n_{q}\right\} $,
i.e., this set of volume quadrature points contains the surface integration
points, then we can simply take 
\[
\theta_{v}=\begin{cases}
\frac{\left|J_{\kappa}(\xi_{v})\right|w_{v}}{|\kappa|} & x\left(\xi_{v}\right)\notin\partial\mathcal{D}_{\kappa}\\
0 & x\left(\xi_{v}\right)\in\partial\mathcal{D}_{\kappa}
\end{cases}
\]
and
\[
\theta_{f,l}=\frac{\left|J_{\kappa}\left(\xi\left(\zeta_{l}^{(f)}\right)\right)\right|w_{f,l}}{|\kappa|N_{f,l}},
\]
where $w_{f,l}$ is the volume quadrature weight corresponding to
the volume quadrature point that satisfies $\xi_{v}=\xi\left(\zeta_{l}^{(f)}\right)$
and $N_{f,l}$ denotes the number of faces belonging to $\kappa$
shared by the given point. If $\partial\mathcal{D}_{\kappa}\nsubseteq\left\{ x(\xi_{v}),v=1,\ldots,n_{q}\right\} $,
we can instead take 
\[
\theta_{v}=\frac{\left|J_{\kappa}(\xi_{v})\right|w_{v}}{|\kappa|}-\sum_{f=1}^{n_{f}}\sum_{l=1}^{n_{q,f}^{\partial}}\theta_{f,l}\psi_{v}\left(\xi\left(\zeta_{l}^{(f)}\right)\right),
\]
where $\left\{ \psi_{1},\ldots,\psi_{n_{d}}\right\} $ is a set of
Lagrange basis functions, with $n_{b}\leq n_{d}\leq n_{q}$, whose
interpolation nodes are located at a subset of the $n_{q}$ quadrature
points, while $\psi_{v}=0$ for $v=n_{d}+1,\ldots,n_{q}$, such that
\begin{align*}
\sum_{v=1}^{n_{q}}\theta_{v}y_{\kappa}\left(\xi_{v}\right) & =\sum_{v=1}^{n_{q}}\left[\frac{\left|J_{\kappa}(\xi_{v})\right|w_{v}}{|\kappa|}-\sum_{f=1}^{n_{f}}\sum_{l=1}^{n_{q,f}^{\partial}}\theta_{f,l}\psi_{v}\left(\xi\left(\zeta_{l}^{(f)}\right)\right)\right]y_{\kappa}\left(\xi_{v}\right)\\
 & =\sum_{v=1}^{n_{q}}\frac{\left|J_{\kappa}(\xi_{v})\right|w_{v}}{|\kappa|}y_{\kappa}\left(\xi_{v}\right)-\sum_{f=1}^{n_{f}}\sum_{l=1}^{n_{q,f}^{\partial}}\theta_{f,l}\sum_{v=1}^{n_{q}}y_{\kappa}\left(\xi_{v}\right)\psi_{v}\left(\xi\left(\zeta_{l}^{(f)}\right)\right)\\
 & =\sum_{v=1}^{n_{q}}\frac{\left|J_{\kappa}(\xi_{v})\right|w_{v}}{|\kappa|}y_{\kappa}\left(\xi_{v}\right)-\sum_{f=1}^{n_{f}}\sum_{l=1}^{n_{q,f}^{\partial}}\theta_{f,l}y_{\kappa}\left(\xi\left(\zeta_{l}^{(f)}\right)\right).
\end{align*}
$\theta_{f,l}$ will be related to a time-step-size constraint later
in this section. Note that $\sum_{v=1}^{n_{q}}\theta_{v}+\sum_{f=1}^{n_{f}}\sum_{l=1}^{n_{q,f}^{\partial}}\theta_{f,l}=1.$The
positivity of the quadrature weights guarantees the existence of positive
values of $\theta_{f,l}$ that yield $\theta_{v}\geq0$~\citep{Lv15_2}.
Since $\sum_{v}\theta_{v}+\sum_{f}\sum_{l}\theta_{f,l}=1$, the RHS
of Equation~(\ref{eq:element-average-convex-combination-2D}) is
a convex combination of $y_{\kappa}$ evaluated at a set of points
in $\kappa$.  Let $\kappa^{(f)}$ denote the $f$th neighbor of
$\kappa$, and let $\mathcal{D}_{\kappa}$ be the set of points at
which the solution is evaluated in Equation~(\ref{eq:element-average-convex-combination-2D}),
\[
\mathcal{D_{\kappa}}=\partial\mathcal{D}_{\kappa}\bigcup\left\{ x(\xi_{v}),v=1,\ldots,n_{q}\right\} =\bigcup_{f=1}^{n_{f}}\left\{ x\left(\zeta_{l}^{(f)}\right),l=1,\ldots,n_{q,f}^{\partial}\right\} \bigcup\left\{ x(\xi_{v}),v=1,\ldots,n_{q}\right\} .
\]
Without loss of generality, we assume that the $n_{f}$th face is
such that $N=\max_{f}\left\{ n_{q,f}^{\partial}\right\} =n_{q,n_{f}}^{\partial}$
and define $\nu_{f,l}^{\partial}$ as
\begin{equation}
\nu_{f,l}^{\partial}=\begin{cases}
\left|J_{\partial\kappa}^{(f)}(\zeta_{l})\right|w_{f,l}^{\partial}, & l=1,\ldots,n_{q,f}^{\partial}\\
0, & l=n_{q,f}^{\partial}+1,\ldots,N
\end{cases},\label{eq:Jacobian_weights_piecewise}
\end{equation}
such that
\[
\sum_{f=1}^{n_{f}}\sum_{l=1}^{N}\nu_{f,l}^{\partial}=\sum_{f=1}^{n_{f}}\sum_{l=1}^{n_{q,f}^{\partial}}\nu_{f,l}^{\partial}=\sum_{f=1}^{n_{f}}\left|\partial\kappa^{(f)}\right|=\left|\partial\kappa\right|,
\]
where $|\partial\kappa|$ is the surface area of $\kappa$ and $\left|\partial\kappa^{(f)}\right|$
is the surface area of the $f$th face. Standard flux interpolation,
as in Equation~(\ref{eq:flux-projection}), is assumed here. In Section~\ref{subsec:modified-flux-interpolation-2D},
we will account for the modified flux interpolation in Equation~(\ref{eq:modified-flux-projection}).
Taking $\mathfrak{v}\in V_{h}^{0}$ in Equation~(\ref{eq:semi-discrete-form}),
the scheme satisfied by the element averages can then be expanded
as
\begin{eqnarray}
\overline{y}_{\kappa}^{j+1} & = & \overline{y}_{\kappa}^{j}-\frac{\Delta t}{\left|\kappa\right|}\sum_{f=1}^{n_{f}}\int_{\partial\kappa^{(f)}}\mathcal{F}^{\dagger}\left(y_{\kappa}^{j},y_{\kappa^{(f)}}^{j},n\right)ds\nonumber \\
 & = & \overline{y}_{\kappa}^{j}-\sum_{f=1}^{n_{f}}\sum_{l=1}^{n_{q,f}^{\partial}}\frac{\Delta t\nu_{f,l}^{\partial}}{|\kappa|}\mathcal{F}^{\dagger}\left(y_{\kappa}^{j}\left(\xi\left(\zeta_{l}^{(f)}\right)\right),y_{\kappa^{(f)}}^{j}\left(\xi\left(\zeta_{l}^{(f)}\right)\right),n\left(\zeta_{l}^{(f)}\right)\right)\label{eq:fully-discrete-form-average-2D-unexpanded}\\
 & = & \sum_{v=1}^{n_{q}}\theta_{v}y_{\kappa}^{j}\left(\xi_{v}\right)+\sum_{f=1}^{n_{f}}\sum_{l=1}^{n_{q,f}^{\partial}}\left[\theta_{f,l}y_{\kappa}^{j}\left(\xi\left(\zeta_{l}^{(f)}\right)\right)-\frac{\Delta t\nu_{f,l}^{\partial}}{|\kappa|}\mathcal{F}^{\dagger}\left(y_{\kappa}^{j}\left(\xi\left(\zeta_{l}^{(f)}\right)\right),y_{\kappa^{(f)}}^{j}\left(\xi\left(\zeta_{l}^{(f)}\right)\right),n\left(\zeta_{l}^{(f)}\right)\right)\right]\nonumber \\
 & = & \sum_{v=1}^{n_{q}}\theta_{v}y_{\kappa}^{j}\left(\xi_{v}\right)+\sum_{f=1}^{n_{f}-1}\sum_{l=1}^{n_{q,f}^{\partial}}\theta_{f,l}A_{f,l}+\sum_{l=1}^{N-1}\theta_{n_{f},l}B_{l}+\theta_{n_{f},N}C.\label{eq:fully-discrete-form-average-2D}
\end{eqnarray}
Note that the volume quadrature rule used to expand $\overline{y}_{\kappa}^{j}$
need not be explicitly used to evaluate any integrals in Equation~(\ref{eq:semi-discrete-form}).
Equation~(\ref{eq:fully-discrete-form-average-2D-unexpanded}) and
thus Equation~(\ref{eq:fully-discrete-form-average-2D}) still hold
for the quadrature-free approach~\citep{Atk96,Atk98} employed in
this work since the weights $w_{v}$, $w_{f,l}$, and $w_{f,l}^{\partial}$
can be taken to be the integrals of the basis functions over the reference
element/face, corresponding to a generalized Newton-Cotes quadrature
rule~\citep{Win08}. \RevisionTextThree{Consequently, $\theta_{v}$, $\theta_{f,l}$, and $\nu_{f,l}^{\partial}$
can be treated in the same manner as in a quadrature-based approach.}Notice that $\overline{y}_{\kappa}^{j+1}$ in Equation~(\ref{eq:fully-discrete-form-average-2D})
is expressed as a convex combination of $A_{f,l}$ , $B_{l}$, $C$,
and pointwise values $y_{\kappa}^{j}(x_{v})$. $A_{f,l}$, $B_{l}$,
and $C$ in Equation~(\ref{eq:fully-discrete-form-average-2D}) are
given by
\begin{eqnarray}
A_{f,l} & = & y_{\kappa}^{j}\left(\xi\left(\zeta_{l}^{(f)}\right)\right)-\frac{\Delta t\nu_{f,l}^{\partial}}{\theta_{f,l}|\kappa|}\left[\mathcal{F}^{\dagger}\left(y_{\kappa}^{j}\left(\xi\left(\zeta_{l}^{(f)}\right)\right),y_{\kappa^{(f)}}^{j}\left(\xi\left(\zeta_{l}^{(f)}\right)\right),n\left(\zeta_{l}^{(f)}\right)\right)\right.\nonumber \\
 &  & \left.+\mathcal{F}^{\dagger}\left(y_{\kappa}^{j}\left(\xi\left(\zeta_{l}^{\left(f\right)}\right)\right),y_{\kappa}^{j}\left(\xi\left(\zeta_{l}^{\left(n_{f}\right)}\right)\right),-n\left(\zeta_{l}^{(f)}\right)\right)\right],\label{eq:A_f_l-definition}
\end{eqnarray}
for $f=1,\ldots,n_{f}-1,\;l=1,\ldots,n_{q,f}^{\partial}$;
\begin{eqnarray}
B_{l} & = & y_{\kappa}^{j}\left(\xi\left(\zeta_{l}^{\left(n_{f}\right)}\right)\right)\nonumber \\
 &  & -\frac{\Delta t\nu_{n_{f},l}^{\partial}}{\theta_{n_{f},l}|\kappa|}\biggl[\mathcal{F}^{\dagger}\left(y_{\kappa}^{j}\left(\xi\left(\zeta_{l}^{\left(n_{f}\right)}\right)\right),y_{\kappa^{\left(n_{f}\right)}}^{j}\left(\xi\left(\zeta_{l}^{\left(n_{f}\right)}\right)\right),n\left(\zeta_{l}^{\left(n_{f}\right)}\right)\right)\nonumber \\
 &  & +\mathcal{F}^{\dagger}\left(y_{\kappa}^{j}\left(\xi\left(\zeta_{l}^{\left(n_{f}\right)}\right)\right),y_{\kappa}^{j}\left(\xi\left(\zeta_{N}^{\left(n_{f}\right)}\right)\right),-n\left(\zeta_{l}^{\left(n_{f}\right)}\right)\right)\biggr]\nonumber \\
 &  & -\sum_{f=1}^{n_{f}-1}\frac{\Delta t\nu_{f,l}^{\partial}}{\theta_{n_{f},l}|\kappa|}\left[\mathcal{F}^{\dagger}\left(y_{\kappa}^{j}\left(\xi\left(\zeta_{l}^{\left(n_{f}\right)}\right)\right),y_{\kappa}^{j}\left(\xi\left(\zeta_{l}^{\left(f\right)}\right)\right),n\left(\zeta_{l}^{\left(f\right)}\right)\right)\right.\nonumber \\
 &  & \left.+\mathcal{F}^{\dagger}\left(y_{\kappa}^{j}\left(\xi\left(\zeta_{l}^{\left(n_{f}\right)}\right)\right),y_{\kappa}^{j}\left(\xi\left(\zeta_{N}^{\left(n_{f}\right)}\right)\right),-n\left(\zeta_{l}^{\left(f\right)}\right)\right)\right],\label{eq:B_l-definition}
\end{eqnarray}
for $l=N-1$; and 
\begin{eqnarray}
C & = & y_{\kappa}^{j}\left(\xi\left(\zeta_{N}^{\left(n_{f}\right)}\right)\right)\nonumber \\
 &  & -\frac{\Delta t\nu_{n_{f},N}^{\partial}}{\theta_{n_{f},N}|\kappa|}\mathcal{F}^{\dagger}\left(y_{\kappa}^{j}\left(\xi\left(\zeta_{N}^{\left(n_{f}\right)}\right)\right),y_{\kappa^{\left(n_{f}\right)}}^{j}\left(\xi\left(\zeta_{N}^{\left(n_{f}\right)}\right)\right),n\left(\zeta_{N}^{\left(n_{f}\right)}\right)\right)\nonumber \\
 &  & -\sum_{l=1}^{N-1}\frac{\Delta t\nu_{n_{f},l}^{\partial}}{\theta_{n_{f},N}|\kappa|}\mathcal{F}^{\dagger}\left(y_{\kappa}^{j}\left(\xi\left(\zeta_{N}^{\left(n_{f}\right)}\right)\right),y_{\kappa}^{j}\left(\xi\left(\zeta_{l}^{\left(n_{f}\right)}\right)\right),n\left(\zeta_{l}^{\left(n_{f}\right)}\right)\right)\nonumber \\
 &  & -\sum_{f=1}^{n_{f}-1}\frac{\Delta t\nu_{f,N}^{\partial}}{\theta_{n_{f},N}|\kappa|}\mathcal{F}^{\dagger}\left(y_{\kappa}^{j}\left(\xi\left(\zeta_{N}^{\left(n_{f}\right)}\right)\right),y_{\kappa}^{j}\left(\xi\left(\zeta_{N}^{\left(f\right)}\right)\right),n\left(\zeta_{N}^{\left(f\right)}\right)\right)\nonumber \\
 &  & -\sum_{f=1}^{n_{f}-1}\sum_{l=1}^{N-1}\frac{\Delta t\nu_{f,l}^{\partial}}{\theta_{n_{f},N}|\kappa|}\mathcal{F}^{\dagger}\left(y_{\kappa}^{j}\left(\xi\left(\zeta_{N}^{\left(n_{f}\right)}\right)\right),y_{\kappa}^{j}\left(\xi\left(\zeta_{l}^{\left(n_{f}\right)}\right)\right),n\left(\zeta_{l}^{\left(f\right)}\right)\right).\label{eq:C-definition}
\end{eqnarray}
The above expansion relies on the conservation property of the numerical
flux:
\[
\mathcal{F}^{\dagger}\left(y_{1},y_{2},n\right)=-\mathcal{F}^{\dagger}\left(y_{2},y_{1},-n\right).
\]
$A_{f,l}$ in Equation~(\ref{eq:A_f_l-definition}) takes the form
of the three-point system~(\ref{eq:three-point-system-2D}). Invoking
the identity 
\begin{equation}
y_{\kappa}^{j}\left(\xi\left(\zeta_{l}^{\left(n_{f}\right)}\right)\right)=\sum_{f=1}^{n_{f}}\frac{\left|\partial\kappa^{(f)}\right|}{|\partial\kappa|}y_{\kappa}^{j}\left(\xi\left(\zeta_{l}^{\left(n_{f}\right)}\right)\right),\label{eq:identity-for-B_l}
\end{equation}
$B_{l}$ can be rewritten as
\begin{eqnarray}
B_{l} & = & \frac{\left|\partial\kappa^{(n_{f})}\right|}{|\partial\kappa|}\Biggl\{ y_{\kappa}^{j}\left(\xi\left(\zeta_{l}^{\left(n_{f}\right)}\right)\right)-\frac{\Delta t\nu_{n_{f},l}^{\partial}|\partial\kappa|}{\theta_{n_{f},l}|\kappa|\left|\partial\kappa^{(n_{f})}\right|}\left[\mathcal{F}^{\dagger}\left(y_{\kappa}^{j}\left(\xi\left(\zeta_{l}^{\left(n_{f}\right)}\right)\right),y_{\kappa^{\left(n_{f}\right)}}^{j}\left(\xi\left(\zeta_{l}^{\left(n_{f}\right)}\right)\right),n\left(\zeta_{l}^{\left(n_{f}\right)}\right)\right)\right.\nonumber \\
 &  & +\mathcal{F}^{\dagger}\left(y_{\kappa}^{j}\left(\xi\left(\zeta_{l}^{\left(n_{f}\right)}\right)\right),y_{\kappa}^{j}\left(\xi\left(\zeta_{N}^{\left(n_{f}\right)}\right)\right),-n\left(\zeta_{l}^{\left(n_{f}\right)}\right)\right)\biggr]\Biggr\}\nonumber \\
 &  & +\sum_{f=1}^{n_{f}-1}\frac{\left|\partial\kappa^{(f)}\right|}{|\partial\kappa|}\Biggl\{ y_{\kappa}^{j}\left(\xi\left(\zeta_{l}^{\left(n_{f}\right)}\right)\right)-\frac{\Delta t\nu_{f,l}^{\partial}|\partial\kappa|}{\theta_{n_{f},l}|\kappa|\left|\partial\kappa^{(f)}\right|}\left[\mathcal{F}^{\dagger}\left(y_{\kappa}^{j}\left(\xi\left(\zeta_{l}^{\left(n_{f}\right)}\right)\right),y_{\kappa}^{j}\left(\xi\left(\zeta_{l}^{\left(f\right)}\right)\right),n\left(\zeta_{l}^{\left(f\right)}\right)\right)\right.\nonumber \\
 &  & \left.+\mathcal{F}^{\dagger}\left(y_{\kappa}^{j}\left(\xi\left(\zeta_{l}^{\left(n_{f}\right)}\right)\right),y_{\kappa}^{j}\left(\xi\left(\zeta_{N}^{\left(n_{f}\right)}\right)\right),-n\left(\zeta_{l}^{\left(f\right)}\right)\right)\right]\Biggr\},\label{eq:B_l-definition-2}
\end{eqnarray}
for $l=1,\ldots,N-1$. The RHS of Equation~(\ref{eq:B_l-definition-2})
is a convex combination of three-point systems and, if some $\nu_{f,l}^{\partial}$
are zero, the pointwise value $y_{\kappa}^{j}\left(\xi\left(\zeta_{l}^{\left(n_{f}\right)}\right)\right)$.
To analyze $C$, we invoke the following identities:
\[
y_{\kappa}^{j}\left(\xi\left(\zeta_{N}^{\left(n_{f}\right)}\right)\right)=\sum_{f=1}^{n_{f}}\sum_{l=1}^{N}\frac{\nu_{f,l}^{\partial}}{|\partial\kappa|}y_{\kappa}^{j}\left(\xi\left(\zeta_{N}^{\left(n_{f}\right)}\right)\right)
\]
and
\begin{eqnarray*}
\sum_{f=1}^{n_{f}}\sum_{l=1}^{N}\nu_{f,l}^{\partial}\mathcal{F}^{\dagger}\left(y_{\kappa}^{j}\left(\xi\left(\zeta_{N}^{\left(n_{f}\right)}\right)\right),y_{\kappa}^{j}\left(\xi\left(\zeta_{N}^{\left(n_{f}\right)}\right)\right),n\left(\zeta_{l}^{\left(f\right)}\right)\right) & = & \sum_{f=1}^{n_{f}}\sum_{l=1}^{n_{q,f}^{\partial}}\nu_{f,l}^{\partial}\mathcal{F}\left(y_{\kappa}^{j}\left(\xi\left(\zeta_{N}^{\left(n_{f}\right)}\right)\right)\right)\cdot n\left(\zeta_{l}^{\left(f\right)}\right)\\
 & = & \int_{\partial\kappa}\mathcal{F}^{\dagger}\left(y_{\kappa}^{j}\left(\xi\left(\zeta_{N}^{\left(n_{f}\right)}\right)\right)\right)\cdot nds\\
 & = & \int_{\kappa}\nabla\cdot\mathcal{F}^{\dagger}\left(y_{\kappa}^{j}\left(\xi\left(\zeta_{N}^{\left(n_{f}\right)}\right)\right)\right)dx\\
 & = & 0,
\end{eqnarray*}
where the first line is due to the consistency property of the numerical
flux and the second line assumes sufficient accuracy of the surface
quadrature rule. $C$ in Equation~(\ref{eq:C-definition}) can then
be rewritten as
\begin{eqnarray*}
C & = & \frac{\nu_{n_{f},N}^{\partial}}{|\partial\kappa|}\Biggl\{ y_{\kappa}^{j}\left(\xi\left(\zeta_{N}^{\left(n_{f}\right)}\right)\right)-\frac{\Delta t|\partial\kappa|}{\theta_{n_{f},N}|\kappa|}\left[\mathcal{F}^{\dagger}\left(y_{\kappa}^{j}\left(\xi\left(\zeta_{N}^{\left(n_{f}\right)}\right)\right),y_{\kappa^{\left(n_{f}\right)}}^{j}\left(\xi\left(\zeta_{N}^{\left(n_{f}\right)}\right)\right),n\left(\zeta_{N}^{\left(n_{f}\right)}\right)\right)\right.\\
 &  & +\mathcal{F}^{\dagger}\left(y_{\kappa}^{j}\left(\xi\left(\zeta_{N}^{\left(n_{f}\right)}\right)\right),y_{\kappa}^{j}\left(\xi\left(\zeta_{N}^{\left(n_{f}\right)}\right)\right),-n\left(\zeta_{N}^{\left(n_{f}\right)}\right)\right)\biggr]\Biggr\}\\
 &  & +\sum_{l=1}^{N-1}\frac{\nu_{n_{f},l}^{\partial}}{|\partial\kappa|}\Biggl\{ y_{\kappa}^{j}\left(\xi\left(\zeta_{N}^{\left(n_{f}\right)}\right)\right)-\frac{\Delta t|\partial\kappa|}{\theta_{n_{f},N}|\kappa|}\left[\mathcal{F}^{\dagger}\left(y_{\kappa}^{j}\left(\xi\left(\zeta_{N}^{\left(n_{f}\right)}\right)\right),y_{\kappa}^{j}\left(\xi\left(\zeta_{l}^{\left(n_{f}\right)}\right)\right),n\left(\zeta_{l}^{\left(n_{f}\right)}\right)\right)\right.\\
 &  & \left.+\mathcal{F}^{\dagger}\left(y_{\kappa}^{j}\left(\xi\left(\zeta_{N}^{\left(n_{f}\right)}\right)\right),y_{\kappa}^{j}\left(\xi\left(\zeta_{N}^{\left(n_{f}\right)}\right)\right),-n\left(\zeta_{l}^{\left(n_{f}\right)}\right)\right)\right]\Biggr\}\\
 &  & +\sum_{f=1}^{n_{f}-1}\frac{\nu_{f,N}^{\partial}}{|\partial\kappa|}\Biggl\{ y_{\kappa}^{j}\left(\xi\left(\zeta_{N}^{\left(n_{f}\right)}\right)\right)-\frac{\Delta t|\partial\kappa|}{\theta_{n_{f},N}|\kappa|}\left[\mathcal{F}^{\dagger}\left(y_{\kappa}^{j}\left(\xi\left(\zeta_{N}^{\left(n_{f}\right)}\right)\right),y_{\kappa}^{j}\left(\xi\left(\zeta_{N}^{\left(f\right)}\right)\right),n\left(\zeta_{N}^{\left(f\right)}\right)\right)\right.\\
 &  & \left.+\mathcal{F}^{\dagger}\left(y_{\kappa}^{j}\left(\xi\left(\zeta_{N}^{\left(n_{f}\right)}\right)\right),y_{\kappa}^{j}\left(\xi\left(\zeta_{N}^{\left(n_{f}\right)}\right)\right),-n\left(\zeta_{N}^{\left(f\right)}\right)\right)\right]\Biggr\}\\
 &  & +\sum_{f=1}^{n_{f}-1}\sum_{l=1}^{N-1}\frac{\nu_{f,l}^{\partial}}{|\partial\kappa|}\Biggl\{ y_{\kappa}^{j}\left(\xi\left(\zeta_{N}^{\left(n_{f}\right)}\right)\right)-\frac{\Delta t|\partial\kappa|}{\theta_{n_{f},N}|\kappa|}\left[\mathcal{F}^{\dagger}\left(y_{\kappa}^{j}\left(\xi\left(\zeta_{N}^{\left(n_{f}\right)}\right)\right),y_{\kappa}^{j}\left(\xi\left(\zeta_{l}^{\left(n_{f}\right)}\right)\right),n\left(\xi\left(\zeta_{l}^{\left(f\right)}\right)\right)\right)\right.\\
 &  & \left.+\mathcal{F}^{\dagger}\left(y_{\kappa}^{j}\left(\xi\left(\zeta_{N}^{\left(n_{f}\right)}\right)\right),y_{\kappa}^{j}\left(\xi\left(\zeta_{N}^{\left(n_{f}\right)}\right)\right),-n\left(\xi\left(\zeta_{l}^{\left(f\right)}\right)\right)\right)\right],
\end{eqnarray*}
which is a convex combination of three-point systems (regardless of
whether some $\nu_{f,l}^{\partial}$ are zero). As such, $\overline{y}_{\kappa}^{j+1}$
is a convex combination of the following components:
\begin{itemize}
\item Pointwise values, $y_{\kappa}^{j}(x_{v})$, in $\kappa$
\item Three-point systems involving pointwise values, $y_{\kappa}^{j}\left(\xi\left(\zeta_{l}^{\left(f\right)}\right)\right)$
and $y_{\kappa^{(f)}}^{j}\left(\xi\left(\zeta_{l}^{\left(f\right)}\right)\right)$,
along $\partial\kappa$
\item If some $\nu_{f,l}^{\partial}$ are zero, pointwise values, $y_{\kappa}^{j}\left(\xi\left(\zeta_{l}^{\left(n_{f}\right)}\right)\right)$,
along $\partial\kappa$
\end{itemize}
The above results lead to the following theorem, where we use $y_{\kappa}^{-}$
to denote the exterior state along $\partial\kappa$.
\begin{thm}
\label{thm:CFL-condition-2D}If $y_{\kappa}^{j}(x)\in\mathcal{G}_{\sigma},\;\forall x\in\mathcal{D_{\kappa}}$,
and $y_{\kappa}^{-,j}\in\mathcal{G}_{\sigma},\;\forall x\in\partial\mathcal{D}_{\kappa}$,
with
\begin{equation}
\sigma\leq\min\left\{ \min\left\{ s\left(y_{\kappa}^{j}(x)\right)\vert x\in\mathcal{D_{\kappa}}\right\} ,\min\left\{ s\left(y_{\kappa}^{-,j}(x)\right)\vert x\in\mathcal{\partial D_{\kappa}}\right\} \right\} ,\label{eq:sigma-definition-1}
\end{equation}
 then $\overline{y}_{\kappa}^{j+1}$ in Equation~(\ref{eq:fully-discrete-form-average-2D-unexpanded})
is also in $\mathcal{G}_{\sigma}$ under the constraint
\begin{align}
\frac{\Delta t\lambda}{|\kappa|} & \leq\frac{1}{2}\min\left\{ L_{A},L_{B},L_{C}\right\} ,\label{eq:CFL-condition-2D}\\
L_{A} & =\min\left\{ \left.\frac{\theta_{f,l}}{\nu_{f,l}^{\partial}}\right|f=1,\ldots,n_{f}-1,\;l=1,\ldots,n_{q,f}^{\partial}\right\} ,\nonumber \\
L_{B} & =\min\left\{ \left.\frac{\theta_{n_{f},l}}{\nu_{f,l}^{\partial}}\frac{\left|\partial\kappa^{(f)}\right|}{|\partial\kappa|}\right|,f=1,\ldots,n_{f},\;l=1,\ldots,\min\left\{ n_{q,f}^{\partial},N-1\right\} \right\} ,\nonumber \\
L_{C} & =\frac{\theta_{n_{f},N}}{|\partial\kappa|},\nonumber 
\end{align}
and the conditions
\begin{equation}
\begin{cases}
\theta_{v}\geq0, & v=1,\ldots,n_{q}\\
\theta_{f,l}>0, & f=1,\ldots,n_{f},\;l=1,\ldots,n_{q,f}^{\partial}.
\end{cases}\label{eq:theta-conditions-2D}
\end{equation}
\end{thm}

\begin{proof}
The proof follows similar logic to that of the one-dimensional version
in Part I~\citep[Theorem 1]{Chi22}. The constraint $\Delta t\lambda/|\kappa|\leq L_{A}/2$
ensures that $A_{f,l}$ (Equation~(\ref{eq:A_f_l-definition})) is
in $\mathcal{G}_{\sigma}$, the constraint $\Delta t\lambda/|\kappa|\leq L_{B}/2$
ensures that $B_{l}$ (Equation~(\ref{eq:B_l-definition-2})) is
in $\mathcal{G}_{\sigma}$, and the constraint $\Delta t\lambda/|\kappa|\leq L_{C}/2$
ensures that $C$ (Equation~(\ref{eq:C-definition})) is in $\mathcal{G}_{\sigma}$.
It follows from Equation~(\ref{eq:fully-discrete-form-average-2D})
that $\overline{y}_{\kappa}^{j+1}$ is in $\mathcal{G}_{\sigma}$.
\end{proof}
\begin{rem}
A direct result of Theorem~\ref{thm:CFL-condition-2D} is that
\[
s\left(\overline{y}_{\kappa}^{j+1}\right)\geq\min\left\{ \min\left\{ s\left(y_{\kappa}^{j}(x)\right)\vert x\in\mathcal{D_{\kappa}}\right\} ,\min\left\{ s\left(y_{\kappa}^{-,j}(x)\right)\vert x\in\mathcal{\partial D_{\kappa}}\right\} \right\} .
\]
\end{rem}

\begin{rem}
A simple linear-scaling limiter, described in Section~\ref{subsec:limiting-procedure},
is directly applied to enforce $y_{\kappa}^{j+1}(x)\in\mathcal{G}_{s_{b,\kappa}^{j+1}},\:\forall x\in\mathcal{D_{\kappa}}$,
where $s_{b}$ is a lower bound on the specific thermodynamic entropy.
Motivated by the minimum entropy principle~(\ref{eq:minimum-entropy-principle}),
$s_{b}$ is computed in this work in an element-local manner as
\begin{equation}
s_{b,\kappa}^{j+1}(y)=\min\left\{ s\left(y^{j}(x)\right)\left|x\in\bigcup_{f=1}^{n_{f}}\mathcal{D}_{\kappa^{(f)}}\bigcup\mathcal{D}_{\kappa}\right.\right\} ,\label{eq:local-entropy-bound}
\end{equation}
which is an approximation of the minimum entropy over $\kappa$ and
the neighboring elements. An alternative to the local entropy bound
in Equation~(\ref{eq:local-entropy-bound}) is the following global
entropy bound: 
\begin{equation}
s_{b}(y)=\min\left\{ s\left(y(x)\right)\vert x\in\bigcup_{\kappa\in\mathcal{T}}\mathcal{D_{\kappa}}\right\} .\label{eq:global-entropy-bound}
\end{equation}
It was demonstrated that the local entropy bound~(\ref{eq:local-entropy-bound})
can more effectively dampen nonlinear instabilities~\citep{Chi22},
particularly when the entropy varies signifcantly throughout the domain.
Additional information on these local and global entropy bounds can
be found in Part I~\citep{Chi22}. This completes the construction
of a positivity-preserving, entropy-bounded, high-order DG scheme.
\end{rem}

\begin{rem}
As discussed in Part I~\citep{Chi22}, in practice, we loosen some
of the above requirements. First, the maximum wave speed at a given
point is simply approximated as $\left|v\right|+c$, where $c$ is
the speed of sound. Simple algorithms to compute bounds on the wave
speeds exist for the monocomponent case~\citep{Gue16,Tor20}; extending
these to the multicomponent, thermally perfect case may indeed be
worthy of future investigation. A similar comment can be made for
most invariant-region-preserving numerical fluxes, which often require
wave-speed estimates. Second, we revise the definition of $\mathcal{G}_{\sigma}$
as
\begin{equation}
\mathcal{G}_{\sigma}=\left\{ y\mid\rho>0,\rho u^{*}>0,C_{1}\geq0,\ldots,C_{n_{s}}\geq0,\chi_{\sigma}\geq0\right\} ,
\end{equation}
where $\chi_{\sigma}=\rho s-\rho\sigma$ (introduced in~\citep{Jia18}),
which is concave, and the species concentrations are permitted to
be equal to zero, given that requiring only positive concentrations
would preclude the calculation of practical problems. However, it
is important to note that if $C_{i}=0$ for some $i$, the specific
thermodynamic entropy, $s$, becomes ill-defined and the entropy functions
$U=-\rho s$ and $U=-\rho f_{\epsilon}(s)$ lose convexity. Nevertheless,
by making use of $0\log0=0$~\citep[Chapter 6]{Gio99}, $\rho s$
remains well-defined. We did not face any issues associated with relaxing
these two requirements in this study.
\end{rem}

In the case that $\partial\mathcal{D}_{\kappa}\nsubseteq\left\{ x(\xi_{v}),v=1,\ldots,n_{q}\right\} $,
as described by Lv and Ihme~\citep{Lv15_2}, an optimization problem
can be solved for each element in a pre-processing step to maximize
the RHS of~(\ref{eq:CFL-condition-2D}). They also introduced another
simpler but more restrictive approach that reformulates the time-step-size
constraint as
\[
\frac{\Delta t\lambda}{\mathsf{L}}\leq\frac{1}{2}\mathsf{B},
\]
where $\mathsf{L}$ is a proposed length scale and $\mathsf{B}$ is
tabulated in~\citep{Lv15_2} for common element shapes and quadrature
rules. The constraint~(\ref{eq:CFL-condition-2D}) can be reformulated
in a similar manner, utilizing the values of $\mathsf{B}$ in~\citep{Lv15_2}.
\RevisionTextThree{These values of $\mathsf{B}$ naturally apply to the quadrature-free
approach~\citep{Atk96,Atk98} employed in this work since, as previously
mentioned, $\theta_{f,l}$ and $\nu_{f,l}^{\partial}$ can be treated
in the same manner between quadrature-free and quadrature-based approaches. } \RevisionTextThree{Here, instead of explicitly accounting for the constraint~(\ref{eq:CFL-condition-2D}),
we simply prescribe the time-step size according to a user-prescribed
CFL based on the acoustic time scale, which is often less restrictive
than the constraint~(\ref{eq:CFL-condition-2D}). }Note that the condition~(\ref{eq:CFL-condition-2D}) is sufficient
but not necessary for $\overline{y}_{\kappa}^{j+1}$ to be in $\mathcal{G}_{\sigma}$
. Therefore, as recommended by Zhang~\citep{Zha17}, larger time-step
sizes can be initially taken and, if $\overline{y}_{\kappa}^{j+1}\notin\mathcal{G}_{\sigma}$,
the time step can be restarted with smaller $\Delta t$. The condition~(\ref{eq:CFL-condition-2D})
excludes the possibility of infinite loops. \RevisionTextThree{In this study, this backtracking strategy is needed only during the
first time steps of the two-dimensional, $p=1$ and three-dimensional,
$p=2$ detonation-wave simulations in Sections~\ref{subsec:2D-detonation-wave}
and~\ref{subsec:3D-detonation-wave}.} \RevisionTextThree{We also neglect the artificial-viscosity formulation when determining
the time-step size. In principle, however, the artificial viscosity
may need to be bounded to maintain stability.}

If the faces are straight-sided (i.e., the surface Jacobian, $J_{\partial\kappa}^{(f)}$,
and the normal, $n$, are constant over any $\partial\kappa^{(f)}$,
for all $f$) and identical surface quadrature rules are employed
for each face, the time-step-size constraint~(\ref{eq:CFL-condition-2D})
can be simplified. Equation~(\ref{eq:fully-discrete-form-average-2D})
can instead be written as
\begin{eqnarray}
\overline{y}_{\kappa}^{j+1} & = & \overline{y}_{\kappa}^{j}-\sum_{f=1}^{n_{f}}\sum_{l=1}^{n_{q}^{\partial}}\frac{\Delta t\nu_{f,l}^{\partial}}{|\kappa|}\mathcal{F}^{\dagger}\left(y_{\kappa}^{j}\left(\xi\left(\zeta_{l}^{\left(f\right)}\right)\right),y_{\kappa^{(f)}}^{j}\left(\xi\left(\zeta_{l}^{\left(f\right)}\right)\right),n\left(\zeta_{l}^{\left(f\right)}\right)\right)\nonumber \\
 & = & \sum_{v=1}^{n_{q}}\theta_{v}y_{\kappa}^{j}\left(\xi_{v}\right)+\sum_{f=1}^{n_{f}}\sum_{l=1}^{n_{q}^{\partial}}\left[\theta_{f,l}y_{\kappa}^{j}\left(\xi\left(\zeta_{l}^{\left(f\right)}\right)\right)-\frac{\Delta t\nu_{f,l}^{\partial}}{|\kappa|}\mathcal{F}^{\dagger}\left(y_{\kappa}^{j}\left(\xi\left(\zeta_{l}^{\left(f\right)}\right)\right),y_{\kappa^{(f)}}^{j}\left(\xi\left(\zeta_{l}^{\left(f\right)}\right)\right),n\left(\zeta_{l}^{\left(f\right)}\right)\right)\right]\nonumber \\
 & = & \sum_{v=1}^{n_{q}}\theta_{v}y_{\kappa}^{j}\left(\xi_{v}\right)+\sum_{f=1}^{n_{f}-1}\sum_{l=1}^{n_{q}^{\partial}}\theta_{f,l}A_{f,l}+\sum_{l=1}^{n_{q}^{\partial}}\theta_{n_{f},l}B_{l},\label{eq:fully-discrete-form-average-2D-straight-sided}
\end{eqnarray}
where $\nu_{f,l}^{\partial}$ is now defined as
\[
\nu_{f,l}^{\partial}=\left|\partial\kappa^{(f)}\right|\widehat{w}_{l}^{\partial},
\]
with $\widehat{w}_{l}^{\partial}$ denoting normalized weights such
that $\sum_{l}\widehat{w}_{l}^{\partial}=1$. $A_{f,l}$ in Equation~(\ref{eq:fully-discrete-form-average-2D-straight-sided})
is still of the form~(\ref{eq:A_f_l-definition}) for $f=1,\ldots,n_{f}-1,\;l=1,\ldots,n_{q}^{\partial}$,
while $B_{l}$ is given by 
\begin{eqnarray*}
B_{l} & = & y_{\kappa}^{j}\left(\xi\left(\zeta_{l}^{\left(n_{f}\right)}\right)\right)\\
 &  & -\frac{\Delta t\nu_{n_{f},l}^{\partial}}{\theta_{n_{f},l}|\kappa|}\mathcal{F}^{\dagger}\left(y_{\kappa}^{j}\left(\xi\left(\zeta_{l}^{\left(n_{f}\right)}\right)\right),y_{\kappa^{\left(n_{f}\right)}}^{j}\left(\xi\left(\zeta_{l}^{\left(n_{f}\right)}\right)\right),n\left(\zeta_{l}^{\left(n_{f}\right)}\right)\right)\\
 &  & -\sum_{f=1}^{n_{f}-1}\frac{\Delta t\nu_{f,l}^{\partial}}{\theta_{n_{f},l}|\kappa|}\mathcal{F}^{\dagger}\left(y_{\kappa}^{j}\left(\xi\left(\zeta_{l}^{\left(n_{f}\right)}\right)\right),y_{\kappa}^{j}\left(\xi\left(\zeta_{l}^{\left(f\right)}\right)\right),n\left(\zeta_{l}^{\left(f\right)}\right)\right)
\end{eqnarray*}
for $l=1,\ldots,n_{q}^{\partial}$. With the identity
\begin{eqnarray*}
\sum_{f=1}^{n_{f}}\left|\partial\kappa^{(f)}\right|\mathcal{F}^{\dagger}\left(y_{\kappa}^{j}\left(\xi\left(\zeta_{l}^{\left(n_{f}\right)}\right)\right),y_{\kappa}^{j}\left(\xi\left(\zeta_{l}^{\left(n_{f}\right)}\right)\right),n\left(\zeta_{l}^{\left(f\right)}\right)\right) & = & \sum_{f=1}^{n_{f}}\left|\partial\kappa^{(f)}\right|\mathcal{F}\left(y_{\kappa}^{j}\left(\xi\left(\zeta_{l}^{\left(n_{f}\right)}\right)\right)\right)\cdot n\left(\zeta_{l}^{\left(f\right)}\right)\\
 & = & \int_{\partial\kappa}\mathcal{F}^{\dagger}\left(y_{\kappa}^{j}\left(\xi\left(\zeta_{l}^{\left(n_{f}\right)}\right)\right)\right)\cdot nds\\
 & = & \int_{\kappa}\nabla\cdot\mathcal{F}^{\dagger}\left(y_{\kappa}^{j}\left(\xi\left(\zeta_{l}^{\left(n_{f}\right)}\right)\right)\right)dx\\
 & = & 0,
\end{eqnarray*}
$B_{l}$ can be rewritten as
\begin{eqnarray*}
B_{l} & = & \frac{\left|\partial\kappa^{(n_{f})}\right|}{|\partial\kappa|}\Biggl\{ y_{\kappa}^{j}\left(\xi\left(\zeta_{l}^{\left(n_{f}\right)}\right)\right)-\frac{\Delta t\widehat{w}_{l}^{\partial}|\partial\kappa|}{\theta_{n_{f},l}|\kappa|}\left[\mathcal{F}^{\dagger}\left(y_{\kappa}^{j}\left(\xi\left(\zeta_{l}^{\left(n_{f}\right)}\right)\right),y_{\kappa^{\left(n_{f}\right)}}^{j}\left(\xi\left(\zeta_{l}^{\left(n_{f}\right)}\right)\right),n\left(\zeta_{l}^{\left(n_{f}\right)}\right)\right)\right.\\
 &  & +\mathcal{F}^{\dagger}\left(y_{\kappa}^{j}\left(\xi\left(\zeta_{l}^{\left(n_{f}\right)}\right)\right),y_{\kappa}^{j}\left(\xi\left(\zeta_{l}^{\left(n_{f}\right)}\right)\right),-n\left(\zeta_{l}^{\left(n_{f}\right)}\right)\right)\biggr]\Biggr\}\\
 &  & +\sum_{f=1}^{n_{f}-1}\frac{\left|\partial\kappa^{(f)}\right|}{|\partial\kappa|}\Biggl\{ y_{\kappa}^{j}\left(\xi\left(\zeta_{l}^{\left(n_{f}\right)}\right)\right)-\frac{\Delta t\widehat{w}_{l}^{\partial}|\partial\kappa|}{\theta_{n_{f},l}|\kappa|}\left[\mathcal{F}^{\dagger}\left(y_{\kappa}^{j}\left(\xi\left(\zeta_{l}^{\left(n_{f}\right)}\right)\right),y_{\kappa}^{j}\left(\xi\left(\zeta_{l}^{\left(f\right)}\right)\right),n\left(\zeta_{l}^{\left(f\right)}\right)\right)\right.\\
 &  & \left.+\mathcal{F}^{\dagger}\left(y_{\kappa}^{j}\left(\xi\left(\zeta_{l}^{\left(n_{f}\right)}\right)\right),y_{\kappa}^{j}\left(\xi\left(\zeta_{l}^{\left(n_{f}\right)}\right)\right),-n\left(\zeta_{l}^{\left(f\right)}\right)\right)\right]\Biggr\}.
\end{eqnarray*}
$B_{l}$ is therefore a convex combination of three-point systems.
The time-step-size constraint can then be modified as
\begin{align}
\frac{\Delta t\lambda}{|\kappa|} & \leq\frac{1}{2}\min\left\{ L_{A},L_{B}\right\} ,\label{eq:CFL-condition-2D-straight-sided}\\
L_{A} & =\min\left\{ \left.\frac{\theta_{f,l}}{\left|\partial\kappa^{(f)}\right|\widehat{w}_{l}^{\partial}}\right|f=1,\ldots,n_{f}-1,\;l=1,\ldots,n_{q}^{\partial}\right\} ,\nonumber \\
L_{B} & =\min\left\{ \left.\left.\frac{\theta_{n_{f},l}}{\left|\partial\kappa\right|\widehat{w}_{l}^{\partial}}\right|\right|,f=1,\ldots,n_{f},\;l=1,\ldots,n_{q}^{\partial}\right\} ,\nonumber 
\end{align}
which is similar to that in~\citep{Jia18}.

\subsubsection{Limiting procedure\label{subsec:limiting-procedure}}

The limiting procedure to enforce $y_{\kappa}^{j+1}(x)\in\mathcal{G}_{s_{b}},\:\forall x\in\mathcal{D_{\kappa}}$,
is the same as in Part I. For completeness, we summarize it here.
The $j$+1 superscript and $\kappa$ subscript are dropped for brevity.
$\overline{y}_{\kappa}^{j}(x)$ is assumed to be in $\mathcal{G}_{s_{b}}$.
The limiting operator is of the same form as in \citep{Wan12}, \citep{Zha17},
\citep{Jia18}, \citep{Wu21_2}, and related papers.
\begin{enumerate}
\item Positive density: if $\rho(x)>\epsilon,\:\forall x\in\mathcal{D}_{\kappa}$,
where $\epsilon>0$ is a small number (e.g., $\epsilon=10^{-10}$),
then set $C_{i}^{(1)}=C_{i}=\sum_{j=1}^{n_{b}}C_{i}(x_{j})\phi_{j},i=1,\ldots,n_{s}$;
otherwise, compute
\[
C_{i}^{(1)}=\overline{C}_{i}+\theta^{(1)}\left(C_{i}-\overline{C}_{i}\right),\quad\theta^{(1)}=\frac{\rho(\overline{y})-\epsilon}{\rho(\overline{y})-\underset{x\in\mathcal{D}}{\min}\rho(y(x))}.
\]
for $i=1,\ldots,n_{s}.$
\item Nonnegative concentrations: if $C_{i}^{(1)}(x)\geq0,\:\forall x\in\mathcal{D}_{\kappa}$,
then set $C_{i}^{(2)}=C_{i}^{(1)},i=1,\ldots,n_{s}$; otherwise, compute
\[
C_{i}^{(2)}=\overline{C}_{i}+\theta^{(2)}\left(C_{i}^{(1)}-\overline{C}_{i}\right),\quad\theta^{(2)}=\frac{\overline{C}_{i}}{\overline{C}_{i}-\underset{x\in\mathcal{D}}{\min}C_{i}^{(1)}(x)}.
\]
Let $y^{(2)}=\left(\rho v_{1},\ldots,\rho v_{d},\rho e_{t},C_{1}^{(2)},\ldots,C_{n_{s}}^{(2)}\right)$. 
\item Positive temperature: if $\rho u^{*}\left(y^{(2)}(x)\right)>\epsilon,\:\forall x\in\mathcal{D}_{\kappa}$,
then set $y^{(3)}=y^{(2)}$; otherwise, compute
\[
y^{(3)}=\overline{y}+\theta^{(3)}\left(y^{(2)}-\overline{y}\right),\quad\theta^{(3)}=\frac{\rho u^{*}(\overline{y})-\epsilon}{\rho u^{*}(\overline{y})-\underset{x\in\mathcal{D}}{\min}\rho u^{*}(y^{(2)}(x))}.
\]
The ``positivity-preserving limiter'' refers to the limiting procedure
up to this point. 
\item Entropy constraint: if $\chi\left(y^{(3)}(x)\right)\geq0,\:\forall x\in\mathcal{D}_{\kappa}$,
then set $y^{(4)}=y^{(3)}$; otherwise, compute
\[
y^{(4)}=\overline{y}+\theta^{(4)}\left(y^{(3)}-\overline{y}\right),\quad\theta^{(4)}=\frac{\chi(\overline{y})}{\chi(\overline{y})-\underset{x\in\mathcal{D}}{\min}\chi(y^{(3)}(x))}.
\]
The ``entropy limiter'' refers to the limiting procedure up to this
point.
\end{enumerate}
$y^{(4)}$ then replaces $y$ as the solution. This limiting procedure
is applied at the end of every RK stage. It is conservative and in
general preserves the formal order of accuracy for smooth solutions~\citep{Zha10,Zha17,Zha12_2,Lv15_2,Jia18}.
\RevisionTextThree{Note that Step 1 may seem superfluous given Step 2; however, consider
a two-species flow where both species concentrations are negative
at the same point in $\mathcal{D}_{\kappa}$. If Step 1 is neglected,
Step 2 would limit the concentrations at that point to zero, resulting
in zero density.}

\subsubsection{Modified flux interpolation\label{subsec:modified-flux-interpolation-2D}}

We now discuss how to account for over-integration with the modified
flux interpolation in Equation~(\ref{eq:modified-flux-projection}).
The scheme satisfied by the element averages becomes
\begin{eqnarray}
\overline{y}_{\kappa}^{j+1} & = & \overline{y}_{\kappa}^{j}-\sum_{f=1}^{n_{f}}\sum_{l=1}^{n_{q,f}^{\partial}}\frac{\Delta t\nu_{f,l}^{\partial}}{|\kappa|}\mathcal{F}^{\dagger}\left(\widetilde{y}_{\kappa}^{j}\left(\xi\left(\zeta_{l}^{\left(f\right)}\right)\right),\widetilde{y}_{\kappa^{(k)}}^{j}\left(\xi\left(\zeta_{l}^{\left(f\right)}\right)\right),n\left(\zeta_{l}^{\left(f\right)}\right)\right)\nonumber \\
 & = & \sum_{v=1}^{n_{q}}\theta_{v}y_{\kappa}^{j}\left(\xi_{v}\right)+\sum_{f=1}^{n_{f}}\sum_{l=1}^{n_{q,f}^{\partial}}\left[\theta_{f,l}y_{\kappa}^{j}\left(\xi\left(\zeta_{l}^{\left(f\right)}\right)\right)-\frac{\Delta t\nu_{f,l}^{\partial}}{|\kappa|}\mathcal{F}^{\dagger}\left(\widetilde{y}_{\kappa}^{j}\left(\xi\left(\zeta_{l}^{\left(f\right)}\right)\right),\widetilde{y}_{\kappa^{(f)}}^{j}\left(\xi\left(\zeta_{l}^{\left(f\right)}\right)\right),n\left(\zeta_{l}^{\left(f\right)}\right)\right)\right]\nonumber \\
 & = & \sum_{v=1}^{n_{q}}\theta_{v}y_{\kappa}^{j}\left(\xi_{v}\right)+\sum_{f=1}^{n_{f}-1}\sum_{l=1}^{n_{q,f}^{\partial}}\theta_{f,l}\widetilde{A}_{f,l}+\sum_{l=1}^{N-1}\theta_{n_{f},l}\widetilde{B}_{l}+\theta_{n_{f},N}\widetilde{C}.\label{eq:fully-discrete-form-average-2D-modified}
\end{eqnarray}
Following the same steps as in Section~\ref{subsec:entropy-bounded-high-order-DG-2D},
$\widetilde{A}_{f,l}$, $\widetilde{B}_{l}$, and $\widetilde{C}$
can be written as
\begin{eqnarray}
\widetilde{A}_{f,l} & = & y_{\kappa}^{j}\left(\xi\left(\zeta_{l}^{(f)}\right)\right)-\frac{\Delta t\nu_{f,l}^{\partial}}{\theta_{f,l}|\kappa|}\left[\mathcal{F}^{\dagger}\left(\widetilde{y}_{\kappa}^{j}\left(\xi\left(\zeta_{l}^{(f)}\right)\right),\widetilde{y}_{\kappa^{(f)}}^{j}\left(\xi\left(\zeta_{l}^{(f)}\right)\right),n\left(\zeta_{l}^{(f)}\right)\right)\right.\nonumber \\
 &  & \left.+\mathcal{F}^{\dagger}\left(\widetilde{y}_{\kappa}^{j}\left(\xi\left(\zeta_{l}^{\left(f\right)}\right)\right),\widetilde{y}_{\kappa}^{j}\left(\xi\left(\zeta_{l}^{\left(n_{f}\right)}\right)\right),-n\left(\zeta_{l}^{(f)}\right)\right)\right],\label{eq:A_f_l-definition-tilde}
\end{eqnarray}
\begin{eqnarray}
\widetilde{B}_{l} & = & \frac{\left|\partial\kappa^{(n_{f})}\right|}{|\partial\kappa|}\Biggl\{ y_{\kappa}^{j}\left(\xi\left(\zeta_{l}^{\left(n_{f}\right)}\right)\right)-\frac{\Delta t\nu_{n_{f},l}^{\partial}|\partial\kappa|}{\theta_{n_{f},l}|\kappa|\left|\partial\kappa^{(n_{f})}\right|}\left[\mathcal{F}^{\dagger}\left(\widetilde{y}_{\kappa}^{j}\left(\xi\left(\zeta_{l}^{\left(n_{f}\right)}\right)\right),\widetilde{y}_{\kappa^{\left(n_{f}\right)}}^{j}\left(\xi\left(\zeta_{l}^{\left(n_{f}\right)}\right)\right),n\left(\zeta_{l}^{\left(n_{f}\right)}\right)\right)\right.\nonumber \\
 &  & +\mathcal{F}^{\dagger}\left(\widetilde{y}_{\kappa}^{j}\left(\xi\left(\zeta_{l}^{\left(n_{f}\right)}\right)\right),\widetilde{y}_{\kappa}^{j}\left(\xi\left(\zeta_{N}^{\left(n_{f}\right)}\right)\right),-n\left(\zeta_{l}^{\left(n_{f}\right)}\right)\right)\biggr]\Biggr\}\nonumber \\
 &  & +\sum_{f=1}^{n_{f}-1}\frac{\left|\partial\kappa^{(f)}\right|}{|\partial\kappa|}\Biggl\{ y_{\kappa}^{j}\left(\xi\left(\zeta_{l}^{\left(n_{f}\right)}\right)\right)-\frac{\Delta t\nu_{f,l}^{\partial}|\partial\kappa|}{\theta_{n_{f},l}|\kappa|\left|\partial\kappa^{(f)}\right|}\left[\mathcal{F}^{\dagger}\left(\widetilde{y}_{\kappa}^{j}\left(\xi\left(\zeta_{l}^{\left(n_{f}\right)}\right)\right),\widetilde{y}_{\kappa}^{j}\left(\xi\left(\zeta_{l}^{\left(f\right)}\right)\right),n\left(\zeta_{l}^{\left(f\right)}\right)\right)\right.\nonumber \\
 &  & \left.+\mathcal{F}^{\dagger}\left(\widetilde{y}_{\kappa}^{j}\left(\xi\left(\zeta_{l}^{\left(n_{f}\right)}\right)\right),\widetilde{y}_{\kappa}^{j}\left(\xi\left(\zeta_{N}^{\left(n_{f}\right)}\right)\right),-n\left(\zeta_{l}^{\left(f\right)}\right)\right)\right]\Biggr\},\label{eq:B_l-definition-2-tilde}
\end{eqnarray}
and
\begin{eqnarray*}
\widetilde{C} & = & \frac{\nu_{n_{f},N}^{\partial}}{|\partial\kappa|}\Biggl\{ y_{\kappa}^{j}\left(\xi\left(\zeta_{N}^{\left(n_{f}\right)}\right)\right)-\frac{\Delta t|\partial\kappa|}{\theta_{n_{f},N}|\kappa|}\left[\mathcal{F}^{\dagger}\left(\widetilde{y}_{\kappa}^{j}\left(\xi\left(\zeta_{N}^{\left(n_{f}\right)}\right)\right),\widetilde{y}_{\kappa^{\left(n_{f}\right)}}^{j}\left(\xi\left(\zeta_{N}^{\left(n_{f}\right)}\right)\right),n\left(\zeta_{N}^{\left(n_{f}\right)}\right)\right)\right.\\
 &  & +\mathcal{F}^{\dagger}\left(\widetilde{y}_{\kappa}^{j}\left(\xi\left(\zeta_{N}^{\left(n_{f}\right)}\right)\right),\widetilde{y}_{\kappa}^{j}\left(\xi\left(\zeta_{N}^{\left(n_{f}\right)}\right)\right),-n\left(\zeta_{N}^{\left(n_{f}\right)}\right)\right)\biggr]\Biggr\}\\
 &  & +\sum_{l=1}^{N-1}\frac{\nu_{n_{f},l}^{\partial}}{|\partial\kappa|}\Biggl\{ y_{\kappa}^{j}\left(\xi\left(\zeta_{N}^{\left(n_{f}\right)}\right)\right)-\frac{\Delta t|\partial\kappa|}{\theta_{n_{f},N}|\kappa|}\left[\mathcal{F}^{\dagger}\left(\widetilde{y}_{\kappa}^{j}\left(\xi\left(\zeta_{N}^{\left(n_{f}\right)}\right)\right),\widetilde{y}_{\kappa}^{j}\left(\xi\left(\zeta_{l}^{\left(n_{f}\right)}\right)\right),n\left(\zeta_{l}^{\left(n_{f}\right)}\right)\right)\right.\\
 &  & \left.+\mathcal{F}^{\dagger}\left(\widetilde{y}_{\kappa}^{j}\left(\xi\left(\zeta_{N}^{\left(n_{f}\right)}\right)\right),\widetilde{y}_{\kappa}^{j}\left(\xi\left(\zeta_{N}^{\left(n_{f}\right)}\right)\right),-n\left(\zeta_{l}^{\left(n_{f}\right)}\right)\right)\right]\Biggr\}\\
 &  & +\sum_{f=1}^{n_{f}-1}\frac{\nu_{f,N}^{\partial}}{|\partial\kappa|}\Biggl\{ y_{\kappa}^{j}\left(\xi\left(\zeta_{N}^{\left(n_{f}\right)}\right)\right)-\frac{\Delta t|\partial\kappa|}{\theta_{n_{f},N}|\kappa|}\left[\mathcal{F}^{\dagger}\left(\widetilde{y}_{\kappa}^{j}\left(\xi\left(\zeta_{N}^{\left(n_{f}\right)}\right)\right),\widetilde{y}_{\kappa}^{j}\left(\xi\left(\zeta_{N}^{\left(f\right)}\right)\right),n\left(\zeta_{N}^{\left(f\right)}\right)\right)\right.\\
 &  & \left.+\mathcal{F}^{\dagger}\left(\widetilde{y}_{\kappa}^{j}\left(\xi\left(\zeta_{N}^{\left(n_{f}\right)}\right)\right),\widetilde{y}_{\kappa}^{j}\left(\xi\left(\zeta_{N}^{\left(n_{f}\right)}\right)\right),-n\left(\zeta_{N}^{\left(f\right)}\right)\right)\right]\Biggr\}\\
 &  & +\sum_{f=1}^{n_{f}-1}\sum_{l=1}^{N-1}\frac{\nu_{f,l}^{\partial}}{|\partial\kappa|}\Biggl\{ y_{\kappa}^{j}\left(\xi\left(\zeta_{N}^{\left(n_{f}\right)}\right)\right)-\frac{\Delta t|\partial\kappa|}{\theta_{n_{f},N}|\kappa|}\left[\mathcal{F}^{\dagger}\left(\widetilde{y}_{\kappa}^{j}\left(\xi\left(\zeta_{N}^{\left(n_{f}\right)}\right)\right),\widetilde{y}_{\kappa}^{j}\left(\xi\left(\zeta_{l}^{\left(n_{f}\right)}\right)\right),n\left(\xi\left(\zeta_{l}^{\left(f\right)}\right)\right)\right)\right.\\
 &  & \left.+\mathcal{F}^{\dagger}\left(\widetilde{y}_{\kappa}^{j}\left(\xi\left(\zeta_{N}^{\left(n_{f}\right)}\right)\right),\widetilde{y}_{\kappa}^{j}\left(\xi\left(\zeta_{N}^{\left(n_{f}\right)}\right)\right),-n\left(\xi\left(\zeta_{l}^{\left(f\right)}\right)\right)\right)\right].
\end{eqnarray*}
Unfortunately, the corresponding three-point systems are not necessarily
of the type~(\ref{eq:three-point-system-2D}) since in general, $y_{\kappa}^{j}\left(\xi\left(\zeta_{l}^{\left(f\right)}\right)\right)\neq\widetilde{y}_{\kappa}^{j}\left(\xi\left(\zeta_{l}^{\left(f\right)}\right)\right)$.
The incompatibility is a result of expressing $\overline{y}_{\kappa}$
as a convex combination of pointwise values of $y_{\kappa}(x)$ (as
opposed to $\widetilde{y}_{\kappa}(x)$). Since the element average
of $\widetilde{y}_{\kappa}$, denoted $\overline{\widetilde{y}}_{\kappa}$,
is not necessarily equal to $\overline{y}_{\kappa}$, $\overline{y}_{\kappa}$
cannot be directly written as a convex combination of pointwise values
of $\widetilde{y}_{\kappa}(x)$. In the one-dimensional case, if the
nodal set includes the endpoints, then this issue is circumvented
since $y_{\kappa}^{j}(x_{L})=\widetilde{y}_{\kappa}^{j}(x_{L})$ and
$y_{\kappa}^{j}(x_{R})=\widetilde{y}_{\kappa}^{j}(x_{R})$. However,
the multidimensional case is more complicated. One simple approach,
assuming that the nodal set includes surface points that can be used
for integration, is as follows:
\begin{itemize}
\item Compute $y_{\kappa}^{j+1}$ using over-integration with the modified
flux interpolation~(\ref{eq:modified-flux-projection}).
\item If $\overline{y}_{\kappa}^{j+1}\in\mathcal{G}_{s_{b}}$, then proceed
to the next time step. This will typically be true since in general,
$y_{\kappa}^{j}\left(x^{(f)}\left(\zeta_{l}\right)\right)\approx\widetilde{y}_{\kappa}^{j}\left(x^{(f)}\left(\zeta_{l}\right)\right)$
and the conditions laid out in Section~\ref{subsec:entropy-bounded-high-order-DG-2D}
are not necessary for $\overline{y}_{\kappa}^{j+1}$ to be in $\mathcal{G}_{s_{b}}$.
\item In the rare case that $\overline{y}_{\kappa}^{j+1}\notin\mathcal{G}_{s_{b}}$,
recompute $y_{\kappa}^{j+1}$ with integration points that are in
the nodal set, which effectively maintains pressure equilibrium. $\overline{y}_{\kappa}^{j+1}$
is then guaranteed to be in $\mathcal{G}_{s_{b}}$ since $y_{\kappa}=\widetilde{y}_{\kappa}$
at the solution nodes.
\end{itemize}
Note that this would only need to be done for the surface integrals;
since the second term in Equation~(\ref{eq:semi-discrete-form})
(i.e., the volumetric flux integral) does not factor into the scheme
satisfied by the element averages, over-integration with the modified
flux interpolation can be freely employed in said integral. Furthermore,
if $\overline{y}_{\kappa}^{j+1}$ satisfies the positivity property
but $s\left(\overline{y}_{\kappa}^{j+1}\right)<s_{b}$ (and therefore
$\overline{y}_{\kappa}^{j+1}\notin\mathcal{G}_{s_{b}}$), then it
may still be reasonable to proceed to the next time step, provided
that $s\left(y_{\kappa}^{j}(x)\right)\not\ll s_{b},\:\forall x\in\mathcal{D_{\kappa}}$
\RevisionTextThree{(e.g., $s\left(y_{\kappa}^{j}(x)\right)\not<0.95s_{b},\:\forall x\in\mathcal{D_{\kappa}}$)}.
\RevisionTextThree{In Appendix~\ref{sec:Effect-of-over-integration}, we investigate
the effect of over-integation on the frequency of limiter activation
in a one-dimensional test case involving the advection of a low-density
Gaussian wave.}

It should also be noted that using over-integration with the modified
flux interpolation~(\ref{eq:modified-flux-projection}), $\overline{y}_{\kappa}^{j+1}$
is guaranteed to at least satisfy the positivity property under a
different time-step-size constraint. To show this, let $\mathcal{G}$
denote the set
\[
\mathcal{G}=\left\{ y\mid C_{1}\geq0,\ldots,C_{n_{s}}\geq0,\rho>0,\rho u^{*}>0\right\} ,
\]
and, as a representative example, rewrite Equation~(\ref{eq:A_f_l-definition-tilde})
as 
\begin{eqnarray*}
\widetilde{A}_{f,l} & = & \acute{y}_{\kappa}^{j}\left(\xi\left(\zeta_{l}^{(f)}\right)\right)-\frac{\Delta t\nu_{f,l}^{\partial}}{\theta_{f,l}|\kappa|}\Delta\mathcal{F}\left(\xi\left(\zeta_{l}^{(f)}\right)\right),
\end{eqnarray*}
where 
\begin{align*}
\acute{y}_{\kappa}^{j}\left(\xi\left(\zeta_{l}^{(f)}\right)\right)= & y_{\kappa}^{j}\left(\xi\left(\zeta_{l}^{(f)}\right)\right)-\frac{\Delta t\nu_{f,l}^{\partial}}{\theta_{f,l}|\kappa|}\left[\mathcal{F}^{\dagger}\left(y_{\kappa}^{j}\left(\xi\left(\zeta_{l}^{(f)}\right)\right),y_{\kappa^{(f)}}^{j}\left(\xi\left(\zeta_{l}^{(f)}\right)\right),n\left(\zeta_{l}^{(f)}\right)\right)\right.\\
 & \left.+\mathcal{F}^{\dagger}\left(y_{\kappa}^{j}\left(\xi\left(\zeta_{l}^{\left(f\right)}\right)\right),y_{\kappa}^{j}\left(\xi\left(\zeta_{l}^{\left(n_{f}\right)}\right)\right),-n\left(\zeta_{l}^{(f)}\right)\right)\right]
\end{align*}
\begin{align*}
\Delta\mathcal{F}\left(\xi\left(\zeta_{l}^{(f)}\right)\right)= & \mathcal{F}^{\dagger}\left(\widetilde{y}_{\kappa}^{j}\left(\xi\left(\zeta_{l}^{(f)}\right)\right),\widetilde{y}_{\kappa^{(f)}}^{j}\left(\xi\left(\zeta_{l}^{(f)}\right)\right),n\left(\zeta_{l}^{(f)}\right)\right)\\
 & -\mathcal{F}^{\dagger}\left(y_{\kappa}^{j}\left(\xi\left(\zeta_{l}^{(f)}\right)\right),y_{\kappa^{(f)}}^{j}\left(\xi\left(\zeta_{l}^{(f)}\right)\right),n\left(\zeta_{l}^{(f)}\right)\right)\\
 & +\mathcal{F}^{\dagger}\left(\widetilde{y}_{\kappa}^{j}\left(\xi\left(\zeta_{l}^{\left(f\right)}\right)\right),\widetilde{y}_{\kappa}^{j}\left(\xi\left(\zeta_{l}^{\left(n_{f}\right)}\right)\right),-n\left(\zeta_{l}^{(f)}\right)\right)\\
 & -\mathcal{F}^{\dagger}\left(y_{\kappa}^{j}\left(\xi\left(\zeta_{l}^{\left(f\right)}\right)\right),y_{\kappa}^{j}\left(\xi\left(\zeta_{l}^{\left(n_{f}\right)}\right)\right),-n\left(\zeta_{l}^{(f)}\right)\right).
\end{align*}
$\acute{y}_{\kappa}^{j}\left(\xi\left(\zeta_{l}^{(f)}\right)\right)$
is expressed as a three-point system of the type~(\ref{eq:three-point-system-2D})
and is therefore in $\mathcal{G}$ under the constraint~(\ref{eq:CFL-condition-2D}).
Then, according to Lemma~(\ref{lem:alpha-constraints}) in Appendix~\ref{sec:supporting-lemma},
$\widetilde{A}_{f,l}$ is also in $\mathcal{G}$ if
\begin{equation}
\frac{\Delta t}{\left|\kappa\right|}<\frac{\theta_{f,l}}{\nu_{f,l}^{\partial}\alpha^{*}\left(\acute{y}_{\kappa}^{j}\left(\xi\left(\zeta_{l}^{\left(f\right)}\right)\right),\Delta\mathcal{F}\left(\xi\left(\zeta_{l}^{\left(f\right)}\right)\right)\right)},\label{eq:dt-alpha-constraint}
\end{equation}
where $\alpha^{*}$ is defined as in~(\ref{eq:alpha-constraint}).
Note that since the modified flux interpolation~(\ref{eq:modified-flux-projection})
does not alter species concentrations or momentum, $\alpha^{*}$ can
actually be simplified as in~\ref{rem:simpler-alpha-constraint}.
In general, since $y_{\kappa}^{j}\left(x^{(f)}\left(\zeta_{l}\right)\right)\approx\widetilde{y}_{\kappa}^{j}\left(x^{(f)}\left(\zeta_{l}\right)\right)$,
$\Delta\mathcal{F}$ and thus $\alpha^{*}$ are small, such that the
constraint~(\ref{eq:dt-alpha-constraint}) is not restrictive. The
same analysis can be performed for $\widetilde{B}_{l}$ and $C$ to
yield similar constraints on $\Delta t$. In practice, we do not find
it necessary to explicitly account for these additional constraints.

Though not pursued in this work, in Appendix~\ref{sec:alternative-approach-with-over-integration},
we present an alternative approach compatible with over-integration.
Said approach utilizes an additional auxiliary polynomial to guarantee
$\overline{y}_{\kappa}^{j+1}\in\mathcal{G}_{s_{b}}$ while employing
the modified flux interpolation~(\ref{eq:modified-flux-projection}).

\section{Reaction step}

In this section, we briefly discuss the reaction step, most of which
is independent of the number of spatial dimensions. Only the multidimensional
considerations are detailed here; the reader is referred to Part I~\citep{Chi22}
for more information.

The element-local, semi-discrete form of Equation~(\ref{eq:strang-splitting-2})
is written as
\begin{gather}
\int_{\kappa}\mathfrak{v}^{T}\frac{\partial y}{\partial t}dx-\int_{\kappa}\mathfrak{v}^{T}\mathcal{S}(y)dx=0,\;\forall\mathfrak{v}\in V_{h}^{p}.\label{eq:semi-discrete-form-ode}
\end{gather}
We approximate $\mathcal{S}\left(y\right)$ locally as a polynomial
in $V_{h}^{p}$ as
\[
\mathcal{S}_{\kappa}\approx\sum_{j=1}^{n_{b}}\mathcal{S}\left(y\left(x_{j}\right)\right)\phi_{j},
\]
giving the following spatially decoupled system of ODEs advanced at
the solution nodes from $t=t_{0}$ to $t=t_{f}$:
\[
\frac{d}{dt}y_{\kappa}\left(x_{j},t\right)-\mathcal{S}\left(y_{\kappa}\left(x_{j},t\right)\right)=0,\quad j=1,\ldots,n_{b}.
\]
Suppose it can be guaranteed that 
\begin{align}
y_{\kappa}\left(x_{j},t_{f}\right) & \in\mathcal{G}_{s\left(y_{\kappa}\left(x_{j},t_{0}\right)\right)},\quad j=1,\ldots,n_{b}.\label{eq:dgode-nodal-solution-in-G}
\end{align}
With $s_{b}$ now given by
\[
s_{b}=\min_{j=1,\ldots,n_{b}}s\left(y_{\kappa}\left(x_{j},t_{0}\right)\right),
\]
$\overline{y}_{\kappa}(t_{f})$ is in $\mathcal{G}_{s_{b}}$ under
the following two conditions:
\begin{itemize}
\item The nodal set corresponds to the quadrature points of a rule with
positive weights (e.g., Gauss-Lobatto nodes).
\item Said quadrature rule is sufficiently accurate to compute $\overline{y}_{\kappa}$.
\end{itemize}
With $\overline{y}_{\kappa}(t_{f})\in\mathcal{G}_{s_{b}}$, the limiting
procedure in Section~\ref{subsec:limiting-procedure} can then be
applied to enforce $y_{\kappa}\left(x,t_{f}\right)\in\mathcal{G}_{s_{b}},\:\forall x\in\mathcal{D}_{\kappa}$
(unless $\mathcal{D}_{\kappa}=\left\{ x_{j},j=1,\ldots,n_{b}\right\} $,
in which case the limiting procedure is superfluous). However, if
the geometric Jacobian is not constant, the quadrature rule may no
longer be sufficiently accurate to compute $\overline{y}_{\kappa}$\RevisionTextThree{; therefore, even if~(\ref{eq:dgode-nodal-solution-in-G}) is satisfied,
$\overline{y}_{\kappa}(t_{f})$ may not be in $\mathcal{G}_{s_{b}}$.
One approach to guarantee $\overline{y}_{\kappa}(t_{f})\in\mathcal{G}_{s_{b}}$} is to first expand $y_{\kappa}$ as
\[
y_{\kappa}=\sum_{i=1}^{n_{a}}y_{\kappa}(x_{i})\Psi_{i},
\]
where $\left\{ \Psi_{1},\ldots,\Psi_{n_{a}}\right\} $ is a basis
of $V_{h}^{\widehat{p}}$, with $n_{a}>n_{b}$ and $\widehat{p}>p$,
and then solve
\begin{gather}
\int_{\kappa}\mathfrak{v}^{T}\frac{\partial y}{\partial t}dx-\int_{\kappa}\mathfrak{v}^{T}\mathcal{S}(y)dx=0,\;\forall\mathfrak{v}\in V_{h}^{\widehat{p}}\label{eq:semi-discrete-form-ode-p-hat}
\end{gather}
for $y\in V_{h}^{\widehat{p}}$. Following the same procedure as above,
the resulting spatially decoupled system of ODEs is
\begin{equation}
\frac{d}{dt}y_{\kappa}\left(x_{i},t\right)-\mathcal{S}\left(y_{\kappa}\left(x_{i},t\right)\right)=0,\quad i=1,\ldots,n_{a}.\label{eq:spatially_decoupled_ODEs_p_hat}
\end{equation}
Finally, $L^{2}$ projection is applied to project $y$ from $V_{h}^{\widehat{p}}$
to $V_{h}^{p}$. Since $L^{2}$ projection is a conservative operation,
$\overline{y}_{\kappa}(t_{f})\in\mathcal{G}_{s_{b}}$, where $s_{b}$
is now given by
\[
s_{b}=\min_{i=1,\ldots,n_{a}}s\left(y_{\kappa}\left(x_{i},t_{0}\right)\right).
\]
The limiter can then be employed to enforce $y_{\kappa}\left(x,t_{f}\right)\in\mathcal{G}_{s_{b}},\:\forall x\in\mathcal{D}_{\kappa}$,
unless $\mathcal{D}_{\kappa}=\left\{ x_{i},i=1,\ldots,n_{a}\right\} $.
Note that in~\citep{Lv15} and~\citep{Ban20}, slightly modified
ODEs are similarly solved at a generic set of quadrature points.

Each system of ODEs is solved \RevisionTextThree{in a spatially decoupled manner}
using a DG discretization in time. \RevisionTextThree{$h$-adaptivity in time is employed to ensure that the state at a
given point is in $\mathcal{G}_{s_{b}}$, with a maximum of ten Newton
iterations used to solve the nonlinear system at each sub-time-step. }Since the remainder of the reaction step is identical to that in the
one-dimensional case, we refer the reader to Part I~\citep{Chi22}
for further details on the temporal discretization.

\section{Multidimensional results\label{sec:results-2D}}

First, we compute two-dimensional thermal-bubble advection to assess
the ability of the proposed formulation to maintain pressure equilibrium
on curved grids. Next, we present solutions to a two-dimensional moving
detonation on a series of increasingly refined meshes over a range
of polynomial orders. Finally, we present the solution to a three-dimensional,
large-scale, moving detonation wave in order to demonstrate the utility
of the proposed formulation. The SSPRK3 time integration scheme~\citep{Got01,Spi02}
is used in Section~\ref{subsec:thermal-bubble}, while the SSPRK2
scheme is used in Sections~\ref{subsec:2D-detonation-wave} and~\ref{subsec:3D-detonation-wave}.
All simulations are performed using a modified version of the JENRE\textregistered~Multiphysics
Framework~\citep{Cor18_SCITECH,Joh20_2} that incorporates the developments
and extensions described in this work.

\subsection{Two-dimensional thermal-bubble advection\label{subsec:thermal-bubble}}

\RevisionTextOne{This flow configuration is a two-dimensional version of the one-dimensional
thermal-bubble advection we computed in Part I~\citep{Chi22}. The
initial conditions are written as
\begin{eqnarray}
\left(v_{1},v_{2}\right) & = & \left(1,0\right)\textrm{ m/s},\nonumber \\
Y_{H_{2}} & = & \frac{1}{2}\left[1-\tanh\left(\sqrt{x_{1}^{2}+x_{2}^{2}}-10\right)\right],\nonumber \\
Y_{O_{2}} & = & 1-Y_{H_{2}},\label{eq:thermal-bubble}\\
T & = & 1200-900\tanh\left(\sqrt{x_{1}^{2}+x_{2}^{2}}-10\right)\textrm{ K},\nonumber \\
P & = & 1\textrm{ bar}.\nonumber 
\end{eqnarray}
The thermodynamic relations for the hydrogen-oxygen mechanism used
here are given by
\begin{align*}
\frac{W_{H_{2}}c_{p,H_{2}}\left(T\right)}{R^{0}} & =3.47-0.220\widehat{T}+0.577\widehat{T}^{2}-0.194\widehat{T}^{3}+0.0210\widehat{T}^{4},\\
\frac{W_{H_{2}}h_{H_{2}}\left(T\right)}{R^{0}} & =\int\frac{W_{H_{2}}c_{p,H_{2}}\left(\tau\right)d\tau}{R^{0}}-1028.7\text{ K},\quad\frac{W_{H_{2}}s_{H_{2}}^{o}}{R^{0}}=\int\frac{W_{H_{2}}c_{p,H_{2}}\left(\tau\right)d\tau}{R^{0}\tau}-4.00,\\
\frac{W_{O_{2}}c_{p,O_{2}}\left(T\right)}{R^{0}} & =3.09+1.77\widehat{T}-0.911\widehat{T}^{2}+0.243\widehat{T}^{3}-0.0242\widehat{T}^{4},\\
\frac{W_{O_{2}}h_{O_{2}}\left(T\right)}{R^{0}} & =\int\frac{W_{O_{2}}c_{p,O_{2}}\left(\tau\right)d\tau}{R^{0}}-992.9\text{ K},\quad\frac{W_{O_{2}}s_{O_{2}}^{o}}{R^{0}}=\int\frac{W_{O_{2}}c_{p,O_{2}}\left(\tau\right)d\tau}{R^{0}\tau}+6.57,
\end{align*}
where $\widehat{T}=T/T_{r}$, with $T_{r}=1000$ K, and $\tau$ is
a dummy variable for $T$. The computational domain is $\text{\ensuremath{\Omega}}=\left(-25,25\right)\mathrm{m}\times\left(-25,25\right)\mathrm{m}$.
Symmetry conditions are imposed along the top and bottom boundaries,
while the left and right boundaries are periodic. The time-step size
is prescribed according to $\mathrm{CFL}=0.4$ based on the acoustic
time scale. The purpose of this test case is to determine whether
the proposed formulation can maintain pressure equilibrium (in an
approximate sense) when using multidimensional curved grids. Here,
the curved elements are of quadratic geometric order. We first generate
a uniform quadrilateral grid with 100 cells in each direction. At
interior edges, the midpoint nodes are then slightly perturbed from
their initial positions. We compute $p=1$, $p=2$, and $p=3$ solutions
without artificial viscosity up to $t=50\text{ s}$, corresponding
to one period. }

\RevisionTextOne{Figure~\ref{fig:thermal_bubble_p1} presents the $p=1$ temperature,
pressure, and streamwise-velocity distributions. The initial temperature
field is also shown. The curved grid is superimposed on the pressure
and velocity fields. Slight underresolution can can be observed even
in the initial temperature distribution. The pressure and velocity
ranges given in Figure~\ref{fig:thermal_bubble_p1} correspond to
the actual respective minima and maxima in the corresponding solutions.
Despite visible errors in the temperature and small deviations pressure
and velocity equilibrium, the solution remains stable throughout the
simulation. In contrast, with standard overintegration, the solver
fails before $t=0.06$ s. Note that it was previously found that spurious
pressure oscillations still occur even when using entropy-stable (fully
conservative) schemes and may be tied to spurious increases in thermodynamic
entropy~\citep{Gou20_2}. Here, with standard overintegration, such
oscillations appear early in the simulation and grow rapidly in time,
in turn causing substantial undershoots/overshoots in temperature
and other quantities. The large-scale instabilities can cause failure
by, for example, requiring a near-zero time-step size to maintain
positivity and entropy boundedness or yielding temperatures at which
the thermodynamic curve fits are nonsensical. }

\begin{figure}[tbph]
\subfloat[\label{fig:thermal_bubble_p1_initial_T}Initial temperature field.]{\includegraphics[width=0.48\columnwidth]{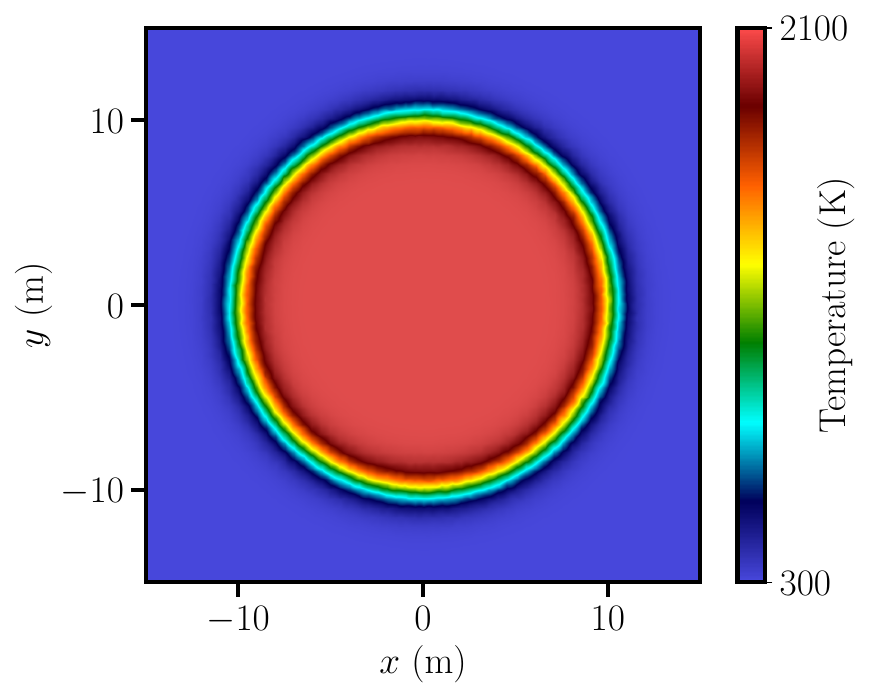}}\hfill{}\subfloat[\label{fig:thermal_bubble_p1_final_T}Final temperature field.]{\includegraphics[width=0.48\columnwidth]{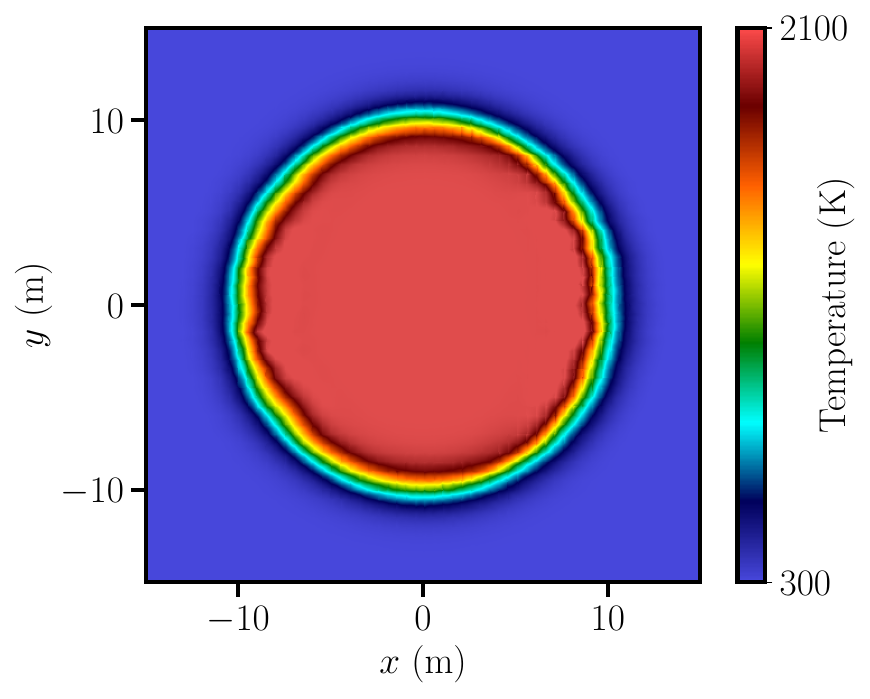}}\hfill{}\subfloat[\label{fig:thermal_bubble_p1_P}Final pressure field.]{\includegraphics[width=0.48\columnwidth]{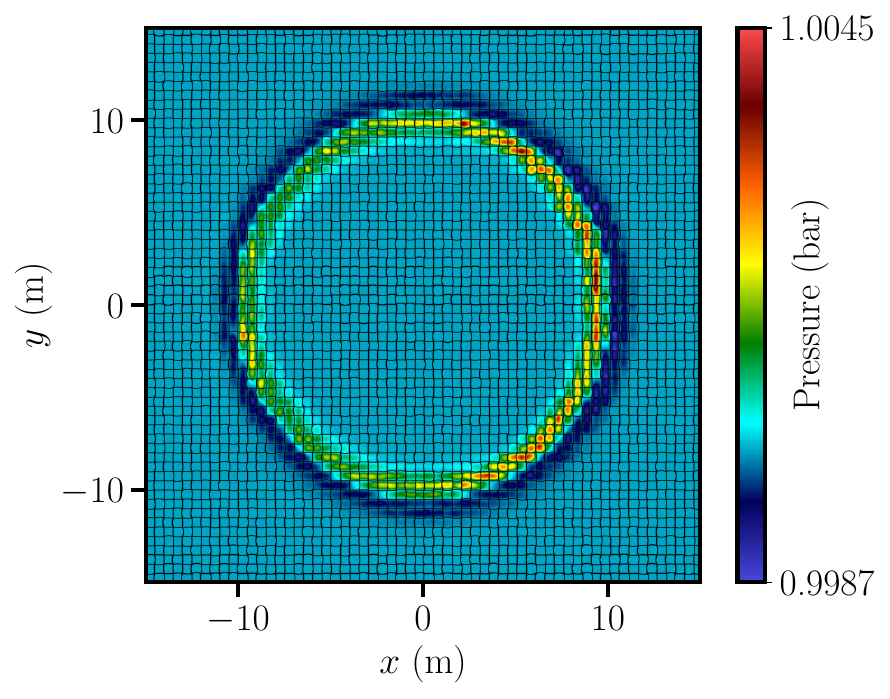}}\hfill{}\subfloat[\label{fig:thermal_bubble_p1_v1}Final streamwise-velocity field.]{\includegraphics[width=0.48\columnwidth]{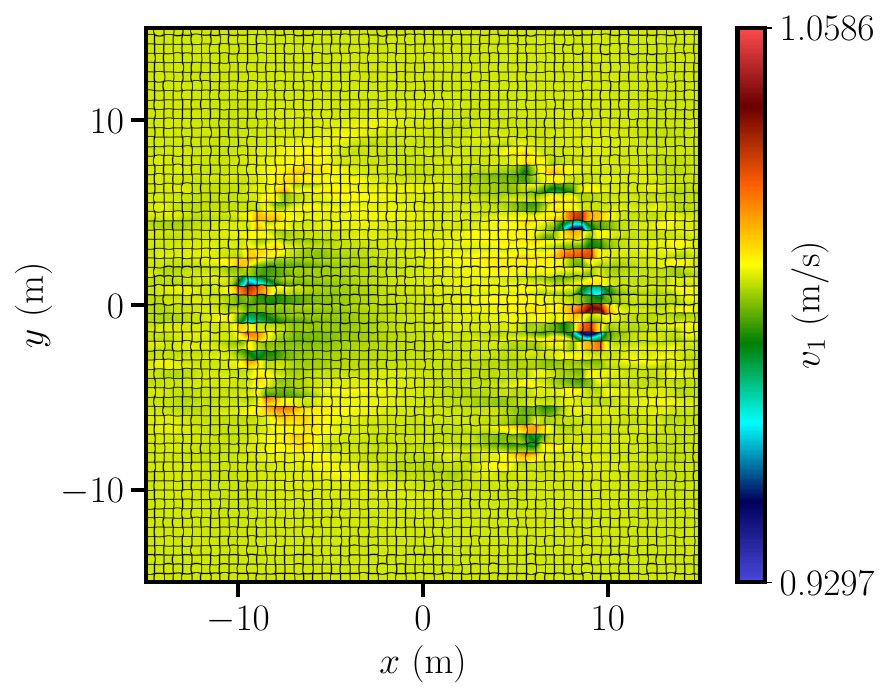}}

\caption{\label{fig:thermal_bubble_p1} \protect\RevisionTextOne{$p=1$ solution to two-dimensional thermal-bubble advection on a curved
quadrilateral grid without artificial viscosity. The final time is
$t=50$ s, corresponding to one period. The colorbar minima and maxima
for the pressure and velocity fields are the respective global minima
and maxima.}}
\end{figure}

\RevisionTextOne{Figures~\ref{fig:thermal_bubble_p2} and~\ref{fig:thermal_bubble_p3}
display the $p=2$ and $p=3$ solutions, respectively. As the polynomial
order increases, the profile of the thermal bubble is better maintained
and the deviations from pressure and velocity equilibrium are decreased. }

\begin{figure}[tbph]
\subfloat[\label{fig:thermal_bubble_p2_initial_T}Initial temperature field.]{\includegraphics[width=0.48\columnwidth]{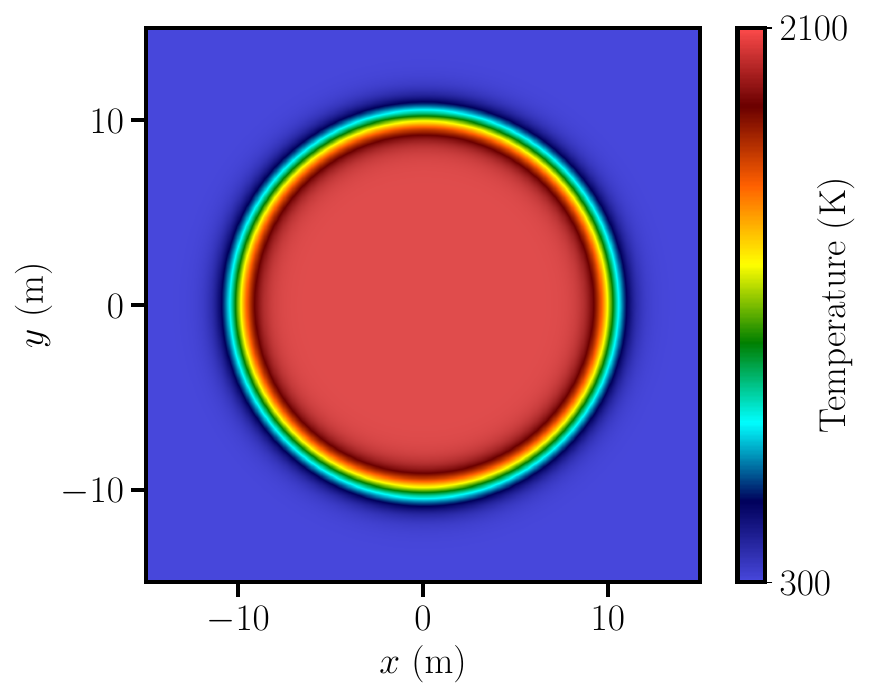}}\hfill{}\subfloat[\label{fig:thermal_bubble_p2_final_T}Final temperature field.]{\includegraphics[width=0.48\columnwidth]{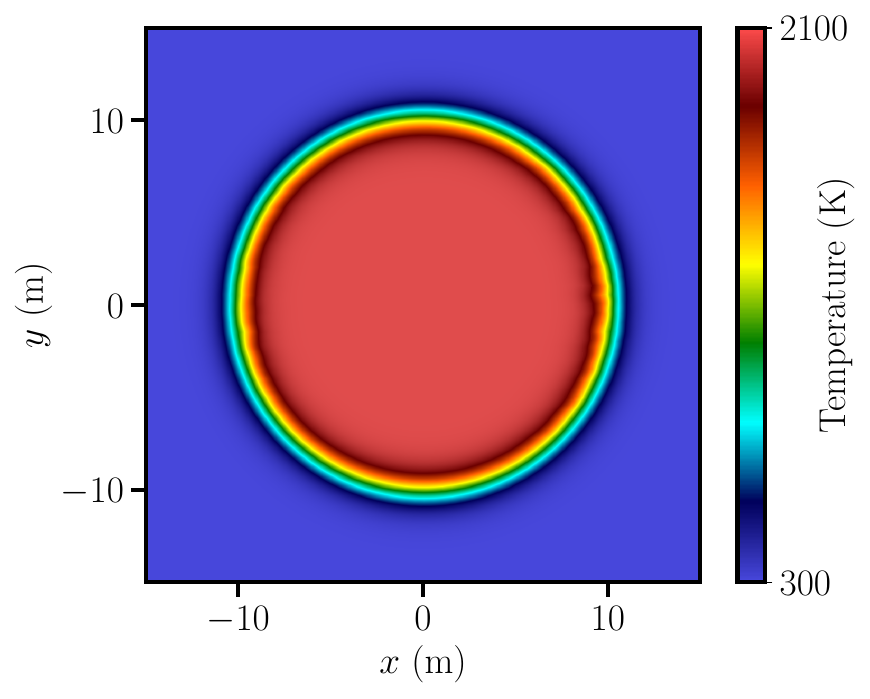}}\hfill{}\subfloat[\label{fig:thermal_bubble_p2_P}Final pressure field.]{\includegraphics[width=0.48\columnwidth]{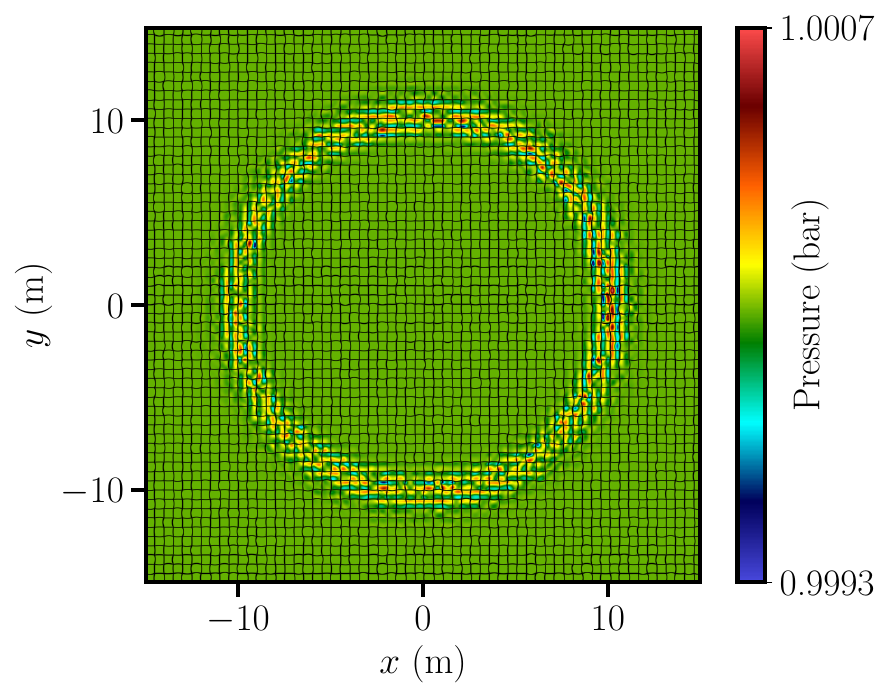}}\hfill{}\subfloat[\label{fig:thermal_bubble_p2_v1}Final streamwise-velocity field.]{\includegraphics[width=0.48\columnwidth]{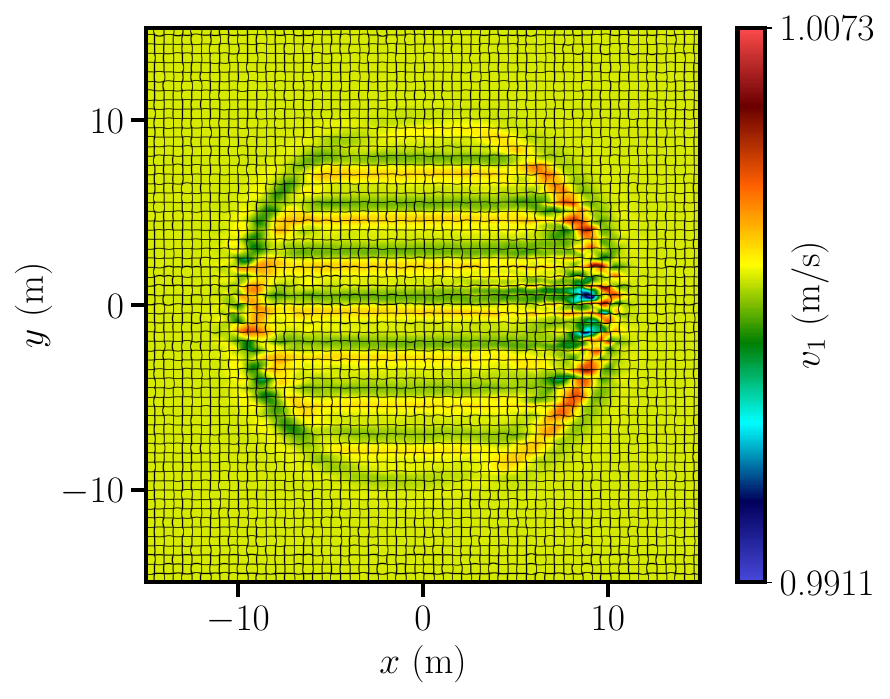}}

\caption{\label{fig:thermal_bubble_p2} \protect\RevisionTextOne{$p=2$ solution to two-dimensional thermal-bubble advection on a curved
quadrilateral grid without artificial viscosity. The final time is
$t=50$ s, corresponding to one period. The colorbar minima and maxima
for the pressure and velocity fields are the respective global minima
and maxima.}}
\end{figure}
\begin{figure}[tbph]
\subfloat[\label{fig:thermal_bubble_p3_initial_T}Initial temperature field.]{\includegraphics[width=0.48\columnwidth]{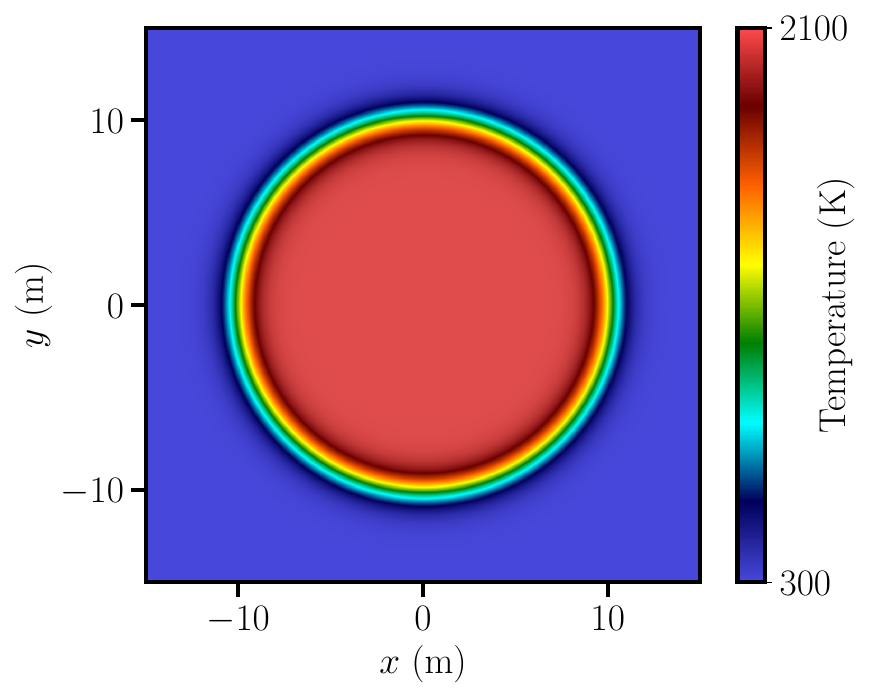}}\hfill{}\subfloat[\label{fig:thermal_bubble_p3_final_T}Final temperature field.]{\includegraphics[width=0.48\columnwidth]{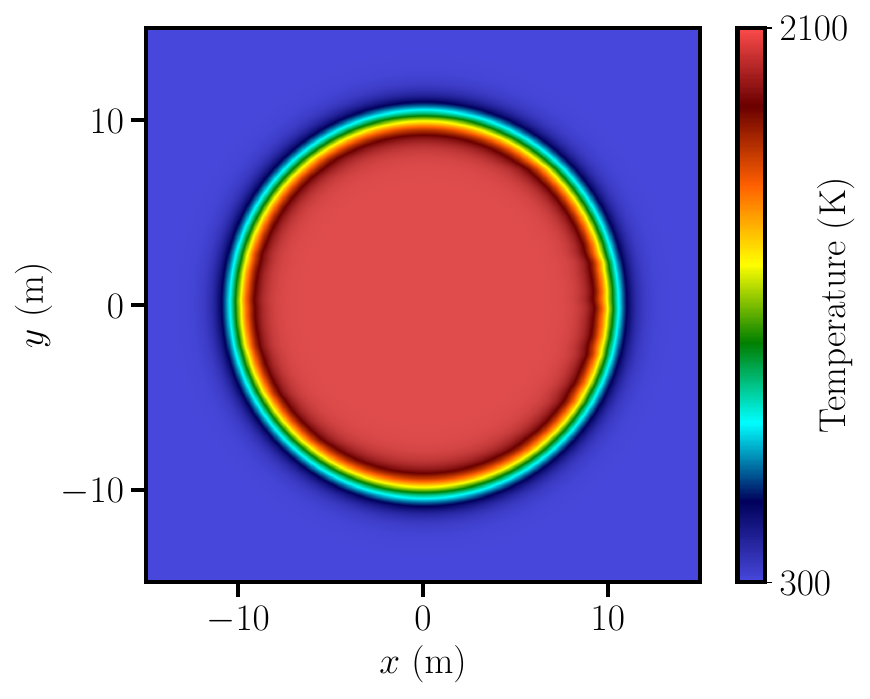}}\hfill{}\subfloat[\label{fig:thermal_bubble_p3_P}Final pressure field.]{\includegraphics[width=0.48\columnwidth]{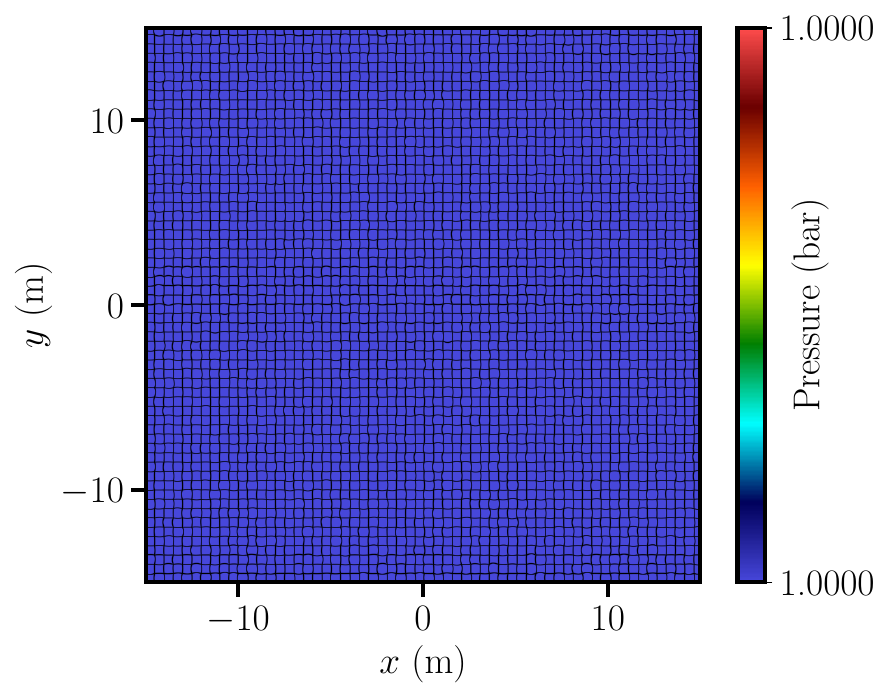}}\hfill{}\subfloat[\label{fig:thermal_bubble_p3_v1}Final streamwise-velocity field.]{\includegraphics[width=0.48\columnwidth]{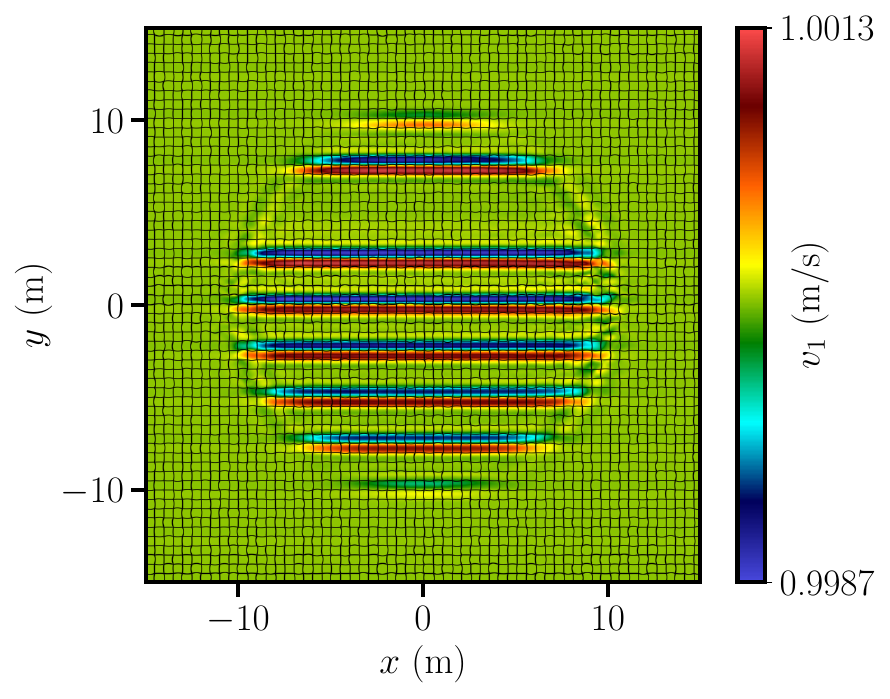}}

\caption{\label{fig:thermal_bubble_p3} \protect\RevisionTextOne{$p=3$ solution to two-dimensional thermal-bubble advection on a curved
quadrilateral grid without artificial viscosity. The final time is
$t=50$ s, corresponding to one period. The colorbar minima and maxima
for the pressure and velocity fields are the respective global minima
and maxima.}}
\end{figure}

\subsection{Two-dimensional detonation wave\label{subsec:2D-detonation-wave}}

This test case is the two-dimensional extension of the one-dimensional
hydrogen-oxygen detonation wave diluted in Argon presented in Part
I~\citep{Chi22}. The computational domain is $\text{\ensuremath{\Omega}}=\left(0,0.45\right)\mathrm{m}\times\left(0,0.06\right)\mathrm{m}$,
with walls at the left, right, bottom, and top boundaries. \RevisionTextThree{An acoustic CFL value of 0.8 is used. Although the results presented
here are obtained with the SSPRK2 time integration scheme, we have
confirmed that the SSPRK3 scheme does not appreciably change the results.} The detonation wave moves from left to right. To perturb the detonation,
two high-temperature/high-pressure circular regions,
\begin{align*}
\mathcal{C}_{1} & =\left\{ x\left|\sqrt{\left(x_{1}-0.021\right)^{2}+\left(x_{2}-0.015\right)^{2}}<0.0025\text{ m}\right.\right\} ,\\
\mathcal{C}_{2} & =\left\{ x\left|\sqrt{\left(x_{1}-0.022\right)^{2}+\left(x_{2}-0.044\right)^{2}}<0.0025\text{ m}\right.\right\} ,
\end{align*}
are placed to the right of the initial detonation wave. The initial
conditions are given by

\begin{equation}
\begin{array}{cccc}
\qquad\qquad\qquad\left(v_{1},v_{2}\right) & = & \left(0,0\right)\text{ m/s},\\
X_{Ar}:X_{H_{2}O}:X_{OH}:X_{O_{2}}:X_{H_{2}} & = & \begin{cases}
8:2:0.1:0:0\\
7:0:0:1:2
\end{cases} & \begin{array}{c}
x_{1}<0.015\text{ m},x\in\mathcal{C}_{1},x\in\mathcal{C}_{2}\\
\mathrm{otherwise}
\end{array},\\
\qquad\qquad\qquad\qquad P & = & \begin{cases}
\expnumber{5.50}5 & \text{ Pa}\\
\expnumber{6.67}3 & \text{ Pa}
\end{cases} & \begin{array}{c}
x_{1}<0.015\text{ m},x\in\mathcal{C}_{1},x\in\mathcal{C}_{2}\\
\mathrm{otherwise}
\end{array},\\
\qquad\qquad\qquad\qquad T & = & \begin{cases}
3500 & \text{ K}\\
300\text{\hspace{1em}\hspace{1em}} & \text{ K}
\end{cases} & \begin{array}{c}
x_{1}<0.015\text{ m},x\in\mathcal{C}_{1},x\in\mathcal{C}_{2}\\
\mathrm{otherwise}
\end{array}.
\end{array}\label{eq:2D-detonation-initialization}
\end{equation}
No smoothing of the discontinuities in the initial conditions is performed.
\RevisionTextThree{The thermodynamic fits and reaction-rate coefficients employed here
and in the following subsection are provided in Appendix~\ref{sec:Chemical-mechanism-detonation}.}

This flow was also computed by Oran et al.~\citep{Ora98}, but with
a slightly longer domain. Their simulations revealed a complex cellular
structure wherein each cell is of size $0.055\text{ m}\times0.03\text{ m}$.
These results were supported by the experiments conducted by Lefebvre
et al.~\citep{Lef98}. In addition, Houim and Kuo~\citep{Hou11}
and Lv and Ihme~\citep{Lv15} simulated this case with a domain of
height $0.03$ m. In all the aforementioned simulations (not including
those here), the solutions were initialized from a one-dimensional
detonation. 

In previous work, Johnson and Kercher~\citep{Joh20_2} computed this
flow with $p=1$ and artificial viscosity for stabilization. The Westbrook
mechanism~\citep{Wes82} was employed. Complex flow features, such
as Kelvin-Helmholtz instabilities, pressure waves, and triple points,
were well-captured, and the correct cellular structure was predicted.
In particular, given that their domain height was $0.06$ m, there
were two cells in the vertical direction. However, a very fine mesh,
with spacing $h=9\times10^{-5}$ m, was required. Increasing $h$
or $p$ led to instabilities that could not be cured with the artificial
viscosity, resulting in solver divergence. Although increasing $p$
enhances the resolution, the solution also generally becomes more
susceptible to spurious oscillations. Here, we aim to achieve robust
solutions with $p\geq1$ and relatively coarse meshes using the proposed
entropy-bounded DG formulation.  Gmsh~\citep{Geu09} is used to
generate unstructured triangular meshes with the following characteristic
mesh sizes: $2h$, $4h$, $8h$, $16h$, $32h$, and $64h$. 

Figure~\ref{fig:2D-detonation-X-OH-p2-all-h} displays the distribution
of OH mole fraction obtained from $p=2$ solutions at $t=200$~$\mu\mathrm{s}$
on a sequence of meshes. For $64h$, although the solution is stable,
the post-shock flow is extremely smeared due to the exceedingly low
resolution, and some spurious artifacts near the leading shock are
present. Interestingly, the detonation-front location is nevertheless
accurately predicted. For $32h$, despite evident smearing of the
post-shock flow, the overall flow topology can be discerned. With
each successive refinement of the mesh, the flow features behind the
detonation front become sharper. For $4h$ and $2h$, the flow topology
is very well-captured, including the triple points that connect the
Mach stems and incident shock, the transverse waves traveling in the
vertical directions, and the vortices that form behind the detonation
front, demonstrating the ability of the developed entropy-bounded
DG formulation to achieve accurate solutions to complex reacting flows.
Furthermore, in all cases, the solution is stable throughout the simulation,
illustrating the robustness of the proposed formulation. 

\begin{figure}[tbph]
\begin{centering}
\includegraphics[width=0.9\columnwidth]{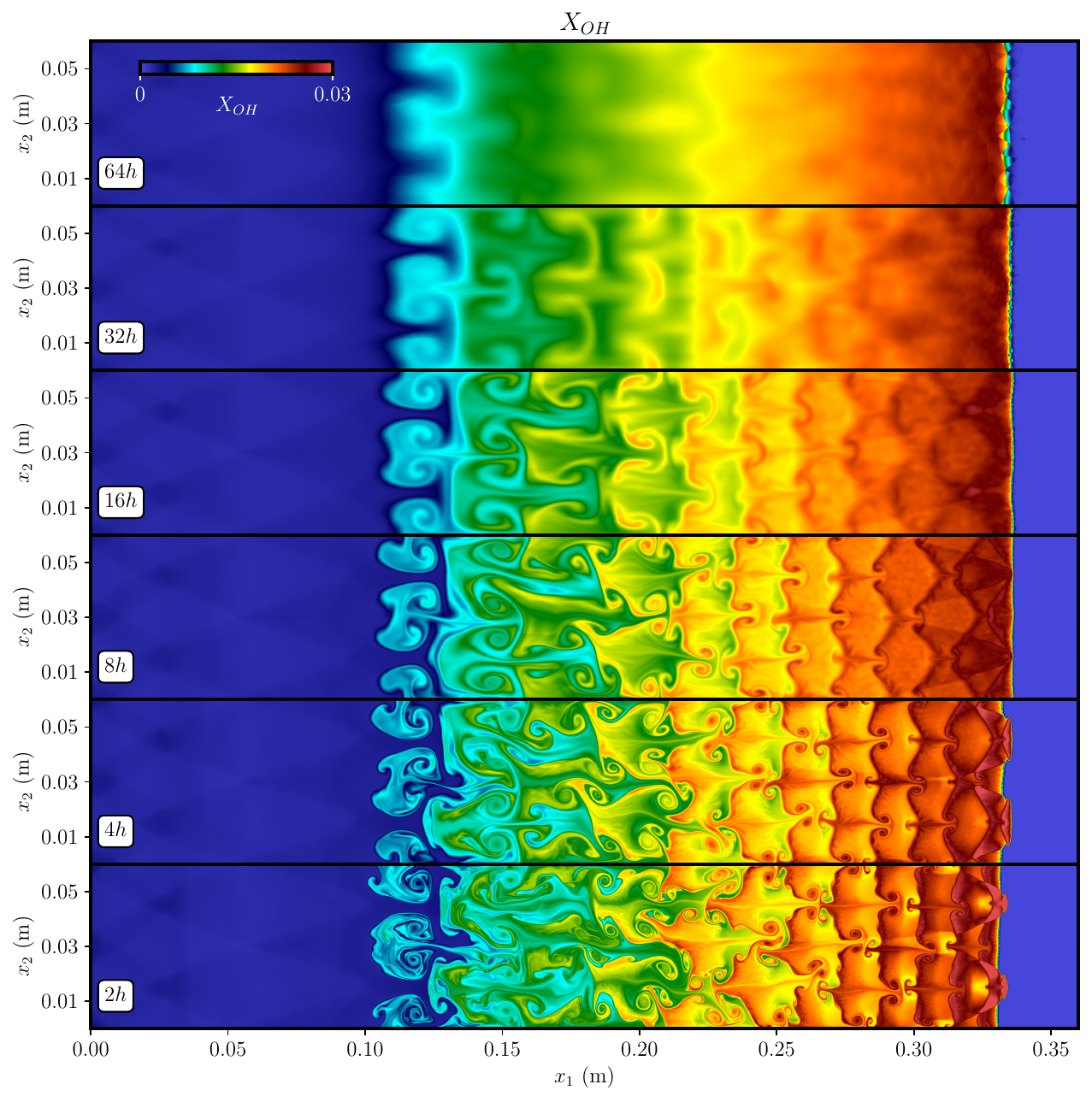}
\par\end{centering}
\caption{\label{fig:2D-detonation-X-OH-p2-all-h}OH mole-fraction field for
a two-dimensional moving detonation wave at $t=200$~$\mu\mathrm{s}$,
computed with $p=2$ on a sequence of meshes, where $h=9\times10^{-5}$
m. The initial conditions are given in Equation~(\ref{eq:2D-detonation-initialization}).}
\end{figure}

Figure~\ref{fig:2D-detonation-pmax} shows the maximum-pressure history,
$P^{*}$, where $P^{*,j+1}(x)=\max\left\{ P^{j+1}(x),P^{*,j}(x)\right\} $,
for the $p=2$ solutions at $t=200$~$\mu\mathrm{s}$. This quantity
reveals the expected diamond-like cellular structure, with two cells
in the vertical direction. The $64h$ lacks any clear cellular structure
due to the excessive smearing. However, detonation cells can be discerned
in the $32h$ case, despite the coarse resolution. From $32h$ to
$8h$, the cellular structure begins to dissipate towards the right
of the domain, though it nevertheless remains intact. At $4h$ and
$2h$, the detonation cells remain sharp throughout. 

\begin{figure}[tbph]
\begin{centering}
\includegraphics[width=0.9\columnwidth]{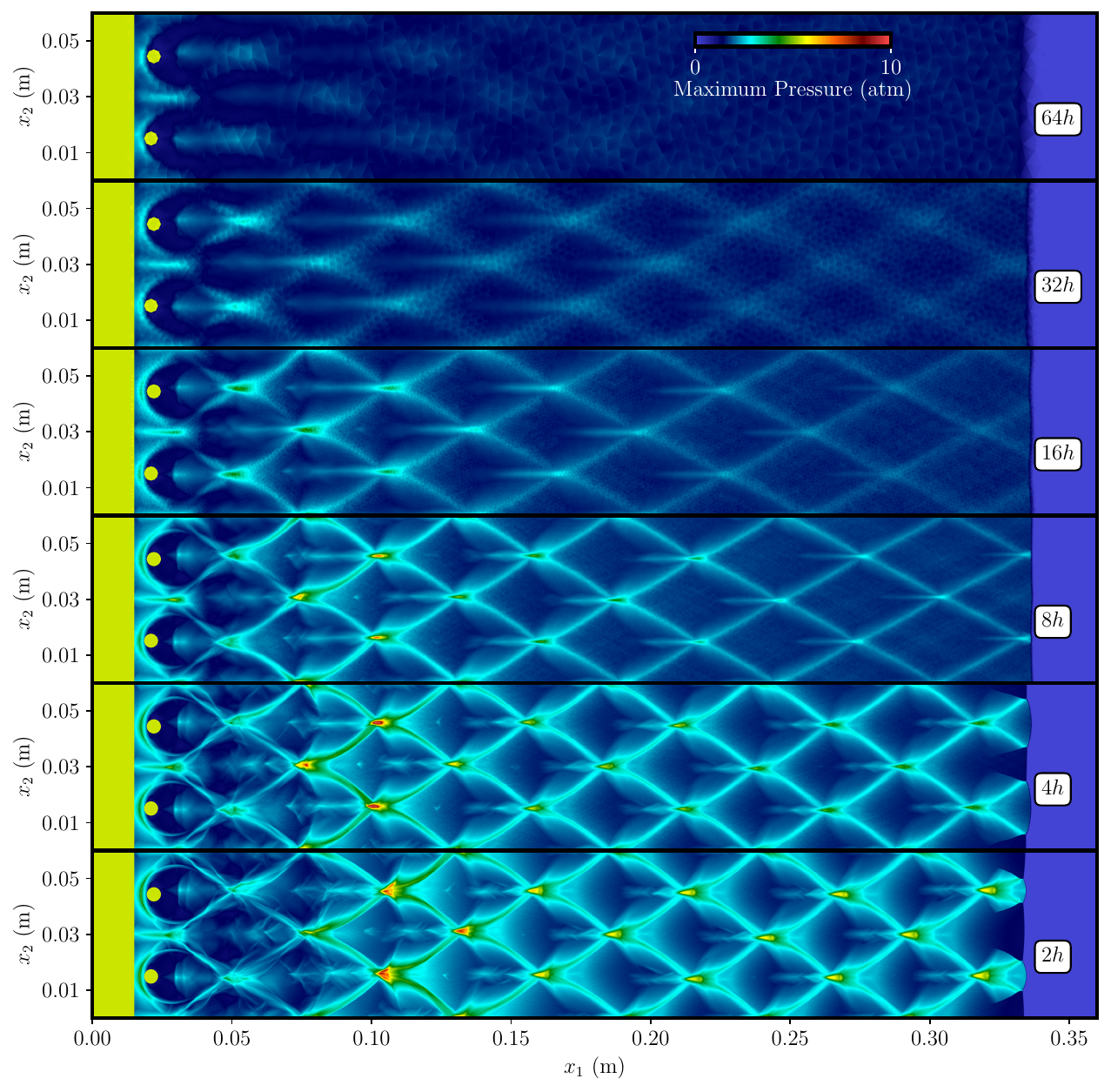}
\par\end{centering}
\caption{\label{fig:2D-detonation-pmax}Maximum-pressure history,$P^{*}$,
where $P^{*,j+1}(x)=\max\left\{ P^{j+1}(x),P^{*,j}(x)\right\} $,
for a two-dimensional moving detonation wave at $t=200$~$\mu\mathrm{s}$
computed with $p=2$ on a sequence of uniformly refined meshes, where
$h=9\times10^{-5}$ m. The initial conditions are given in Equation~(\ref{eq:2D-detonation-initialization}).}
\end{figure}

It remains to be seen how much smearing of the solution at and behind
the detonation front is acceptable in a larger-scale configuration,
although it is encouraging that the front location and cellular structure
are overall well-predicted even with extremely coarse meshes. However,
a thorough investigation of this matter is outside the scope of this
study.  Figure~\ref{fig:2D-detonation-X-OH-8h-all-p} compares the
solutions for $p=1$, $p=2$, and $p=3$ on the $8h$ mesh, which
has 123,466 cells. Some smearing of the flow field is observed in
the $p=1$ solution. Increasing $p$ results in sharper predictions
of the complex flow structures. \RevisionTextThree{Also included is a $p=1$ solution on the $4h$, 488,772-cell mesh,
which has slightly more degrees of freedom than the $p=3,\:8h$ solution.
The flow near the detonation front is sharper in the $p=1,\:4h$ solution,
while the vortical structures farther behind the detonation front
are better resolved in the $p=3,\:8h$ solution.} Note, however, that a rigorous comparison among polynomial orders
on the basis of accuracy vs. cost is left for future work; our goal
here is to demonstrate (as in Figures~\ref{fig:2D-detonation-pmax}
and~\ref{fig:2D-detonation-X-OH-8h-all-p}) that the proposed methodology
can robustly compute accurate solutions using high-order polynomial
approximations.
\begin{figure}[tbph]
\begin{centering}
\includegraphics[width=0.9\columnwidth]{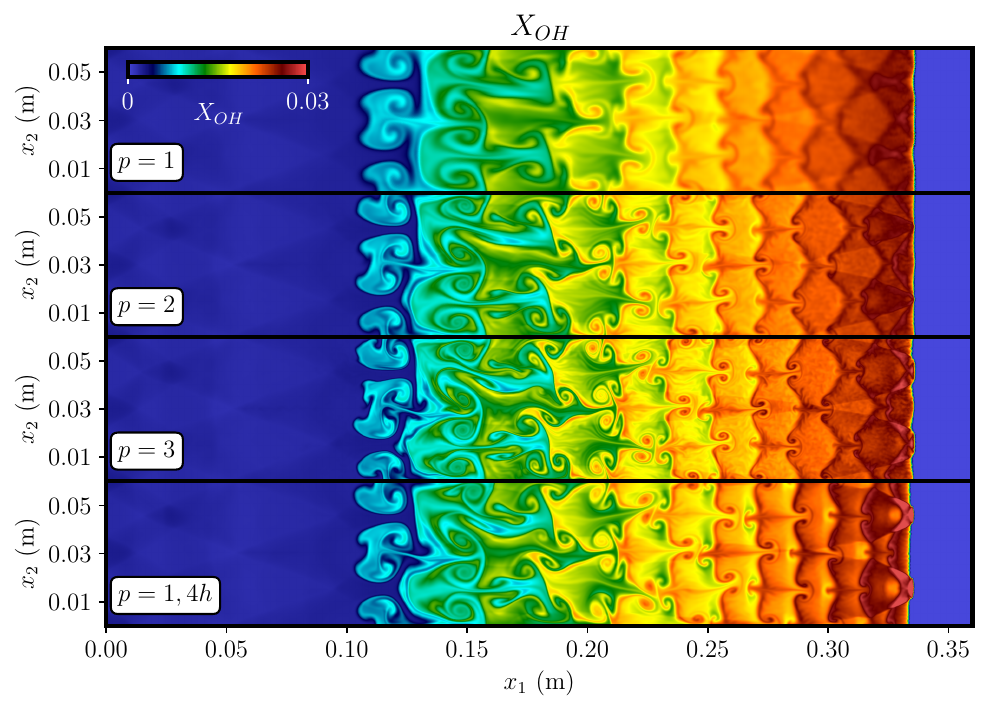}
\par\end{centering}
\caption{\label{fig:2D-detonation-X-OH-8h-all-p}OH mole-fraction field for
a two-dimensional moving detonation wave at $t=200$~$\mu\mathrm{s}$
computed with the $8h$ mesh, where $h=9\times10^{-5}$ m, using the
following polynomial orders: $p=1$, $p=2$, and $p=3$. \protect\RevisionTextThree{Also included is a $p=1$ solution on the $4h$ mesh.}
The initial conditions are given in Equation~(\ref{eq:2D-detonation-initialization}).}
\end{figure}

It is worth noting that the entropy limiter, as opposed to solely
the positivity-preserving limiter, significantly improves the stability
of these calculations. Specifically, throughout each simulation, the
entropy limiter maintains a minimum temperature of around 300 K. Without
it, however, the temperature can dip significantly below 300 K, where
the curve fits for thermodynamic properties are no longer valid. Consequently,
the nonlinear solver in the reaction step can slow down significantly
and even stall. As an example, for $p=2$, $64h$, the simulation
with the entropy limiter is over 14 times less expensive than a corresponding
simulation with only the positivity-preserving limiter. In other types
of flows, these errors in temperature may pollute the solution in
different ways as well. This further highlights not only the numerical
difficulties of calculating multicomponent, reacting flows with realistic
thermodynamics, but also the advantages of using the entropy limiter
in simulations of such flows.

\RevisionTextThree{Additional diagnostics are obtained for the $p=2$, $32h$ calculation
as a representative example. Figure~\ref{fig:2D_detonation_dt} displays
the time-step size at every step. At the beginning of the simulation,
large variations in the time-step size are observed due to start-up
effects and the onset of high-temperature chemical reactions. The
time-step size then gradually increases and eventually oscillates
around a steady value, reflecting the periodic manner in which the
triple-point tracks collide with each other and reflect from the walls.
At each time step, only one DGODE sub-step is required (i.e., no $h$-adaptivity
in time is needed).}
\begin{figure}[tbph]
\begin{centering}
\includegraphics[width=0.6\columnwidth]{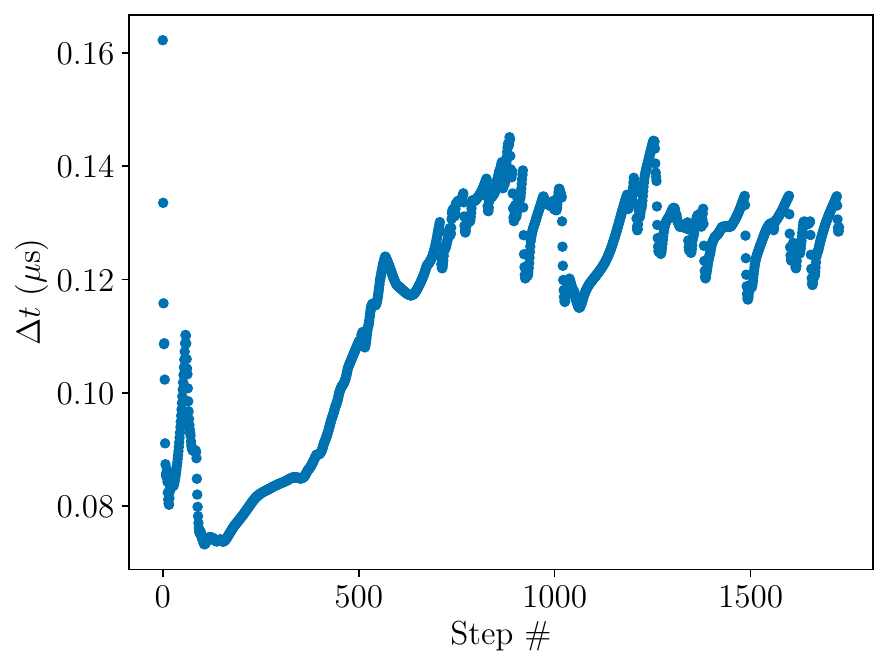}
\par\end{centering}
\caption{\label{fig:2D_detonation_dt}Variation of time-step size for the $p=2$,
$32h$ calculation. The initial conditions for this two-dimensional
hydrogen detonation problem are given in Equation~(\ref{eq:2D-detonation-initialization}). }
\end{figure}

Figure~\ref{fig:2D_detonation_conservation_error} presents the percent
error in mass, energy, and atom conservation for $p=3$, $64h$ as
a representative example, calculated every $0.200\;\mu\mathrm{s}$
(for a total of 1000 samples). Specifically, $\mathsf{N}_{O}$, $\mathsf{N}_{H}$,
and $\mathsf{N}_{Ar}$ denote the total numbers of oxygen, hydrogen,
and argon atoms in the mixture. The error profiles oscillate around
$10^{-13}$\%  due to minor numerical-precision issues. Overall, the
errors remain close to machine precision, confirming that the entropy-bounded
DG formulation is conservative. Figure~\ref{fig:2D_detonation_conservation_error}
also shows the error in mass conservation (calculated every time step)
for a solution computed using a commonly employed clipping procedure
in which negative species concentrations are set to zero, instead
of the limiting procedure described in Section~\ref{subsec:limiting-procedure}.
Although artificial viscosity is still employed, the error increases
rapidly until the solver diverges.

\begin{figure}[tbph]
\begin{centering}
\includegraphics[width=0.6\columnwidth]{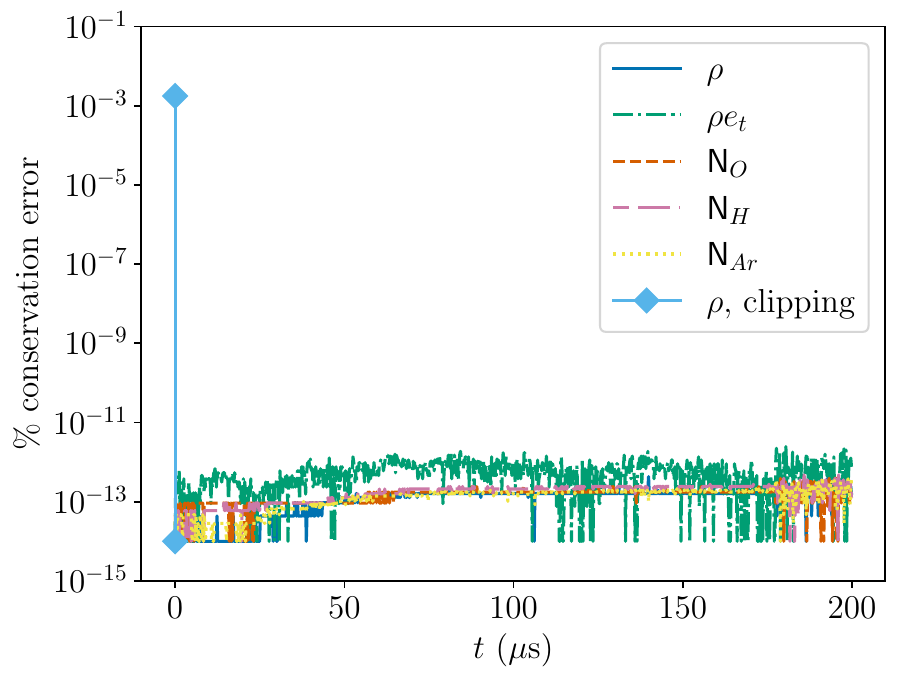}
\par\end{centering}
\caption{\label{fig:2D_detonation_conservation_error}Percent error in mass,
energy, and atom conservation for $p=3$, $64h$, where $h=9\times10^{-5}$
m. The initial conditions for this two-dimensional hydrogen detonation
problem are given in Equation~(\ref{eq:2D-detonation-initialization}).
Also included is the error in mass conservation for a solution computed
using a commonly employed clipping procedure in which negative species
concentrations are set to zero instead of the limiting procedure described
in Section~\ref{subsec:limiting-procedure}.}
\end{figure}

Finally, we recompute the $p=2$, $64h$ case with curved elements
of quadratic geometric order. In particular, high-order geometric
nodes are inserted at the midpoints of the vertices of each element.
At interior edges, the midpoint nodes are then slightly perturbed
from their initial positions. These perturbations are performed only
for $x>0.05\text{ m}$ to guarantee identical initial conditions.
We intentionally run this low-resolution case to ensure extensive
activation of the limiter and aggressively test the robustness of
the formulation when curved elements are employed. The distributions
of OH mole fraction are given in Figure~\ref{fig:2D-detonation-X-OH-64h-linear-curved}
for both the linear and curved meshes. Though slight mesh imprinting
can be observed for the curved mesh, likely due to the lower-quality
elements, the two solutions are extremely similar and remain stable
throughout. This result confirms that the proposed formulation is
indeed compatible with nonlinear elements.

\begin{figure}[tbph]
\begin{centering}
\includegraphics[width=0.9\columnwidth]{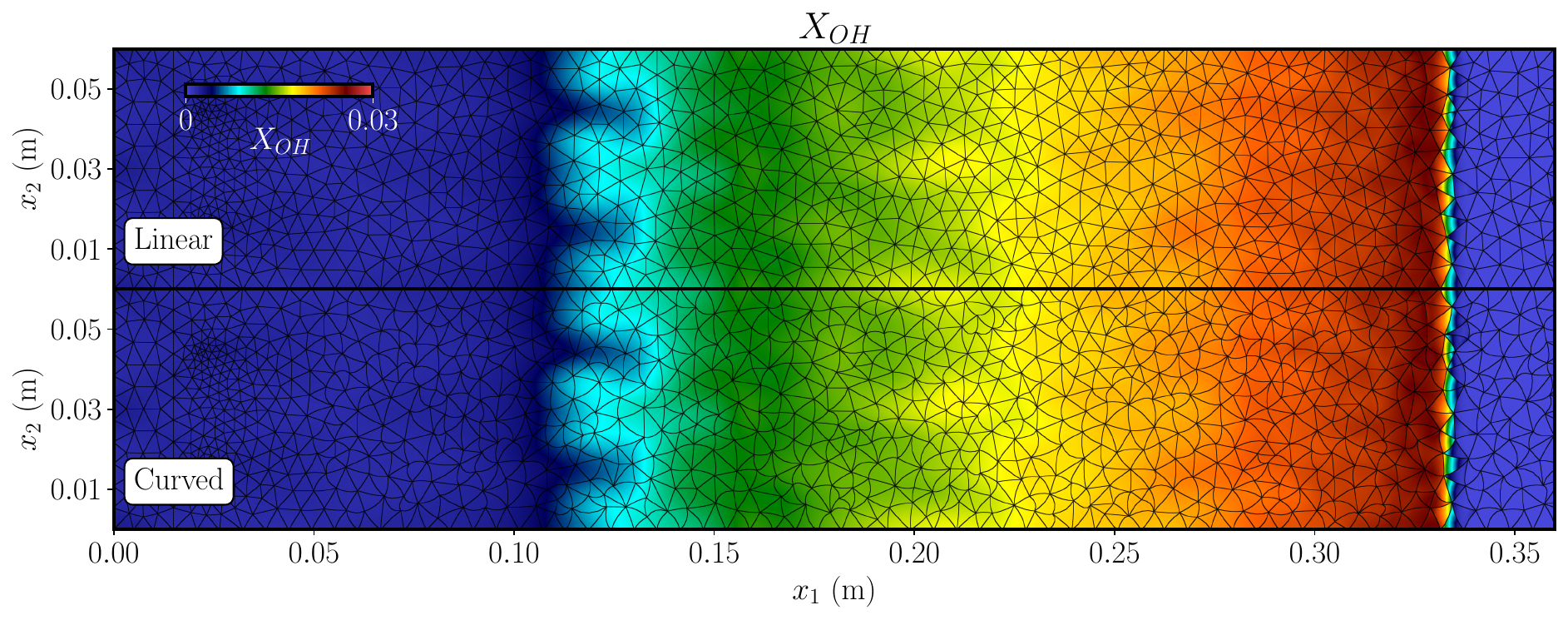}
\par\end{centering}
\caption{\label{fig:2D-detonation-X-OH-64h-linear-curved}OH mole-fraction
field for a two-dimensional moving detonation wave at $t=200$~$\mu\mathrm{s}$
computed with $p=2$ and $64h$, where $h=9\times10^{-5}$ m, on linear
and curved meshes. The curved mesh, which is of quadratic order, is
obtained by inserting high-order geometric nodes into the linear mesh
and then perturbing the inserted nodes. The initial conditions are
given in Equation~(\ref{eq:2D-detonation-initialization}).}
\end{figure}

\subsection{Three-dimensional detonation wave\label{subsec:3D-detonation-wave}}

This test case is the three-dimensional extension of the two-dimensional
detonation wave computed in Section~\ref{subsec:2D-detonation-wave}.
The computational domain is $\text{\ensuremath{\Omega}}=\left(0,0.3\right)\mathrm{m}\times\left(0,0.03\right)\mathrm{m}\times\left(0,0.03\right)\;\mathrm{m}$,
with walls at all boundaries. To reduce computational and memory demands,
the size of the domain in the $x_{1}$- and $x_{2}$-directions is
smaller than in Section~\ref{subsec:2D-detonation-wave}. \RevisionTextThree{An acoustic CFL value of 0.8 is used. Although the results presented
here are obtained with the SSPRK2 time integration scheme, we have
confirmed that the SSPRK3 scheme does not appreciably change the results.} The initial conditions are given by

\begin{equation}
\begin{array}{cccc}
\qquad\qquad\qquad\left(v_{1},v_{2}\right) & = & \left(0,0\right)\text{ m/s},\\
X_{Ar}:X_{H_{2}O}:X_{OH}:X_{O_{2}}:X_{H_{2}} & = & \begin{cases}
7:0:0:1:2\\
8:2:0.1:0:0
\end{cases} & \begin{array}{c}
x_{1}\geq0.015\text{ m},x\in\mathcal{C}\\
\mathrm{otherwise}
\end{array},\\
\qquad\qquad\qquad\qquad P & = & \begin{cases}
\expnumber{6.67}3 & \text{ Pa}\\
\expnumber{5.50}5 & \text{ Pa}
\end{cases} & \begin{array}{c}
x_{1}\geq0.015\text{ m},x\in\mathcal{C}\\
\mathrm{otherwise}
\end{array},\\
\qquad\qquad\qquad\qquad T & = & \begin{cases}
300\text{\hspace{1em}\hspace{1em}} & \text{ K}\\
3500 & \text{ K}
\end{cases} & \begin{array}{c}
x_{1}\geq0.015\text{ m},x\in\mathcal{C}\\
\mathrm{otherwise}
\end{array},
\end{array}\label{eq:3D-detonation-initialization}
\end{equation}
where $\mathcal{C}$ represents a pocket of unburnt gas, similar to
the two-dimensional configuration in~\citep{Lv15}, defined as
\[
\mathcal{C}=\left\{ x\left|\sqrt{\left(x_{1}-0.014\right)^{2}+\left(x_{2}-0.015\right)^{2}+x_{3}^{2}}<0.005\text{ m}\right.\right\} .
\]
No smoothing of the discontinuities in the initial conditions is performed.
The same type of diamond-like cellular structure is expected for this
case, with one cell in the $x_{2}$- and $x_{3}$-directions~\citep{Dei03,Tsu02}.
Our objective with this demonstration test case is to achieve an accurate
solution to this complex, large-scale flow problem in a robust manner
using high-order polynomials. We leverage previous studies that detail
the physics of this flow configuration~\citep{Dei03,Tsu02}. Based
on the results in Section~\ref{subsec:2D-detonation-wave}, we select
$p=2$ and a characteristic mesh spacing of $4.8h$, in order to balance
accuracy with computational and memory costs. These choices result
in an unstructured mesh containing approximately 15 million tetrahedral
elements. 

Figures~\ref{fig:XOH-temporal-history} and~\ref{fig:temperature-temporal-history}
display the $X_{OH}$ and temperature distributions, respectively,
along the center $x_{3}=0.015\text{ mm}$ plane as the detonation
front traverses the domain. At early times, transverse waves, a large
induction zone, and a high-temperature, highly reactive region are
present as a result of the initialization. As the detonation is established,
the flow field behind the detonation front becomes characterized by
unsteady, multidimensional flow features, such as vortical structures,
wave interactions and reflections, and Kelvin-Helmholtz instabilities.
In the vicinity of the detonation front, transverse waves traveling
in the vertical directions and triple points can be observed. Some
small-scale numerical instabilities are present, but these can be
dampened with improvements to the artificial-viscosity formulation
(e.g., smooth artificial viscosity~\citep{Chi19}) or more complex
limiting strategies (outside the scope of this paper). Furthermore,
the small-scale instabilities do not impede the temporal advancement
of the solution or cause solver divergence.

\begin{figure}[tbph]
\begin{centering}
\includegraphics[width=0.9\columnwidth]{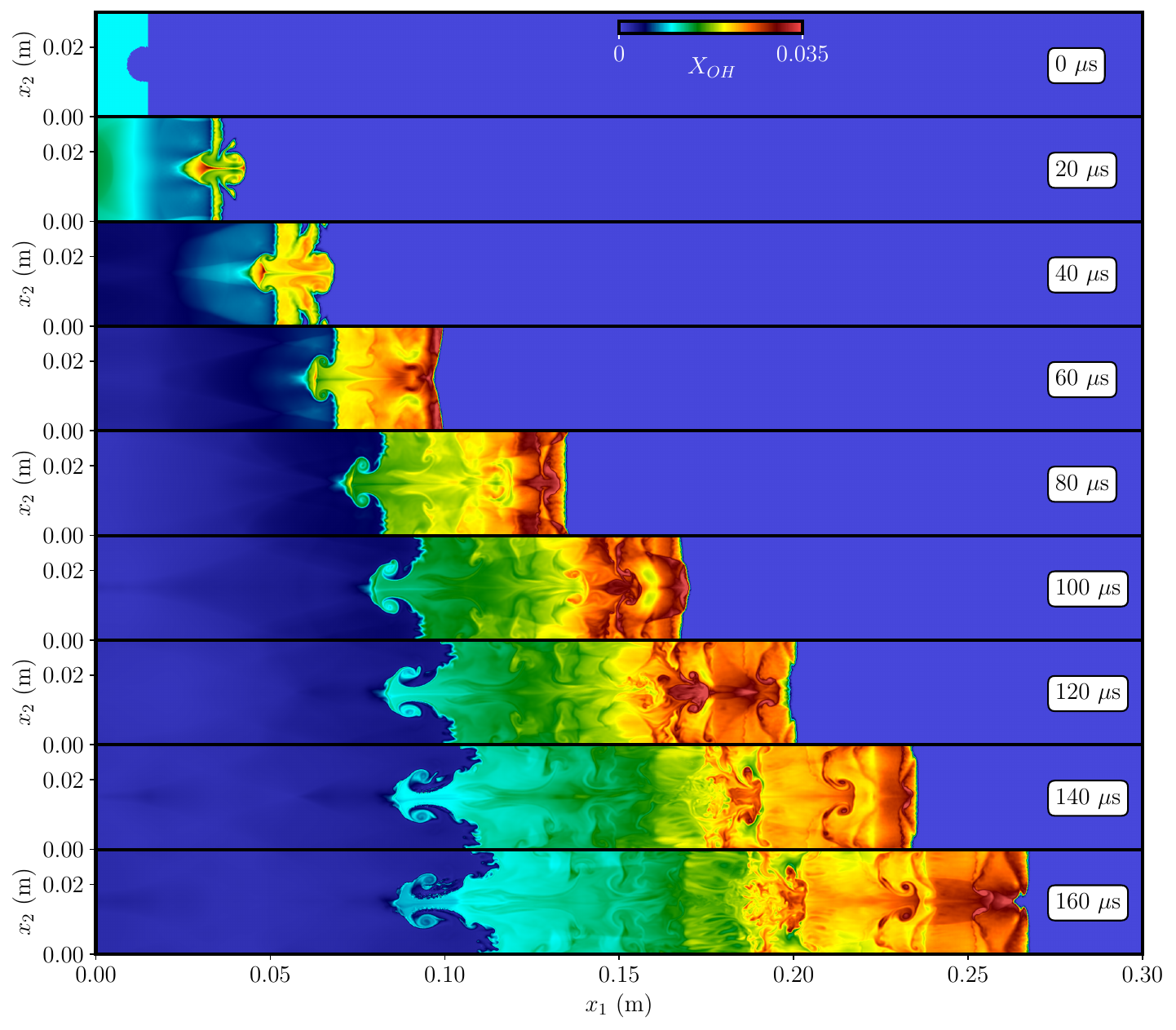}
\par\end{centering}
\caption{\label{fig:XOH-temporal-history}$X_{OH}$ distributions along the
center $x_{3}=0.015\text{ mm}$ plane for a three-dimensional moving
detonation wave computed with $p=2$ and $4.8h$, where $h=9\times10^{-5}$
m. The initial conditions are given in Equation~(\ref{eq:3D-detonation-initialization}).}
\end{figure}
\begin{figure}[tbph]
\begin{centering}
\includegraphics[width=0.9\columnwidth]{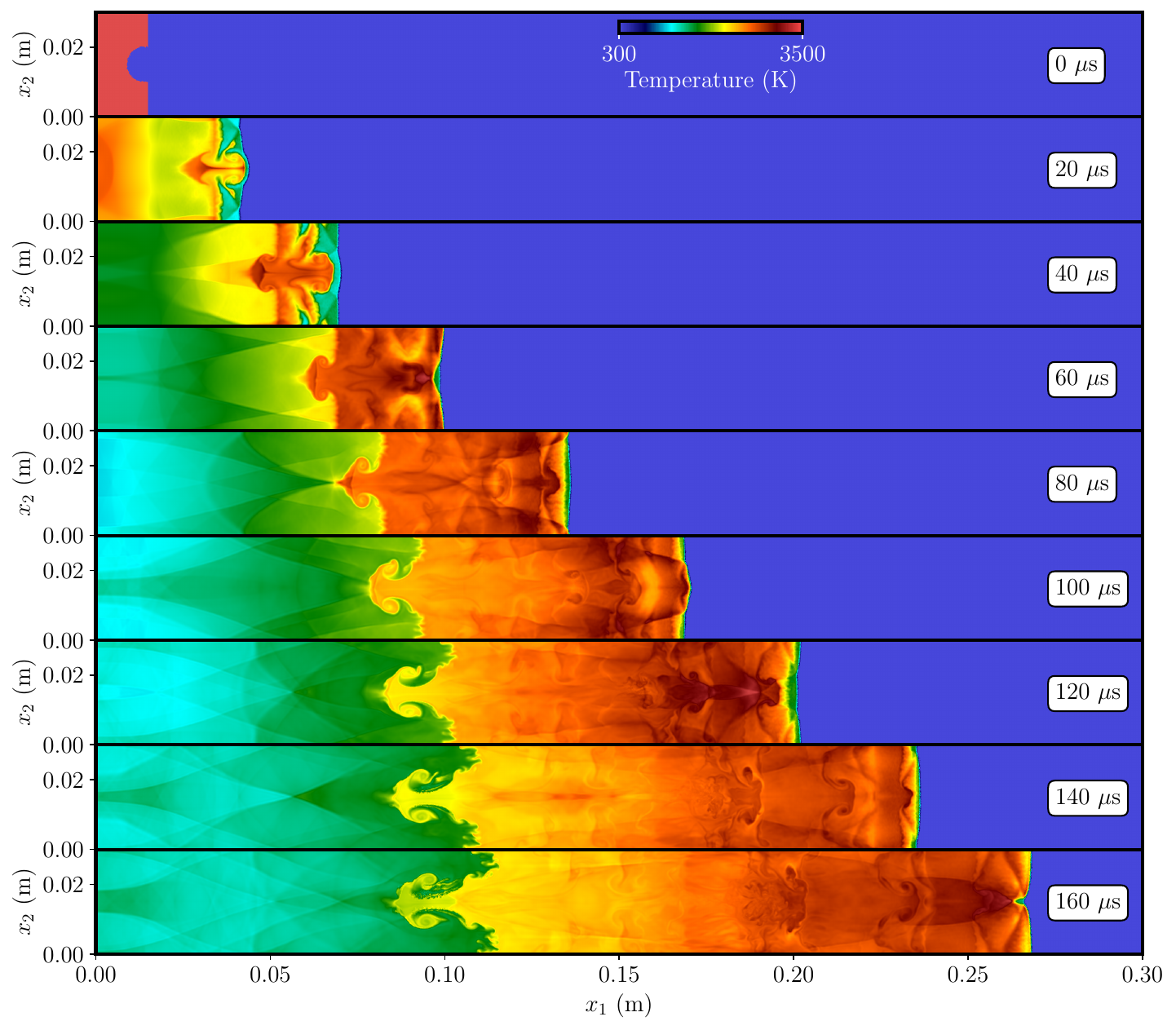}
\par\end{centering}
\caption{\label{fig:temperature-temporal-history}Temperature distributions
along the center $x_{3}=0.015\text{ mm}$ plane for a three-dimensional
moving detonation wave computed with $p=2$ and $4.8h$, where $h=9\times10^{-5}$
m. The initial conditions are given in Equation~(\ref{eq:3D-detonation-initialization}).}
\end{figure}

Figure~\ref{fig:3D-detonation-planes} displays $X_{OH}$ and temperature
distributions along various $x_{1}x_{2}$-planes at $t=176\:\mu\mathrm{s}$,
shortly before the detonation front collides with the wall. The distributions
along the corresponding $x_{1}x_{3}$-planes are qualitatively very
similar and therefore omitted for brevity. In the post-shock region,
$X_{OH}$ and temperature generally decrease with distance from the
shock front. Nevertheless, small pockets of low OH mole fraction,
surrounded by regions of higher OH mole fraction, can be observed.
These results highlight the multidimensional nature of this flow and
further illustrate that the complex flow topology is well-captured
using the proposed formulation.

\begin{figure}[tbph]
\begin{centering}
\subfloat[$X_{OH}$ ]{\begin{centering}
\includegraphics[width=0.9\linewidth]{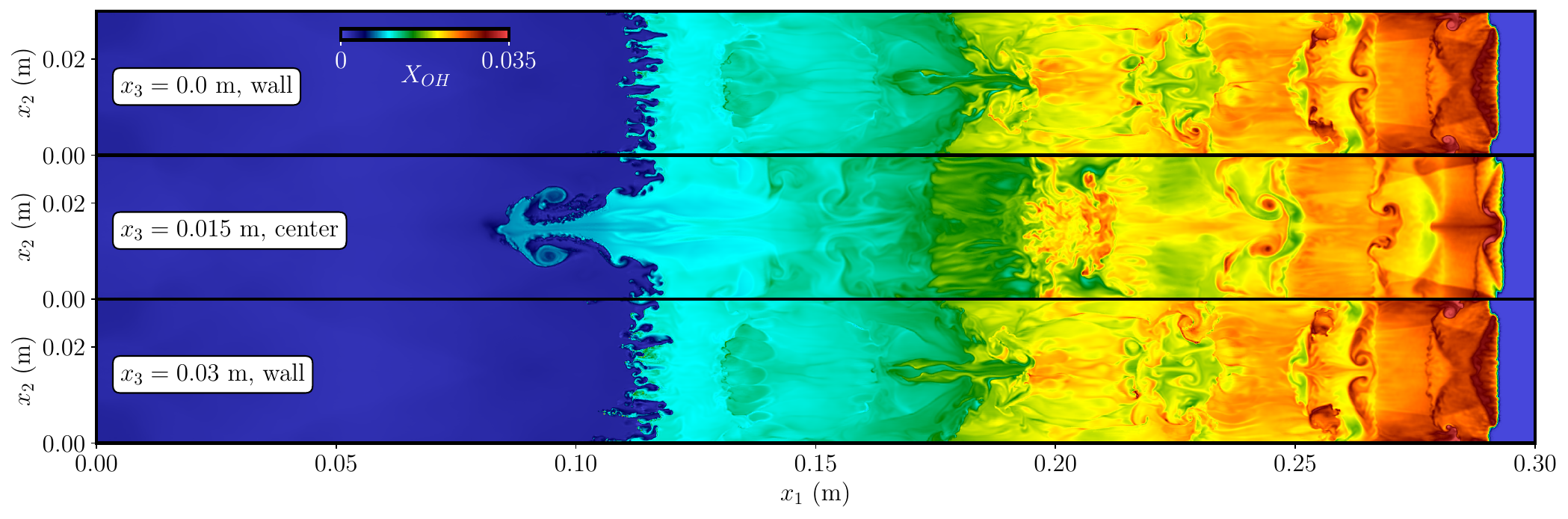}
\par\end{centering}
}\hfill{}\subfloat[Temperature]{\begin{centering}
\includegraphics[width=0.9\linewidth]{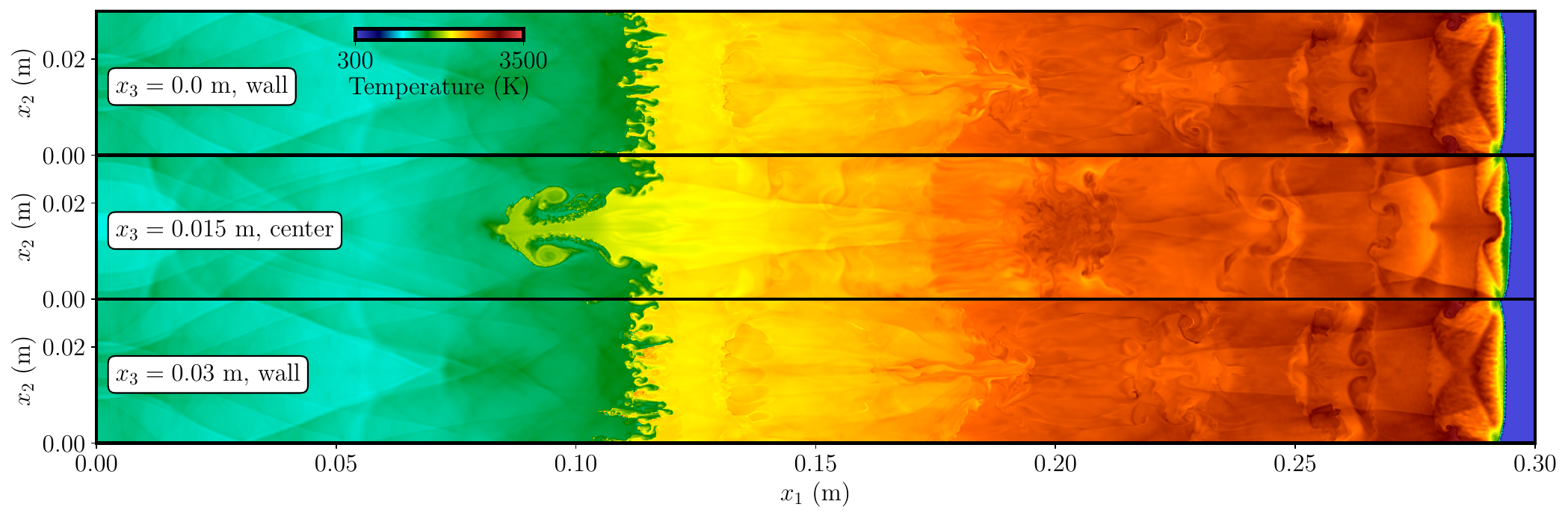}
\par\end{centering}
}
\par\end{centering}
\caption{\label{fig:3D-detonation-planes}$X_{OH}$ and temperature distributions
along various $x_{1}x_{2}$-planes at $t=176\:\mu\mathrm{s}$ for
a three-dimensional moving detonation wave computed with $p=2$ and
$4.8h$, where $h=9\times10^{-5}$ m. The initial conditions are given
in Equation~(\ref{eq:3D-detonation-initialization}). }
\end{figure}

Figure~\ref{fig:pressure-front} presents several instantaneous pressure
isosurfaces colored by the quantity $x_{1}-\overline{x}_{1}$, where
$\overline{x}_{1}$ is the mean $x_{1}$-coordinate of the given isosurface,
in order to illustrate the structure of the detonation front, which
is characterized by two orthogonal, two-dimensional waves. Two vertical
triple-point lines and two horizontal triple-point lines move in-phase
with each other in the horizontal and vertical directions, respectively,
indicating an in-phase rectangular mode~\citep{Tsu02}. Each pair
of triple-point lines propagates outwards, collides with the boundary
walls, and then reflects inwards. Subsequently, they collide with
each other and then reflect outwards, continuing in a periodic manner. 

\begin{figure}[tbph]
\begin{centering}
\includegraphics[width=0.9\columnwidth]{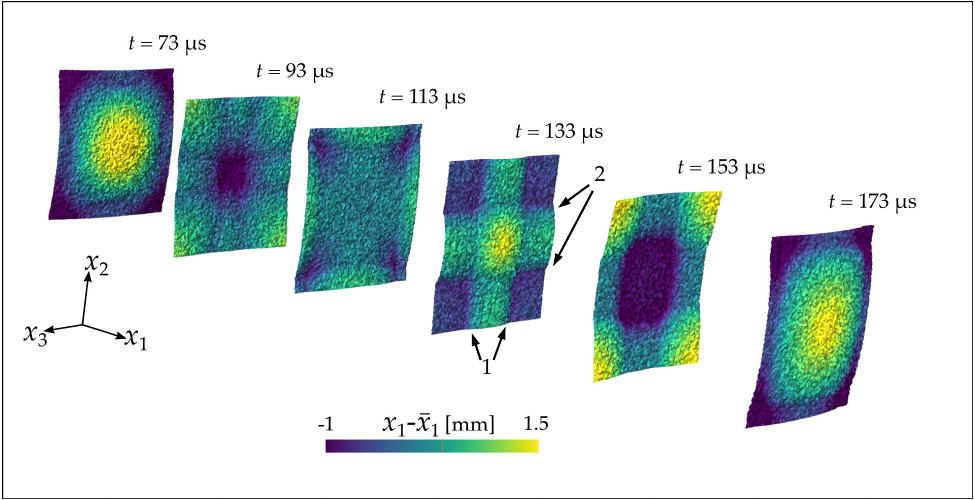}
\par\end{centering}
\caption{\label{fig:pressure-front}Instantaneous pressure isosurfaces colored
by the quantity $x_{1}-\overline{x}_{1}$, where $\overline{x}_{1}$
is the mean $x_{1}$-coordinate of the given isosurface, at various
times for a three-dimensional moving detonation wave computed with
$p=2$ and $4.8h$, where $h=9\times10^{-5}$ m. The ``1'' arrows
indicate the vertical triple-point lines, and the ``2'' arrows indicate
the horizontal triple-point lines. The initial conditions are given
in Equation~(\ref{eq:3D-detonation-initialization}).}
\end{figure}

Distributions of maximum-pressure history along various $x_{1}x_{2}$-planes
and $x_{1}x_{3}$-planes at $t=176\:\mu\mathrm{s}$ are given in Figure~\ref{fig:3D-detonation-pmax-planes}.
Approximately two and a half cells are observed in the region $x_{1}>0.15\text{ mm}$.
High-pressure explosions occur when the triple-point lines collide
with each other and with the corners. ``Slapping'' waves result
from reflections of the triple-point lines, indicated with the white
arrow in Figure~\ref{fig:3D-detonation-pmax-planes-xy} (top)~\citep{Tsu02}.
The $x_{1}x_{2}$-distributions are nearly identical to the corresponding
$x_{1}x_{3}$-distributions. Figure~\ref{fig:pmax-isosurfaces} shows
isosurfaces of maximum-pressure history for $x_{1}\in\left[0.236,0.292\right]$
m, which illustrate, from a three-dimensional perspective, not only
the temporal evolution of triple-point-line intersections, but also
the slapping waves along the walls and the strong, high-pressure explosions.
Overall symmetry in the $x_{2}$- and $x_{3}$-directions is observed.
Figures~\ref{fig:3D-detonation-pmax-planes} and~\ref{fig:pmax-isosurfaces}
further confirm that the detonation front is characterized by an in-phase
rectangular mode. 

\begin{figure}[tbph]
\begin{centering}
\subfloat[$x_{1}x_{2}$-planes\label{fig:3D-detonation-pmax-planes-xy}]{\begin{centering}
\includegraphics[width=0.9\linewidth]{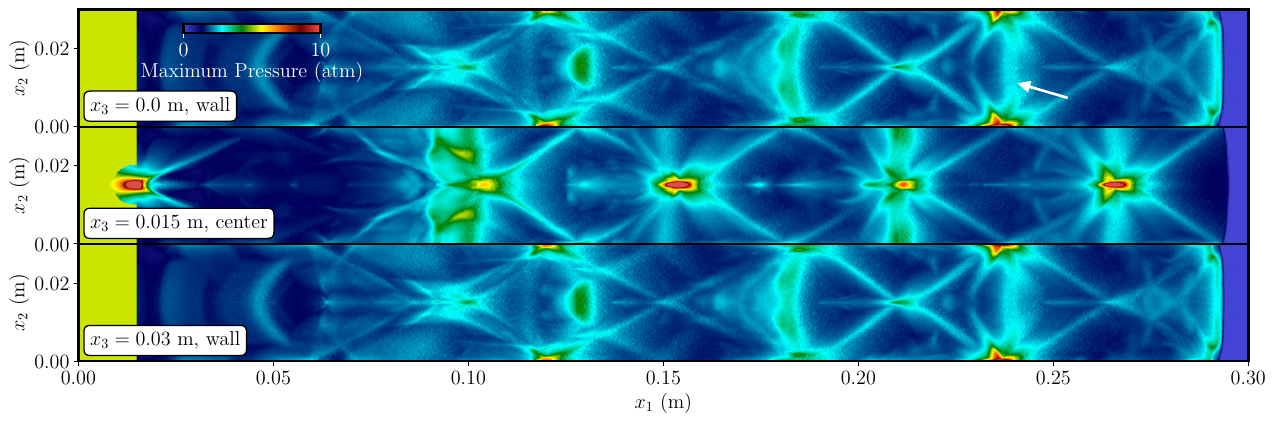}
\par\end{centering}
}\hfill{}\subfloat[$x_{1}x_{3}$-planes]{\begin{centering}
\includegraphics[width=0.9\linewidth]{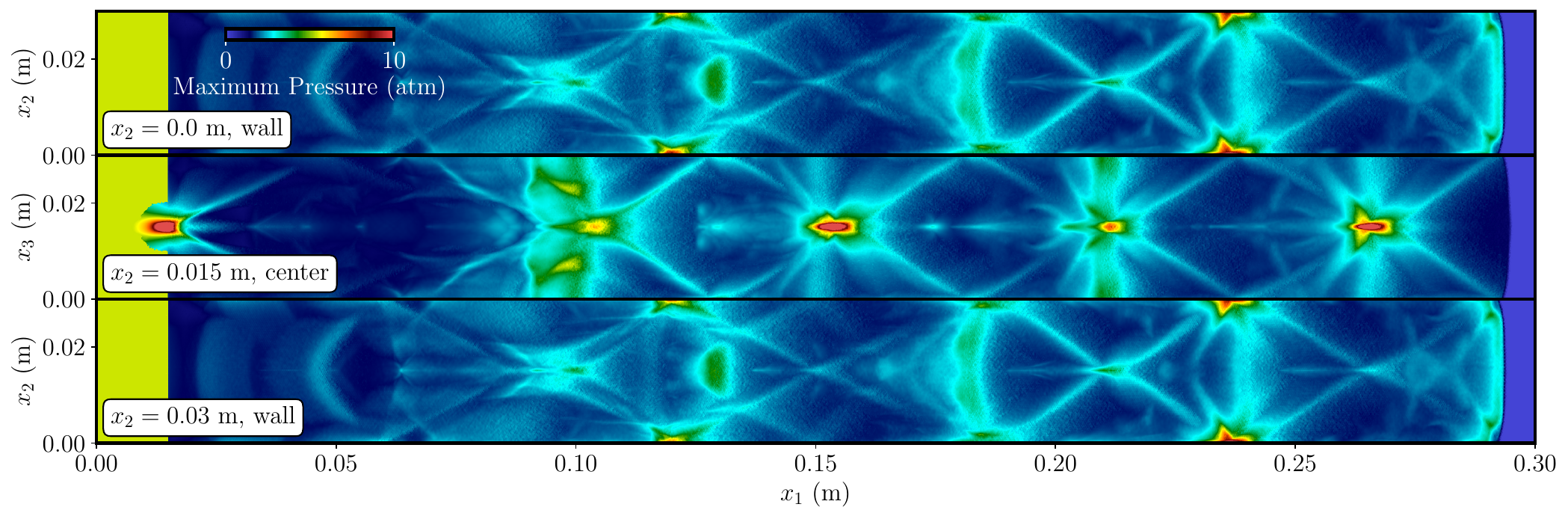}
\par\end{centering}
}
\par\end{centering}
\caption{\label{fig:3D-detonation-pmax-planes}Distributions of maximum-pressure
history along various $x_{1}x_{2}$-planes and $x_{1}x_{3}$-planes
at $t=176\:\mu\mathrm{s}$ for a three-dimensional moving detonation
wave computed with $p=2$ and $4.8h$, where $h=9\times10^{-5}$ m.
The white arrow in Figure~\ref{fig:3D-detonation-pmax-planes-xy}
(top) indicates a ``slapping'' wave. The initial conditions are
given in Equation~(\ref{eq:3D-detonation-initialization}).}
\end{figure}
\begin{figure}[tbph]
\begin{centering}
\includegraphics[width=0.75\columnwidth]{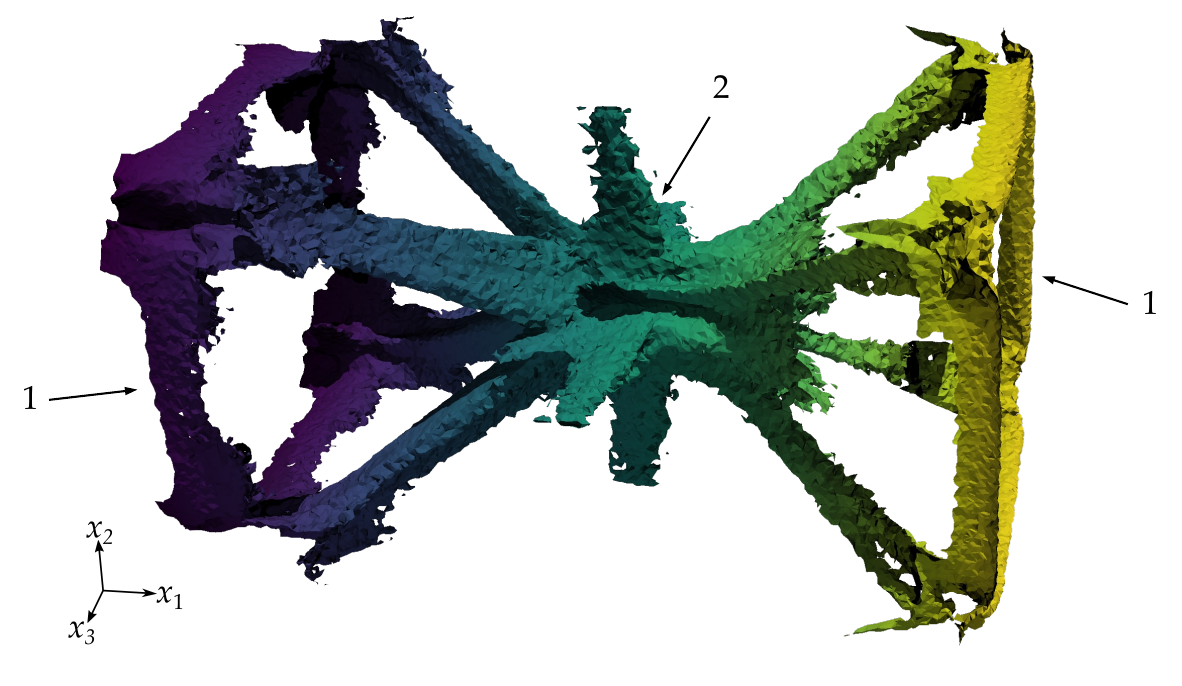}
\par\end{centering}
\caption{\label{fig:pmax-isosurfaces}Isosurfaces of maximum-pressure history
at $t=176\:\mu\mathrm{s}$ for $x_{1}\in\left[0.236,0.292\right]$
m for a three-dimensional moving detonation wave computed with $p=2$
and $4.8h$, where $h=9\times10^{-5}$ m. The isosurfaces are colored
by $x_{1}$-position. Each of the ``1'' arrows indicates a slapping
wave, while the ``2'' arrow indicates the high-pressure explosion
resulting from triple-point-line collisions. The initial conditions
are given in Equation~(\ref{eq:3D-detonation-initialization}).}
\end{figure}

\section{Concluding remarks}

In this second part of our two-part paper, we introduced a positivity-preserving,
entropy-bounded, multidimensional DG methodology for the chemically
reacting, compressible Euler equations, extending the one-dimensional
version presented in Part I~\citep{Chi22}. Compared to current multidimensional
positivity-preserving and/or entropy-bounded DG schemes in the literature,
restrictions on the quadrature rules, numerical flux function, polynomial
order of the geometric approximation, and physical modeling are relaxed.
In particular, the formulation is compatible with arbitrary, curved
elements, any invariant-region-preserving flux, and mixtures of thermally
perfect gases. A simple linear-scaling limiter enforces nonnegative
species concentrations, positive density, positive pressure, and bounded
entropy. Artificial viscosity aids in suppressing small-scale instabilities
not eliminated by the linear scaling. We discussed how to maintain
compatibility with the strategies introduced in~\citep{Joh20_2}
to maintain pressure equilibrium and avoid generating spurious pressure
oscillations. 

We first applied the formulation to the two-dimensional advection
of a thermal bubble. Pressure and velocity equilibrium were shown
to be maintained across various polynomial orders on curved grids.
The formulation was then applied to complex moving detonation waves
in two and three dimensions. In the two-dimensional case, a variety
of mesh sizes and polynomial orders was considered. In~\citep{Joh20_2},
a linear polynomial approximation of the solution and a very fine
mesh were required to obtain a stable solution. With the developed
formulation, we achieved robust and accurate solutions using high-order
polynomials and a relatively coarse mesh. Increasing the polynomial
order resulted in sharper predictions of the rich flow topology. Mass,
total energy, and atomic elements were shown to be conserved. An important
finding is that the entropy limiter was crucial for the coarser two-dimensional
detonation problems; without it (i.e., only the positivity-preserving
limiter was applied), the nonlinear solver during the reaction step
can slow down substantially or even stall, due to the large undershoots
in temperature. In the three-dimensional detonation test case, we
demonstrated that our methodology can compute accurate and robust
solutions to large-scale reacting-flow problems. Future work will
entail the simulation of larger-scale detonation applications involving
more complex geometries. 

\section*{Acknowledgments}

This work is sponsored by the Office of Naval Research through the
Naval Research Laboratory 6.1 Computational Physics Task Area. Discussions
with Dr. Brian Taylor are gratefully acknowledged.

\bibliographystyle{elsarticle-num}
\bibliography{citations}

\appendix

\section{Effect of over-integration on limiting frequency~\label{sec:Effect-of-over-integration}}

\RevisionTextThree{We investigate the effect of over-integation on the frequency of limiter
activation in a one-dimensional test case involving the advection
of a low-density Gaussian wave. This test case was originally introduced
by Trojak and Dzanic~\citep{Tro23} and computed in Part I~\citep{Chi22},
where optimal convergence was observed. For brevity, only $p=2$ is
considered. A periodic domain $\Omega=[-0.5,0.5]\:\mathrm{m}$ is
initialized in nondimensional form as
\begin{eqnarray}
v & = & 1,\nonumber \\
Y_{1} & = & \frac{1}{2}\left[\sin\left(2\pi x\right)+1\right],\nonumber \\
Y_{2} & = & 1-Y_{1},\label{eq:Gaussian-wave}\\
\rho & = & \exp\left(-\sigma x^{2}\right)+4\epsilon,\nonumber \\
P & = & 2\epsilon,\nonumber 
\end{eqnarray}
where $\sigma=500$ and $\epsilon=10^{-12}$. The thermodynamic relations
for the two fictitious species considered here are given by
\begin{align*}
\frac{W_{1}c_{p,1}\left(T\right)}{R^{0}} & =3.5,\quad\frac{W_{1}h_{1}\left(T\right)}{R^{0}}=3.5T_{r}\widehat{T},\quad\frac{W_{1}s_{1}^{o}}{R^{0}}=3.5\ln\widehat{T},\\
\frac{W_{2}c_{p,2}\left(T\right)}{R^{0}} & =2.491,\quad\frac{W_{2}h_{2}\left(T\right)}{R^{0}}=2.491T_{r}\widehat{T},\quad\frac{W_{2}s_{2}^{o}}{R^{0}}=2.491\ln\widehat{T}.
\end{align*}
Five element sizes are considered: $h$, $h/2$, $h/4$, $h/8$, and
$h/16$, where $h=0.04$. The solution is integrated in time using
the SSPRK3 time stepping scheme. The $L^{2}$ error after one advection
period is calculated in terms of the following normalized state variables:
\[
\widehat{\rho v}_{k}=\frac{1}{\sqrt{\rho_{r}P_{r}}}\rho v_{k},\quad\widehat{\rho e}_{t}=\frac{1}{P_{r}}\rho e_{t},\quad\widehat{C}_{i}=\frac{R^{0}T_{r}}{P_{r}}C_{i},
\]
where $\rho_{r}=1\,\mathrm{kg\cdot}\mathrm{m}^{-3}$, $P_{r}=101325\,\mathrm{Pa}$,
and $T_{r}=1000$ K.  Tables~\ref{tab:limiter-frequency-no-overintegration}
and~\ref{tab:limiter-frequency-overintegration} provide the percentage
of RK stages in which the positivity-preserving limiter and the entropy
limiter are activated without and with over-integration, respectively.
Note that the $L^{2}$ errors obtained with and without over-integration
are nearly identical (not shown for brevity). In general, the percentage
of limiting decreases with over-integration, apart from the activation
of the positivity-preserving limiter in the $h/4$ case. One likely
reason for this higher frequency of limiter activation is that the
number of points in $\mathcal{D}_{\kappa}$ increases in the case
of over-integration.}

\begin{table}
\begin{centering}
\caption{\protect\RevisionTextThree{Percentage of RK stages in which the positivity-preserving limiter
and entropy limiter are activated without over-integration. Here,
the positivity-preserving limiter corresponds to Steps 1 to 3 in Section~\ref{subsec:limiting-procedure},
and the entropy limiter corresponds to Step 4.\label{tab:limiter-frequency-no-overintegration}}}
\par\end{centering}
\centering{}%
\begin{tabular}{cccccc}
\hline 
 & $h$ & $h/2$ & $h/4$ & $h/8$ & $h/16$\tabularnewline
\hline 
Positivity-preserving limiter & 40.8\% & 51.9\% & 38.8\% & 0.014\% & 0.0034\%\tabularnewline
Entropy limiter & 18.3\% & 36.9\% & 0.61\% & 0\% & 0\%\tabularnewline
\hline 
\end{tabular}
\end{table}

\begin{table}
\begin{centering}
\caption{\protect\RevisionTextThree{Percentage of RK stages in which the positivity-preserving limiter
and entropy limiter are activated with over-integration. Here, the
positivity-preserving limiter corresponds to Steps 1 to 3 in Section~\ref{subsec:limiting-procedure},
and the entropy limiter corresponds to Step 4.\label{tab:limiter-frequency-overintegration}}}
\par\end{centering}
\centering{}%
\begin{tabular}{cccccc}
\hline 
 & $h$ & $h/2$ & $h/4$ & $h/8$ & $h/16$\tabularnewline
\hline 
Positivity-preserving limiter & 35.2\% & 51.0\% & 64.3\% & 0.012\% & 0.0026\%\tabularnewline
Entropy limiter & 13.8\% & 24.7\% & 0.22\% & 0\% & 0\%\tabularnewline
\hline 
\end{tabular}
\end{table}

\section{Supporting lemma~\label{sec:supporting-lemma}}

In the following, let $\Delta\mathcal{F}\in\mathbb{R}^{m}$ denote
the quantity
\[
\Delta\mathcal{F}=\left(\Delta\mathcal{F}_{\rho v},\Delta\mathcal{F}_{\rho e_{t}},\Delta\mathcal{F}_{C_{1}},\ldots,\Delta\mathcal{F}_{C_{n_{s}}}\right)^{T},
\]
where $\Delta\mathcal{F}_{\rho v}\in\mathbb{R}^{d}$ and the remaining
entries are in $\mathbb{R}$. The proof below is similar to that in~\citep[Lemma 6]{Zha17}.
\begin{lem}
\label{lem:alpha-constraints}Assume that $y=\left(\rho v,\rho e_{t},C_{1},\ldots,C_{n_{s}}\right)^{T}$
is in $\mathcal{G}$. Then $\check{y}=y-\alpha^{-1}\Delta\mathcal{F}$,
where $\alpha>0$, is also in $\mathcal{G}$ under the following conditions:

\begin{equation}
\alpha>\alpha^{*}\left(y,\Delta\mathcal{F}\right)=\max\left\{ \max_{i=1,\ldots,n_{s}}\frac{\mathcal{F}_{C_{i}}\cdot n}{C_{i}},\left.\alpha_{T}\right|_{\left(y,\Delta\mathcal{F}\right)},0\right\} ,\label{eq:alpha-constraint}
\end{equation}
where
\begin{equation}
\alpha_{T}=\begin{cases}
\frac{-\mathsf{b}+\sqrt{\mathsf{b}^{2}-4\rho^{2}u\mathsf{g}}}{2\rho^{2}u}, & \mathsf{b}^{2}-4\rho^{2}u\mathsf{g}\geq0\\
0, & \mathrm{otherwise}
\end{cases},\label{eq:alpha_T}
\end{equation}
$\mathsf{b}=-\rho e_{t}\mathsf{M}-\rho\Delta\mathcal{F}_{\rho e_{t}}+\rho v\cdot\Delta\mathcal{F}_{\rho v}+2\rho u_{0}\mathsf{M}$,
and $\mathsf{g}=\mathsf{M}\Delta\mathcal{F}_{\rho e_{t}}-\frac{1}{2}\left|\Delta\mathcal{F}_{\rho v}\right|^{2}-u_{0}\mathsf{M}^{2}$,
and $\mathsf{M}=\sum_{i=1}^{n_{s}}W_{i}\Delta\mathcal{F}_{C_{i}}$.
\end{lem}

\begin{proof}
$\check{y}=y-\alpha^{-1}\Delta\mathcal{F}$ can be expanded as
\[
\begin{split}\check{y}= & \left(\check{\rho v},\check{\rho e_{t}},\check{C_{1}},\ldots,\check{C}_{n_{s}}\right)^{T}\\
= & \left(\rho v-\alpha^{-1}\Delta\mathcal{F}_{\rho v},\rho e_{t}-\alpha^{-1}\Delta\mathcal{F}_{\rho e_{t}},C_{1}-\alpha^{-1}\Delta\mathcal{F}_{C_{1}},\ldots,C_{n_{s}}-\alpha^{-1}\Delta\mathcal{F}_{C_{n_{s}}}\right)^{T}.
\end{split}
\]
First, we focus on positivity of density and species concentrations.
For the $i$th species, $C_{i}-\alpha^{-1}\mathcal{F}_{C_{i}}\cdot n>0$
if $\alpha>\max\left\{ \left(\mathcal{F}_{C_{i}}\cdot n\right)/C_{i},0\right\} $.
Accounting for all species yields
\[
\alpha>\max\left\{ \max_{i=1,\ldots,n_{s}}\frac{\mathcal{F}_{C_{i}}\cdot n}{C_{i}},0\right\} .
\]
Density is then also positive. 

Next, we focus on positivity of temperature. For a given $y=\left(\rho v,\rho e_{t},C_{1},\ldots,C_{n_{s}}\right)^{T}$,
let $Z(y)$ be defined as
\begin{equation}
Z(y)=\rho^{2}u^{*}(y)=\rho(y)\rho e_{t}-\left|\rho v\right|^{2}/2-\rho^{2}u_{0}(y).\label{eq:Z-definition}
\end{equation}
Note that if $Z(y)>0$, then $T(y)>0$. $Z\left(\check{y}\right)$
can be expressed as
\begin{align*}
Z\left(y-\alpha^{-1}\Delta\mathcal{F}\right)= & \sum_{i=1}^{n_{s}}W_{i}\left(C_{i}-\alpha^{-1}\Delta\mathcal{F}_{C_{i}}\right)\left(\rho e_{t}-\alpha^{-1}\Delta\mathcal{F}_{\rho e_{t}}\right)\\
 & -\frac{1}{2}\left|\rho v-\alpha^{-1}\Delta\mathcal{F}_{\rho v}\right|^{2}-\left[\sum_{i=1}^{n_{s}}W_{i}\left(C_{i}-\alpha^{-1}\Delta\mathcal{F}_{C_{i}}\right)\right]^{2}u_{0},
\end{align*}
which, after multiplying both sides by $\alpha^{2}$, can be rewritten
as 
\begin{align}
\alpha^{2}Z\left(y-\alpha^{-1}\Delta\mathcal{F}\right)= & \rho^{2}u\alpha^{2}-\mathsf{b}\alpha+\mathsf{g},\label{eq:alpha-quadratic-form}
\end{align}
where $\mathsf{b}=-\rho e_{t}\mathsf{M}-\rho\Delta\mathcal{F}_{\rho e_{t}}+\rho v\cdot\Delta\mathcal{F}_{\rho v}+2\rho u_{0}\mathsf{M}$,
$\mathsf{g}=\mathsf{M}\Delta\mathcal{F}_{\rho e_{t}}-\frac{1}{2}\left|\Delta\mathcal{F}_{\rho v}\right|^{2}-u_{0}\mathsf{M}^{2}$,
and $\mathsf{M}=\sum_{i=1}^{n_{s}}W_{i}\Delta\mathcal{F}_{C_{i}}$.
Setting the RHS of Equation~(\ref{eq:alpha-quadratic-form}) equal
to zero yields a quadratic equation with $\alpha$ as the unknown.
Since $\rho^{2}u$ is positive, the quadratic equation is convex.
As such, if $\mathsf{b}^{2}-4\rho^{2}u\mathsf{g}<0$, then no real
roots exist, and $Z\left(y-\alpha^{-1}\Delta\mathcal{F}\right)>0$
for all $\alpha\neq0$; otherwise, at least one real root exists,
in which case a sufficient condition to ensure $Z\left(y-\alpha^{-1}\Delta\mathcal{F}\right)>0$
is $\alpha>\alpha_{0}$, where $\alpha_{0}$ is given by
\[
\alpha_{0}=\max\left\{ \frac{-\mathsf{b}+\sqrt{\mathsf{b}^{2}-4\rho^{2}u\mathsf{g}}}{2\rho^{2}u},0\right\} .
\]
\end{proof}
\begin{rem}
\label{rem:simpler-alpha-constraint}If $\Delta\mathcal{F}_{C_{i}}=0,$
for all $i$, then $\alpha^{*}$ in (\ref{eq:alpha-constraint}) can
be simplified to 
\[
\alpha^{*}\left(y,\Delta\mathcal{F}\right)=\max\left\{ \left.\alpha_{T}\right|_{\left(y,\Delta\mathcal{F}\right)},0\right\} .
\]
\end{rem}

\section{Alternative approach with over-integration\label{sec:alternative-approach-with-over-integration}}

Here, we present an alternative approach to guarantee $\overline{y}_{\kappa}^{j+1}\in\mathcal{G}_{s_{b}}$,
even when over-integration with the modified flux interpolation~(\ref{eq:modified-flux-projection})
for preserving pressure equilibrium is employed. The main idea is
to construct a separate polynomial, denoted $\check{y}_{\kappa}$,
that satisfies the following:
\begin{align*}
\overline{\check{y}}_{\kappa} & =\overline{y}_{\kappa},\\
\check{y}_{\kappa}\left(\xi\left(\zeta_{l}^{\left(f\right)}\right)\right) & =\widetilde{y}_{\kappa}\left(\xi\left(\zeta_{l}^{\left(f\right)}\right)\right),\quad f=1,\ldots,n_{f},\quad l=1,\ldots,n_{q,f}^{\partial},
\end{align*}
such that the scheme satisfied by the element averages becomes
\begin{eqnarray}
\overline{y}_{\kappa}^{j+1} & = & \overline{y}_{\kappa}^{j}-\sum_{f=1}^{n_{f}}\sum_{l=1}^{n_{q,f}^{\partial}}\frac{\Delta t\nu_{f,l}^{\partial}}{|\kappa|}\mathcal{F}^{\dagger}\left(\widetilde{y}_{\kappa}^{j}\left(\xi\left(\zeta_{l}^{\left(f\right)}\right)\right),\widetilde{y}_{\kappa^{(f)}}^{j}\left(\xi\left(\zeta_{l}^{\left(f\right)}\right)\right),n\left(\zeta_{l}^{\left(f\right)}\right)\right)\nonumber \\
 & = & \overline{\check{y}}_{\kappa}^{j}-\sum_{f=1}^{n_{f}}\sum_{l=1}^{n_{q,f}^{\partial}}\frac{\Delta t\nu_{f,l}^{\partial}}{|\kappa|}\mathcal{F}^{\dagger}\left(\check{y}_{\kappa}^{j}\left(\xi\left(\zeta_{l}^{\left(f\right)}\right)\right),\check{y}_{\kappa^{(f)}}^{j}\left(\xi\left(\zeta_{l}^{\left(f\right)}\right)\right),n\left(\zeta_{l}^{\left(f\right)}\right)\right)\nonumber \\
 & = & \sum_{v=1}^{n_{q}}\theta_{v}\check{y}_{\kappa}^{j}\left(\xi_{v}\right)+\sum_{f=1}^{n_{f}}\sum_{l=1}^{n_{q,f}^{\partial}}\left[\theta_{f,l}\check{y}_{\kappa}^{j}\left(\xi\left(\zeta_{l}^{\left(f\right)}\right)\right)-\frac{\Delta t\nu_{f,l}^{\partial}}{|\kappa|}\mathcal{F}^{\dagger}\left(\check{y}_{\kappa}^{j}\left(\xi\left(\zeta_{l}^{\left(f\right)}\right)\right),\check{y}_{\kappa^{(f)}}\left(\xi\left(\zeta_{l}^{\left(f\right)}\right)\right),n\left(\zeta_{l}^{\left(f\right)}\right)\right)\right]\nonumber \\
 & = & \sum_{v=1}^{n_{q}}\theta_{v}\check{y}_{\kappa}^{j}\left(\xi_{v}\right)+\sum_{f=1}^{n_{f}-1}\sum_{l=1}^{n_{q,f}^{\partial}}\theta_{f,l}\check{A}_{f,l}+\sum_{l=1}^{N-1}\theta_{n_{f},l}\check{B}_{l}+\theta_{n_{f},N}\check{C},\label{eq:fully-discrete-form-average-2D-modified-v2}
\end{eqnarray}
where the definitions of $\check{A}_{f,l}$, $\check{B}_{l}$, and
$\check{C}$ can be deduced based on Section~\ref{subsec:entropy-bounded-high-order-DG-2D}.
Thus, $\overline{y}_{\kappa}^{j+1}$ can be written as a convex combination
of three-point systems and pointwise values, and an analogous version
of Theorem~\ref{thm:CFL-condition-2D} holds. The degree of $\check{y}_{\kappa}$
may be different from that of $y_{\kappa}$, provided that the volume
quadrature rule in Equation~(\ref{eq:fully-discrete-form-average-2D-modified-v2})
is sufficiently accurate. First, we propose a strategy to construct
$\check{y}_{\kappa}$ on two-dimensional quadrilateral elements. As
an illustrative example, consider the Gauss-Lobatto nodal set for
$\check{y}_{\kappa}$ displayed in Figure~\ref{fig:quadrilateral_nodal_set}.
Nodes 1 to 8 comprise the integration points used in the surface integrals
in Equation~(\ref{eq:fully-discrete-form-average-2D-modified-v2}),
whereas Node 9 serves as a degree of freedom to ensure $\overline{\check{y}}_{\kappa}=\overline{y}_{\kappa}$.
The coefficients of $\check{y}_{\kappa}$ are specified as
\begin{align*}
\check{y}_{\kappa}\left(x_{k}\right) & =\widetilde{y}_{\kappa}\left(x_{k}\right),k=1,\ldots,8,\\
\check{y}_{\kappa}\left(x_{9}\right) & =\frac{\overline{y}_{\kappa}|\kappa|-\sum_{v=1}^{n_{q}}\left|J_{\kappa}(\xi_{v})\right|w_{v}\sum_{k=1}^{8}\widetilde{y}_{\kappa}\left(x_{k}\right)\varphi_{k}\left(\xi_{v}\right)}{\sum_{v=1}^{n_{q}}\left|J_{\kappa}(\xi_{v})\right|w_{v}\varphi_{9}\left(\xi_{v}\right)}\\
 & =\widetilde{y}_{\kappa}\left(x_{9}\right)+\frac{\sum_{v=1}^{n_{q}}\left|J_{\kappa}(\xi_{v})\right|w_{v}\left[\sum_{i=1}^{n_{b}}y_{\kappa}(x_{i})\phi_{i}(\xi_{v})-\sum_{k=1}^{9}\widetilde{y}_{\kappa}\left(x_{k}\right)\varphi_{k}\left(\xi_{v}\right)\right]}{\sum_{v=1}^{n_{q}}\left|J_{\kappa}(\xi_{v})\right|w_{v}\varphi_{9}\left(\xi_{v}\right)}.
\end{align*}
In general, $\check{y}_{\kappa}\left(x_{9}\right)$ is not expected
to differ significantly from $\widetilde{y}_{\kappa}\left(x_{9}\right)$
or $y_{\kappa}(x_{9})$; in fact, $\check{y}_{\kappa}\left(x_{9}\right)$
reduces to $\widetilde{y}_{\kappa}\left(x_{9}\right)$ as the discrepancies
between $y_{\kappa}(x)$ and $\widetilde{y}_{\kappa}(x)$ vanish.
If necessary, the limiting procedure in Section~\ref{subsec:limiting-procedure}
is applied to ensure that $\check{y}_{\kappa}^{j}(x)\in\mathcal{G}_{s_{b}},\;\forall x\in\mathcal{D_{\kappa}}$.
Note that it is possible for the limiter to modify the pressure at
the nodes; nevertheless, as observed in the thermal-bubble test case
in Part I~\citep{Chi22}, the limiter typically will not destroy
pressure equilibrium or cause large-scale pressure oscillations in
smooth regions of the flow. This will likely remain true when limiting
$\check{y}_{\kappa}^{j}(x)$, in part because $y_{\kappa}^{j}(x)$
will already have been limited.

\begin{figure}[tbph]
\begin{centering}
\includegraphics[width=0.4\columnwidth]{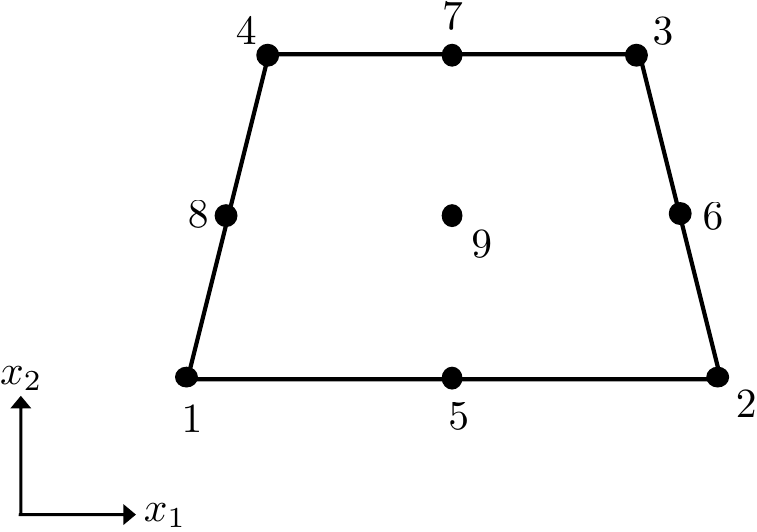}
\par\end{centering}
\caption{\label{fig:quadrilateral_nodal_set}Illustrative nodal set for $\check{y}_{\kappa}$
on two-dimensional quadrilateral elements.}
\end{figure}

Constructing $\check{y}_{\kappa}$ on two-dimensional triangular elements
can be done in a similar manner, provided that the nodal set for $\check{y}_{\kappa}$
includes the surface integration points and the corresponding surface
quadrature rules have positive weights. However, typical $p=2$ nodal
sets, which have six nodes, do not include any interior nodes; therefore,
there are no degrees of freedom to ensure $\overline{\check{y}}_{\kappa}=\overline{y}_{\kappa}$.
Here, we propose a more general strategy that makes use of the transformation
between the reference quadrilateral, $\widehat{\kappa}_{\mathrm{quad}}$,
which is a bi-unit square, and the reference triangle, $\widehat{\kappa}_{\mathrm{tri}}$,
which is an isosceles right triangle with side length of two, as shown
in Figure~\ref{fig:quadrilateral_triangle_transformation}~\citep{Dub91,Kir06}.
The bi-unit square can be mapped to the reference triangle as
\begin{align*}
\xi_{1} & =\frac{(1+\eta_{1})(1-\eta_{2})}{2}-1,\\
\xi_{2} & =\eta_{2},
\end{align*}
where $\eta\in\mathbb{R}^{2}$ are the reference coordinates of $\widehat{\kappa}_{\mathrm{quad}}$
and $\xi\in\mathbb{R}^{2}$ are the reference coordinates of $\widehat{\kappa}_{\mathrm{tri}}$.
The inverse mapping is given by
\begin{align*}
\eta_{1} & =2\frac{1+\xi_{1}}{1-\xi_{2}}-1,\\
\eta_{2} & =\xi_{2}.
\end{align*}
Consider the seven-node triangle in Figure~\ref{fig:triangle_nodal_set},
obtained by degeneration of the Gauss-Lobatto nine-node quadrilateral~\citep{Hug80}.
Specifically, Nodes 3, 4, and 7 of $\widehat{\kappa}_{\mathrm{quad}}$
are coalesced into Node 3 of $\widehat{\kappa}_{\mathrm{tri}}$, such
that
\begin{align*}
\check{y}_{\kappa}\left(\eta\right) & =\underset{k\neq3,4,7}{\sum_{k=1}^{9}}\check{y}_{\kappa}\left(x_{k}\right)\varphi_{k}\left(\eta\right)+\check{y}_{\kappa}\left(x_{3}\right)\left[\varphi_{3}\left(\eta\right)+\varphi_{4}\left(\eta\right)+\varphi_{7}\left(\eta\right)\right],\\
 & =\sum_{k=1}^{7}\check{y}_{\kappa}\left(x_{k}\right)\check{\varphi}_{k}\left(\eta\right),
\end{align*}
where $\left\{ \varphi_{1},\ldots,\varphi_{9}\right\} $ is ordered
according to the $\widehat{\kappa}_{\mathrm{quad}}$ node numbering
and $\left\{ \check{\varphi}_{1},\ldots,\check{\varphi}_{7}\right\} $
is ordered according to the $\widehat{\kappa}_{\mathrm{tri}}$ node
numbering. Note that $\check{\varphi}_{7}$ is equal to $\varphi_{3}+\varphi_{4}+\varphi_{7}$.
Nodes 1 to 6 of $\widehat{\kappa}_{\mathrm{tri}}$ make up the Gauss-Lobatto
points used in the surface integrals. $\overline{\check{y}}_{\kappa}$
can be computed using a sufficiently accurate quadrature rule for
quadrilaterals as~\citep{Kir06}
\begin{align*}
\overline{\check{y}}_{\kappa} & =\frac{1}{|\kappa|}\int_{\kappa}\check{y}_{\kappa}(x)dx=\frac{1}{|\kappa|}\int_{\widehat{\kappa}_{\mathrm{tri}}}\check{y}_{\kappa}(\xi)\left|J_{\kappa}(\xi)\right|d\xi\\
 & =\frac{1}{|\kappa|}\int_{\widehat{\kappa}_{\mathrm{quad}}}\check{y}_{\kappa}(\eta)\left|J_{\kappa}(\xi(\eta))\right|\left|J_{\widehat{\kappa}}(\eta)\right|d\eta\\
 & =\frac{1}{|\kappa|}\sum_{v=1}^{n_{q}}\check{y}_{\kappa}(\eta_{v})\left|J_{\kappa}(\xi(\eta_{v}))\right|\left|J_{\widehat{\kappa}}(\eta_{v})\right|w_{v},
\end{align*}
where $\left|J_{\widehat{\kappa}}(\eta)\right|$ is the Jacobian determinant
of the mapping from the reference quadrilateral to the reference triangle,
given by
\[
\left|J_{\widehat{\kappa}}(\eta)\right|=\frac{1-\eta_{2}}{2},
\]
which is positive everywhere in $\widehat{\kappa}_{\mathrm{tri}}$
except along the collapsed face, such that the analysis in Section~(\ref{subsec:entropy-bounded-high-order-DG-2D})
holds.  Using the $\widehat{\kappa}_{\mathrm{tri}}$ node numbering,
the coefficients of $\check{y}_{\kappa}$ are then prescribed as
\begin{align*}
\check{y}_{\kappa}\left(x_{k}\right) & =\widetilde{y}_{\kappa}\left(x_{k}\right),k=1,\ldots,6\\
\check{y}_{\kappa}\left(x_{7}\right) & =\frac{\overline{y}_{\kappa}|\kappa|-\sum_{v=1}^{n_{q}}\left|J_{\kappa}(\xi(\eta_{v}))\right|\left|J_{\widehat{\kappa}}(\eta_{v})\right|w_{v}\sum_{k=1}^{6}\widetilde{y}_{\kappa}\left(x_{k}\right)\check{\varphi}_{k}\left(\eta_{v}\right)}{\sum_{v=1}^{n_{q}}\left|J_{\kappa}(\xi(\eta_{v}))\right|\left|J_{\widehat{\kappa}}(\eta_{v})\right|w_{v}\check{\varphi}_{7}\left(\eta_{v}\right)}\\
 & =\widetilde{y}_{\kappa}\left(x_{7}\right)+\frac{\sum_{v=1}^{n_{q}}\left|J_{\kappa}(\xi(\eta_{v}))\right|\left|J_{\widehat{\kappa}}(\eta_{v})\right|w_{v}\left[\sum_{i=1}^{n_{b}}y_{\kappa}(x_{i})\phi_{i}(\xi(\eta_{v}))-\sum_{k=1}^{7}\widetilde{y}_{\kappa}\left(x_{k}\right)\check{\varphi}_{k}\left(\eta_{v}\right)\right]}{\sum_{v=1}^{n_{q}}\left|J_{\kappa}(\xi(\eta_{v}))\right|\left|J_{\widehat{\kappa}}(\eta_{v})\right|w_{v}\check{\varphi}_{7}\left(\eta_{v}\right)}.
\end{align*}
This approach can be generalized to other orders and shapes as well.

\begin{center}
\begin{figure}[tbph]
\begin{centering}
\subfloat[\label{fig:quadrilateral_triangle_transformation}Mapping from reference
quadrilateral, $\widehat{\kappa}_{\mathrm{quad}}$, to reference triangle,
$\widehat{\kappa}_{\mathrm{tri}}$.]{\includegraphics[width=0.7\columnwidth]{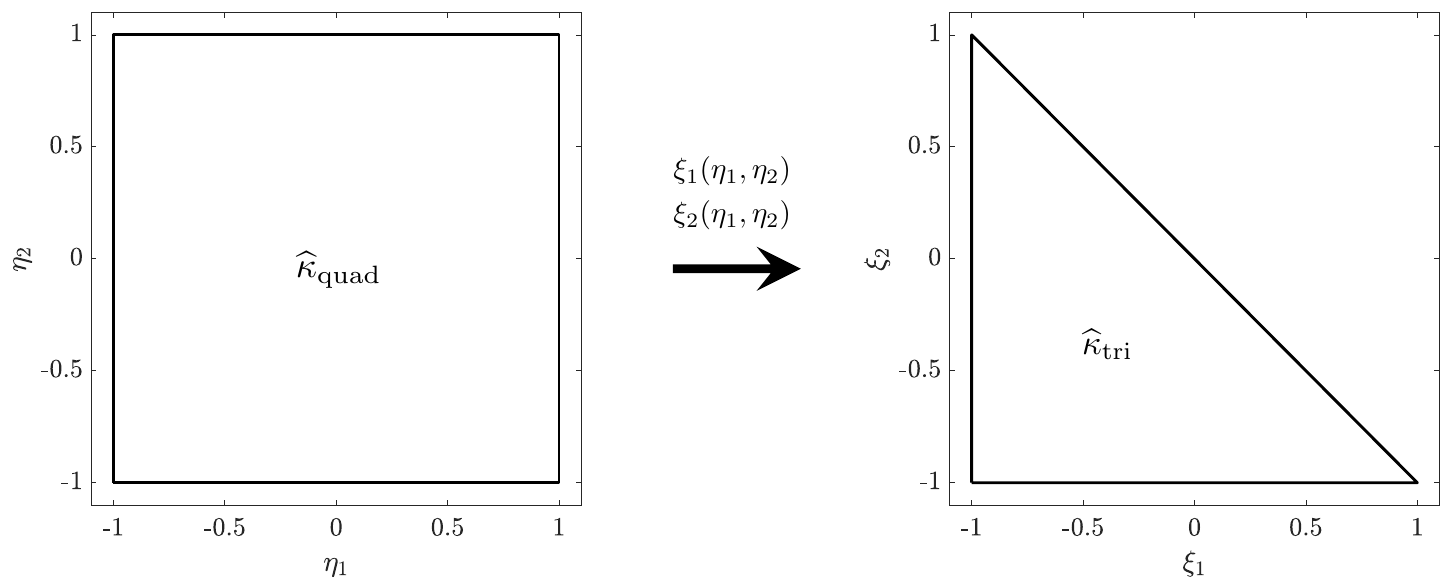}}
\par\end{centering}
\begin{centering}
\subfloat[\label{fig:triangle_nodal_set}Seven-node triangle obtained via degeneration
of the Gauss-Lobatto nine-node quadrilateral.]{\includegraphics[width=0.7\columnwidth]{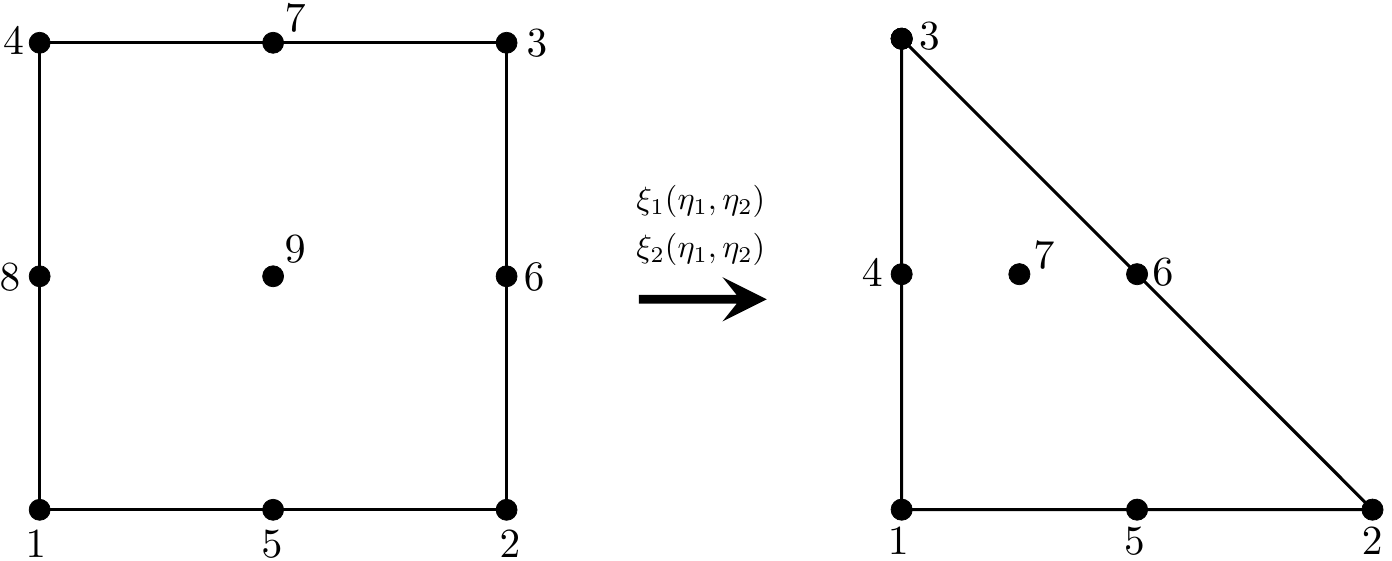}}
\par\end{centering}
\centering{}\caption{Transformation between reference quadrilateral and reference triangle,
as well as an illustrative nodal set for $\check{y}_{\kappa}$ on
triangular elements.}
\end{figure}
\par\end{center}

\section{Chemical mechanism\label{sec:Chemical-mechanism-detonation}}

\RevisionTextThree{Here, we provide information on the thermodynamic fits and reaction-rate
coefficients for the chemical mechanism used in Sections~\ref{subsec:2D-detonation-wave}
and~\ref{subsec:3D-detonation-wave}. The thermodynamic relations
for the $i$th species are given by 
\begin{align*}
\frac{W_{i}c_{p,i}\left(T\right)}{R^{0}} & =a_{i0}-a_{i1}\widehat{T}+a_{i2}\widehat{T}^{2}-a_{i3}\widehat{T}^{3}+a_{i4}\widehat{T}^{4},\\
\frac{W_{i}h_{i}\left(T\right)}{R^{0}} & =\int\frac{W_{i}c_{p,i}\left(\tau\right)d\tau}{R^{0}}-a_{i5}T_{r},\quad\frac{W_{i}s_{i}^{o}}{R^{0}}=\int\frac{W_{i}c_{p,i}\left(\tau\right)d\tau}{R^{0}\tau}-a_{i6},
\end{align*}
where $T_{r}=1000$ K, $\widehat{T}=T/T_{r}$, and $\tau$ is a dummy
variable for $T$. The coefficients are given in Table~\ref{tab:hydrogen-thermo}.}

\begin{table}
\begin{centering}
\caption{\protect\RevisionTextThree{Coefficients for the thermodynamic fits for each species.\label{tab:hydrogen-thermo}}}
\par\end{centering}
\centering{}%
\begin{tabular}{cccccccc}
\hline 
Species & $a_{i0}$ & $a_{i1}$ & $a_{i2}$ & $a_{i3}$ & $a_{i4}$ & $a_{i5}$ & $a_{i6}$\tabularnewline
\hline 
O & 2.72 & -3.61E-01 & 2.01E-01 & -4.53E-02 & 3.79E-03 & 29.17 & 3.98\tabularnewline
O\textsubscript{2} & 3.17 & 1.47 & -5.96E-01 & 1.25E-01 & -9.78E-03 & -1.01 & 6.19\tabularnewline
H & 2.50 & 0 & 0 & 0 & 0 & 25.47 & -0.45\tabularnewline
H\textsubscript{2} & 3.38 & 9.77E-02 & 2.56E-01 & -7.59E-02 & 6.63E-03 & -1.02 & -3.6\tabularnewline
OH & 3.58 & -2.13E-01 & 4.84E-01 & -1.44E-01 & 1.30E-02 & 3.42 & 1.74\tabularnewline
HO\textsubscript{2} & 3.41 & 3.22 & -1.12 & 2.03E-01 & -1.50E-02 & 0.33 & 7.24\tabularnewline
H\textsubscript{2}O & 3.57 & 1.37 & 1.71E-01 & -1.31E-01 & 1.46E-02 & -30.21 & 1.94\tabularnewline
H\textsubscript{2}O\textsubscript{2} & 3.70 & 5.62 & -2.22 & 4.17E-01 & -2.96E-02 & -17.68 & 5.57\tabularnewline
N\textsubscript{2} & 3.23 & 8.55E-01 & -1.51E-01 & -6.39E-03 & 2.68E-03 & -1. & 4.42\tabularnewline
Ar & 2.50 & 0 & 0 & 0 & 0 & -0.75 & 4.37\tabularnewline
\end{tabular}
\end{table}

\RevisionTextThree{The production rate of the $i$th species in Equation~(\ref{eq:reacting-navier-stokes-source-term})
is computed as
\[
\omega_{i}=\sum_{j=1}^{n_{r}}\nu_{ij}q_{j},
\]
where $n_{r}$ is the number of reactions, $\nu_{ij}=\nu_{ij}^{r}-\nu_{ij}^{f}$
is the difference between the reverse ($\nu_{ij}^{r}$) and the forward
stoichiometric coefficients ($\nu_{ij}^{f}$), and $q_{j}$ is the
rate of progress of the $j$th reaction. The chemical mechanism here
contains only irreversible reactions, such that $q_{j}$ is computed
as
\begin{equation}
q_{j}=k_{j}^{f}\prod_{i=1}^{n_{s}}C_{i}^{\nu_{ij}^{f}},\label{eq:chemical-reaction-rate-of-progress}
\end{equation}
where $k_{j}^{f}$ is the forward rate constant of the $j$th reaction.
The units of $q_{j}$ are assumed to be kmol m\textsuperscript{-3}
s\textsuperscript{-1}. The forward rate constants are computed using
the Arrhenius form
\begin{equation}
k_{j}^{f}=A_{j}T^{b_{j}}\exp\left(-\frac{E_{j}}{R^{0}T}\right),\label{eq:arrhenius}
\end{equation}
where $A_{j}>0$ and $b_{j}$ are parameters and $E_{j}$ is the activation
energy~\citep{Gio99,Kee96}. Table~\ref{tab:hydrogen-arrhenius}
lists the elementary chemical reactions included in the chemical mechanism,
for which the production rates are computed as above. }
\begin{table}
\begin{centering}
\caption{\protect\RevisionTextThree{List of elementary reactions for which the forward rate constant is
of the form~(\ref{eq:arrhenius}). The units of $A_{j}$ are a combination
of kmol, m\protect\textsuperscript{3}, s, and K that depend on the
reaction. \label{tab:hydrogen-arrhenius}}}
\par\end{centering}
\centering{}%
\begin{tabular}{ccccc}
\hline 
$j$, reaction \# & Equation & $A_{j}$ & $b_{j}$ & $E_{j}/R^{0}$ (K)\tabularnewline
\hline 
1 & $\mathrm{H}+\mathrm{O}_{2}\rightarrow\mathrm{O}+\mathrm{OH}$ & 1.86E11 & 0 & 8449\tabularnewline
2 & $\mathrm{O}+\mathrm{O}\mathrm{H}\rightarrow\mathrm{H}+\mathrm{O}_{2}$ & 1.48E10 & 0 & 342\tabularnewline
3 & $\mathrm{H}_{2}+\mathrm{O}\rightarrow\mathrm{H}+\mathrm{O}\mathrm{H}$ & 1.82E7 & 1 & 4479\tabularnewline
4 & $\mathrm{H}+\mathrm{O}\mathrm{H}\rightarrow\mathrm{H}_{2}+\mathrm{O}$ & 8.32E6 & 1 & 3497\tabularnewline
5 & $\mathrm{H}_{2}\mathrm{O}+\mathrm{O}\rightarrow\mathrm{O}\mathrm{H}+\mathrm{O}\mathrm{H}$ & 3.39E10 & 0 & 9234\tabularnewline
6 & $\mathrm{O}\mathrm{H}+\mathrm{O}\mathrm{H}\rightarrow\mathrm{H}_{2}\mathrm{O}+\mathrm{O}$ & 3.16E9 & 0 & 554\tabularnewline
7 & $\mathrm{H}_{2}\mathrm{O}+\mathrm{H}\rightarrow\mathrm{H}_{2}+\mathrm{O}\mathrm{H}$ & 9.55E10 & 0 & 10215\tabularnewline
8 & $\mathrm{H}_{2}+\mathrm{O}\mathrm{H}\rightarrow\mathrm{H}_{2}\mathrm{O}+\mathrm{H}$ & 2.19E10 & 0 & 2592\tabularnewline
9 & $\mathrm{H}_{2}\mathrm{O}_{2}+\mathrm{O}\mathrm{H}\rightarrow\mathrm{H}_{2}\mathrm{O}+\mathrm{H}\mathrm{O}_{2}$ & 1.00E10 & 0 & 906\tabularnewline
10 & $\mathrm{H}_{2}\mathrm{O}+\mathrm{H}\mathrm{O}_{2}\rightarrow\mathrm{H}_{2}\mathrm{O}_{2}+\mathrm{O}\mathrm{H}$ & 2.82E10 & 0 & 16501\tabularnewline
11 & $\mathrm{H}\mathrm{O}_{2}+\mathrm{O}\rightarrow\mathrm{O}\mathrm{H}+\mathrm{O}_{2}$ & 5.01E10 & 0 & 503\tabularnewline
12 & $\mathrm{O}\mathrm{H}+\mathrm{O}_{2}\rightarrow\mathrm{H}\mathrm{O}_{2}+\mathrm{O}$ & 6.46E10 & 0 & 28261\tabularnewline
13 & $\mathrm{H}\mathrm{O}_{2}+\mathrm{H}\rightarrow\mathrm{O}\mathrm{H}+\mathrm{O}\mathrm{H}$ & 2.51E11 & 0 & 956\tabularnewline
14 & $\mathrm{O}\mathrm{H}+\mathrm{O}\mathrm{H}\rightarrow\mathrm{H}\mathrm{O}_{2}+\mathrm{H}$ & 1.20E10  & 0 & 20179\tabularnewline
15 & $\mathrm{H}\mathrm{O}_{2}+\mathrm{H}\rightarrow\mathrm{H}_{2}+\mathrm{O}_{2}$ & 2.51E10  & 0 & 352\tabularnewline
16 & $\mathrm{H}_{2}+\mathrm{O}_{2}\rightarrow\mathrm{H}\mathrm{O}_{2}+\mathrm{H}$ & 5.50E10 & 0 & 29086\tabularnewline
17 & $\mathrm{H}\mathrm{O}_{2}+\mathrm{O}\mathrm{H}\rightarrow\mathrm{H}_{2}\mathrm{O}+\mathrm{O}_{2}$ & 5.01E10  & 0 & 503\tabularnewline
18 & $\mathrm{H}_{2}\mathrm{O}+\mathrm{O}_{2}\rightarrow\mathrm{H}\mathrm{O}_{2}+\mathrm{O}\mathrm{H}$ & 6.31E11 & 0 & 37168\tabularnewline
19 & $\mathrm{H}_{2}\mathrm{O}_{2}+\mathrm{O}_{2}\rightarrow\mathrm{H}\mathrm{O}_{2}+\mathrm{H}\mathrm{O}_{2}$ & 3.98E10 & 0 & 21457\tabularnewline
20 & $\mathrm{H}\mathrm{O}_{2}+\mathrm{H}\mathrm{O}_{2}\rightarrow\mathrm{H}_{2}\mathrm{O}_{2}+\mathrm{O}_{2}$ & 1.00E10 & 0 & 503\tabularnewline
21 & $\mathrm{H}_{2}\mathrm{O}_{2}+\mathrm{H}\rightarrow\mathrm{H}\mathrm{O}_{2}+\mathrm{H}_{2}$ & 1.70E9  & 0 & 1887\tabularnewline
22 & $\mathrm{H}\mathrm{O}_{2}+\mathrm{H}_{2}\rightarrow\mathrm{H}_{2}\mathrm{O}_{2}+\mathrm{H}$ & 7.24E8 & 0 & 9410\tabularnewline
\end{tabular}
\end{table}

\RevisionTextThree{Table~\ref{tab:hydrogen-three-body} lists the three-body reactions,
for which the rate of progress is scaled by a prefactor as~\citep{Kee96}
\begin{equation}
q_{j}=\left(\sum_{i=1}^{n_{s}}\alpha_{ij}C_{i}\right)\left(k_{j}^{f}\prod_{i=1}^{n_{s}}C_{i}^{\nu_{ij}^{f}}\right),\label{eq:three-body}
\end{equation}
where $\alpha_{ij}$ are the third-body efficiencies. Unless otherwise
specified, the third-body efficiencies are taken to be unity. Only
$\alpha_{\mathrm{H_{2}O},j}$ and $\alpha_{\mathrm{O_{2}},j}$ assume
non-unity values (although they are also taken to be unity in Reaction
32).}
\begin{table}
\begin{centering}
\caption{\protect\RevisionTextThree{List of three-body reactions for which the rate of progress is of
the form~(\ref{eq:three-body}). The units of $A_{j}$ are a combination
of kmol, m\protect\textsuperscript{3}, s, and K that depend on the
reaction. The third-body efficiencies are assumed to be unity, apart
from $\alpha_{\mathrm{H_{2}O},j}$ and $\alpha_{\mathrm{O_{2}},j}$.
\label{tab:hydrogen-three-body}}}
\par\end{centering}
\centering{}%
\begin{tabular}{ccccccc}
\hline 
$j$, reaction \# & Equation & $A_{j}$ & $b_{j}$ & $E_{j}/R^{0}$ (K) & $\alpha_{\mathrm{H_{2}O},j}$  & $\alpha_{\mathrm{O_{2}},j}$\tabularnewline
\hline 
23 & $\mathrm{H}_{2}\mathrm{O}+\mathrm{M}\rightarrow\mathrm{H}+\mathrm{O}\mathrm{H}+\mathrm{M}$ & 2.19E13 & 0 & 52838 & 6.5 & 0.4\tabularnewline
24 & $\mathrm{H}+\mathrm{O}\mathrm{H}+\mathrm{M}\rightarrow\mathrm{H}_{2}\mathrm{O}+\mathrm{M}$ & 1.41E17 & -2 & 0 & 6.5 & 0.4\tabularnewline
25 & $\mathrm{H}+\mathrm{O}_{2}+\mathrm{M}\rightarrow\mathrm{H}\mathrm{O}_{2}+\mathrm{M}$ & 1.66E9 & 0 & -503 & 6.5 & 0.4\tabularnewline
26 & $\mathrm{H}\mathrm{O}_{2}+\mathrm{M}\rightarrow\mathrm{H}+\mathrm{O}_{2}+\mathrm{M}$ & 2.29E12 & 0 & 23098 & 6.5 & 0.4\tabularnewline
27 & $\mathrm{H}_{2}\mathrm{O}_{2}+\mathrm{M}\rightarrow\mathrm{O}\mathrm{H}+\mathrm{O}\mathrm{H}+\mathrm{M}$ & 1.20E14 & 0 & 22896 & 6.5 & 0.4\tabularnewline
28 & $\mathrm{O}\mathrm{H}+\mathrm{O}\mathrm{H}+\mathrm{M}\rightarrow\mathrm{H}_{2}\mathrm{O}_{2}+\mathrm{M}$ & 9.12E8 & 0 & -2551 & 6.5 & 0.4\tabularnewline
29 & $\mathrm{O}+\mathrm{H}+\mathrm{M}\rightarrow\mathrm{O}\mathrm{H}+\mathrm{M}$ & 1.00E10 & 0 & 0 & 6.5 & 0.4\tabularnewline
30 & $\mathrm{O}\mathrm{H}+\mathrm{M}\rightarrow\mathrm{O}+\mathrm{H}+\mathrm{M}$ & 7.94E16 & -1 & 52194 & 6.5 & 0.4\tabularnewline
31 & $\mathrm{O}_{2}+\mathrm{M}\rightarrow\mathrm{O}+\mathrm{O}+\mathrm{M}$ & 5.13E12 & 0 & 57870 & 6.5 & 0.4\tabularnewline
32 & $\mathrm{O}+\mathrm{O}+\mathrm{M}\rightarrow\mathrm{O}_{2}+\mathrm{M}$ & 4.68E9 & -0.28 & 0 & 1 & 1\tabularnewline
33 & $\mathrm{H}_{2}+\mathrm{M}\rightarrow\mathrm{H}+\mathrm{H}+\mathrm{M}$ & 2.19E11 & 0 & 48309 & 6.5 & 0.4\tabularnewline
34 & $\mathrm{H}+\mathrm{H}+\mathrm{M}\rightarrow\mathrm{H}_{2}+\mathrm{M}$ & 3.02E9 & 0 & 0 & 6.5 & 0.4\tabularnewline
\end{tabular}
\end{table}

\end{document}